\documentclass[12pt]{article}
\usepackage{graphicx}
 \usepackage{mathptmx}      
\usepackage{amsmath}
\usepackage{amssymb}
\usepackage{amsfonts}

\usepackage[usenames,dvipsnames]{color}
\usepackage{amsthm}
\newtheorem{prop}{Proposition}[section]{\bfseries}{\itshape}
\newtheorem{theo}[prop]{Theorem}{\bfseries}{\itshape}
\newtheorem{coro}[prop]{Corollary}{\bfseries}{\itshape}
\newtheorem{lemm}[prop]{Lemma}{\bfseries}{\itshape}

\newtheorem{remk}[prop]{Remark}{\bfseries}{\itshape}

\newcommand{\eqco}{\setcounter{equation}{0}}
\newcommand{\allco}{\eqco}
\renewcommand{\theequation}{\thesection.\arabic{equation}}
\renewcommand{\thefootnote}{}

\newcommand{\lbl}{\label}

\newcommand{\Po}{{\cal P}}
\newcommand{\tPo}{\tilde{{\cal P}}}
\newcommand{\cQ}{{\cal Q}}
\newcommand{\DG}{{\Delta}}
\newcommand{\cR}{{\cal R}}
\newcommand{\Q}{{\end{document}}}
\newcommand{\cW}{{\cal W}}
\newcommand{\cM}{{\cal M}}
\newcommand{\tg}{\tilde{\gamma}}

\newcommand{\Bin}{{\rm Bin}}

\newcommand{\N}{\mathbb{N}}
\newcommand{\Z}{\mathbb{Z}}
\newcommand{\bW}{\mathbb{W}}
\newcommand{\bH}{\mathbb{H}}
\newcommand{\bF}{\mathbb{F}}
\newcommand{\R}{\mathbb{R}}
\newcommand{\E}{\mathbb{E}}
\newcommand{\X}{\mathcal{X}}
\renewcommand{\Pr}{\mathbb{P}}
\renewcommand{\emptyset}{\varnothing}

\newcommand{\bx}{{\bf x}}
\newcommand{\bfo}{{\bf o}}
\newcommand{\bfy}{{\bf y}}
\newcommand{\bz}{{\bf z}}
\newcommand{\bu}{{\bf u}}
\newcommand{\bv}{{\bf v}}
\newcommand{\bw}{{\bf w}}
\newcommand{\bp}{{\bf p}}
\newcommand{\bq}{{\bf q}}
\newcommand{\essinf}{{\rm ess~inf}}
\renewcommand{\E}{\mathbb E \,}
\newcommand{\C}{{\cal C}}
\newcommand{\bN}{{\bf N}}
\newcommand{\FF}{\mathbb{F}}
\newcommand{\tod}{\stackrel{{\cal D}}{\longrightarrow}}
\newcommand{\eqd}{\stackrel{{\cal D}}{=}}
\newcommand{\toP}{\stackrel{{P}}{\longrightarrow}}
\newcommand{\toas}{\stackrel{{{\rm a.s.}}}{\longrightarrow}}
\newcommand{\toL}{\stackrel{{L^2}}{\longrightarrow}}
\newcommand{\toone}{\stackrel{{L^1}}{\longrightarrow}}
\newcommand{\tolone}{\stackrel{{L^1}}{\longrightarrow}}
\newcommand{\tovia}{\rightsquigarrow}
\renewcommand{\leftrightarrow}{\leftrightsquigarrow}

\newcommand{\lb}{\left(}
\newcommand{\rb}{\right)}
\newcommand{\supp}{{\rm supp}}
\newcommand{\fmax}{f_{{\rm max}}}
\newcommand{\alphamax}{{\overline{\alpha}}}
\newcommand{\nalpha}{u}
\newcommand{\rhomax}{\rho_{{\rm max}}}
\newcommand{\phimax}{\phi_{{\rm max}}}
\newcommand{\rhomin}{\rho_{{\rm min}}}
\newcommand{\vold}{\pi_d}  
\newcommand{\hH}{\hat{H}}
\newcommand{\cK}{{\cal K}}

\newcommand{\LL}{{\cal L}}
\newcommand{\cH}{{\cal H}}
\renewcommand{\SS}{{\cal S}}
\newcommand{\A}{{\cal A}}
\newcommand{\Cov}{{\rm Cov}}
\newcommand{\essup}{{\rm ess~sup}}
\newcommand{\esinf}{{\rm ess~inf}}
\newcommand{\Var}{{\rm Var}}
\newcommand{\SD}{{\rm SD}}
\newcommand{\var}{{\rm Var}}
\newcommand{\const}{{\rm const.} \times}
\newcommand{\x}{{\bf x}}

\newcommand{\card}{{\rm card}}
\newcommand{\Vol}{{\rm Vol}}
\newcommand{\diam}{{\rm diam}}
\newcommand{\limn}{\lim_{n \to \infty} }
\newcommand{\limsupn}{\limsup_{n \to \infty} }
\newcommand{\limm}{\lim_{m \to \infty} }
\newcommand{\nn}{\nonumber \\}
\newcommand{\Y}{{\cal Y}}
\newcommand{\BB}{{\cal B}}
\newcommand{\cU}{{\cal U}}
\newcommand{\NN}{{\cal N}}
\newcommand{\K}{{\cal K}}
\newcommand{\cF}{{\cal F}}
\newcommand{\W}{{\bf W}}
\newcommand{\V}{{\bf V}}
\newcommand{\w}{{\bf w}}
\newcommand{\e}{{\bf e}}
\newcommand{\tiw}{\tilde{w}}
\newcommand{\tN}{{\tilde{N}}}
\newcommand{\tA}{{\tilde{A}}}
\renewcommand{\th}{{\tilde{h}}}
\newcommand{\tR}{{\tilde{R}}}
\newcommand{\tD}{{\tilde{D}}}
\newcommand{\tH}{H}
\newcommand{\tG}{{\tilde{G}}}
\newcommand{\tq}{{\tilde{\bf q}}}
\newcommand{\tnbq}{{\tilde{q}}}
\newcommand{\tf}{{\tilde{f}}}
\newcommand{\tQ}{C}
\newcommand{\tih}{{\tilde{h}}}
\newcommand{\tp}{{\tilde{\bf p}}}
\newcommand{\bvol}{{\overline{\Vol}}}
\newcommand{\surv}{\texttt{surv}}

\newcommand{\eps}{\varepsilon}
\newcommand{\edm}{\end{displaymath}}
\newcommand{\be}{\begin{equation}}
\newcommand{\ee}{\end{equation}}
\newcommand{\bea}{\begin{eqnarray}}
\newcommand{\eea}{\end{eqnarray}}
\newcommand{\bean}{\begin{eqnarray*}}
\newcommand{\eean}{\end{eqnarray*}}
\newcommand{\1}{{\bf 1}}
\renewcommand{\epsilon}{\varepsilon}
\newcommand{\bbS}{\mathbb{S}}
\newcommand{\dist}{\,{\rm dist}}
\newcommand{\lglg}{\log \log}
\newcommand{\blue}{black}

\begin{document}

\title{
Random Euclidean coverage from within
\thanks{Supported by EPSRC grant EP/T028653/1 }
}


\author{ Mathew D. Penrose \thanks{
%
%
  Department of
Mathematical Sciences, University of Bath, Bath BA2 7AY, United
Kingdom. 
             \texttt{m.d.penrose@bath.ac.uk}             
}
}


\maketitle

\begin{abstract}
	Let $X_1,X_2, \ldots $ be independent random uniform points in a bounded domain $A \subset \mathbb{R}^d$ with smooth boundary.  Define the {\em coverage threshold} $R_n$ to be the smallest $r$ such that $A$ is covered by the balls of radius $r$ centred on $X_1,\ldots,X_n$.  We obtain the limiting distribution of $R_n$ and also a strong law of large numbers for $R_n$ in the large-$n$ limit.  For example, 
	if $A$ has volume 1 and perimeter $|\partial A|$,
	if $d=3$ then  
	$\Pr[n\pi R_n^3 - \log n - 2 \log (\log n) \leq x]$ converges to $\exp(-2^{-4}\pi^{5/3} |\partial A| e^{-2 x/3})$ and $(n \pi R_n^3)/(\log n) \to 1$ almost surely, and if $d=2$ 
	then $\Pr[n \pi R_n^2 - \log n - \log (\log n) 
	\leq x]$ converges to $\exp(- e^{-x}- |\partial A|\pi^{-1/2}
	e^{-x/2})$.

	We give similar results for general $d$, and also for the case where $A$ is a polytope.  We also generalize to allow for multiple coverage.  The analysis relies on classical results by Hall and by Janson, along with a careful treatment of boundary effects. For the strong laws of large numbers, we can relax the requirement that the underlying density on $A$ be uniform.
\end{abstract}


%
%

\section{Introduction}
\label{SecIntro}

This paper is primarily concerned with the following 
{\em random coverage} problem. Given a 
specified compact region $A$ in a $d$-dimensional Euclidean space, 
what is the probability that $A$ is fully covered by a union
of Euclidean  balls of radius $r$ centred on $n$ points placed
independently uniformly at random in $A$, in the large-$n$ limit
with
 $r =r(n)$ becoming  small in an appropriate manner?

  This is a very natural type of question with
 a long history; see for example
\cite{BB,CSKM,Flatto,HallZW,HallBk,Janson,Moran}. Potential applications include
wireless communications
\cite{BB,Lan},
ballistics \cite{HallBk}, genomics
\cite{Roy}, statistics \cite{Cuevas},  immunology \cite{Moran},
stochastic optimization \cite{ZZ},
and topological data analysis \cite{BW,KTV}.

In an alternative version of this question,
one considers 
a smaller compact region $B \subset A^o$
($A^o$ denotes the interior of $A$), and asks whether
$B$ (rather than all of $A$) is covered.
This version is simpler because boundary effects are avoided. This 
alternative
version of our question was answered  independently
in the 1980s by Hall \cite{HallZW} and Janson 
\cite{Janson}.  Another way to avoid boundary effects would be
to consider coverage of a smooth manifold such as a sphere
(as in \cite{Moran}), and this
was also addressed in \cite{Janson}.

However, the original question does not appear to have been addressed
systematically until now
(at least, not when $d>1$).
Janson \cite[p. 108]{Janson} makes some remarks about the case where $A= [0,1]^d$
and one uses balls of the $\ell_\infty $ norm, but  does not consider
more general classes of $A$ or Euclidean balls.

This question seems well worth addressing. In many of the applications
areas,
 it is natural to consider the influence
 only of the random points placed within the region $A$ rather than
 also hypothetical  points placed outside $A$, for example
in the problem of statistical  set estimation which we shall
discuss below.

We shall express our results in terms of the
 {\em coverage threshold} $R_n$, which we
define to be  the the smallest radius of balls, centred on a
 set $\X_n$ of $n$ independent uniform random points in $A$,
 required to cover $A$. Note that $R_n$ is equal
to the Hausdorff distance between the sets $\X_n$
and $A$.  More generally, for $k \in \N$
 the {\em $k$-coverage threshold} $R_{n,k}$
is the smallest radius required to cover $A$ $k$ times. 
These thresholds are random variables, because the locations of
the centres are random.
We investigate their probabilistic behaviour 
as $n$ becomes large.

We shall determine the limiting behaviour of $\Pr[R_{n,k} \leq r_n]$
for any fixed $k$ and any 
sequence of numbers $(r_n)$, for the case where
$A$ is smoothly bounded (for general $d \geq 2$) or where
$A$ is a polytope 
(for $d=2 $ or $d= 3$).
We also obtain similar  results for a high-intensity {\em Poisson}
 sample in $A$, which may be more relevant in some applications,
as argued in \cite{HallZW}.

We also derive strong laws of large numbers
showing that
that $nR_{n,k}^d/\log n$ converges  almost surely to a finite positive limit,
and establishing the value of the limit. These
strong laws carry over
to  more general cases where $k$ may vary with $n$, and  the distribution
of points may be non-uniform. We give results of this type for
 $A$ smoothly
bounded 
(for general $d$), or for $A$ a polytope  (for $d \leq 3$),
or for $A$ a convex polytope  (for $d \geq 4$).

We emphasise that in all of these results, the limiting behaviour depends
on the geometry of $\partial A$, the topological boundary of $A$. 
\textcolor{\blue}{
	For example, we shall show} that when $d=3$ and the points are
uniformly distributed over a polyhedron, the limiting behaviour of $R_n$
is determined by the angle of the sharpest edge if this angle is
less than $\pi/2$. If this angle exceeds $\pi/2$
 then the location in $A$ furthest
from the sample $\X_n$ is asymptotically uniformly distributed
over $\partial A$, but if this angle is less than $\pi/2$ the location in 
$A$ furthest from $\X_n$ is asymptotically  uniformly distributed
over the union of those edges  which are sharpest, \textcolor{\blue}{i.e. those
edges which achieve the minimum subtended angle.}

We restrict attention here to coverage by Euclidean balls of equal radius.
The work of \cite{HallZW,Janson} allowed for generalizations such as
other shapes or variable radii, in their versions of our problem.
 We do not attempt to address these generalizations here; in principle
it may be possible, but it would add considerably to the complication
of the formulation of results.

One application lies in statistical {\em set estimation}.  
 One may wish to estimate the set $A$ from the sample $\X_n$.
 One possible estimator in the literature
is the union of balls of radius $r_n$ centred on the points of $\X_n$,
for some suitable
sequence $(r_n)_{n \geq 1}$ decreasing to zero.
 In particular, when estimating the perimeter of $A $ one may well wish
to take $r_n$  large enough so that these balls fully cover $A$, that is,
$r_n \geq R_n$. For further discussion see Cuevas and 
Rodr\'iguez-Casal \cite{Cuevas}.

We briefly discuss some related concepts.
One of these is the {\em maximal spacing}
\textcolor{\blue}{
of the sample $\X_n$.
	This is defined to be volume of the
largest Euclidean ball that can be fitted into the set $A \setminus \X_n$
(the reason for the terminology becomes apparent from considering the case with
$d=1$).  
More generally, the
{\em maximal $k$-spacing} of the sample
is defined to be volume of the
largest Euclidean ball that can be fitted inside the set $A$
while containing fewer than $k$ points of $\X_n$.}

The maximal spacing also has a long history; see for example
\cite{Aaron,Deheuvels,Henze,Janson}. As described in \cite{Aaron},
there are many statistical applications. 
\textcolor{\blue}{Essentially, it differs from the coverage threshold $R_{n}$
only because of boundary effects (we shall elaborate on this in 
Section \ref{secdefs})},
but  these effects are often  important in determining 
the asymptotic behaviour of the threshold.

Another interpretation of the coverage threshold is via {\em Voronoi cells}.
Calka and Chenavier \cite{Calka} have considered, among other things,
extremes of  circumscribed radii of a Poisson-Voronoi tessellation
on all of  $\R^d$ (the circumscribed radius of a cell 
is the radius of the smallest
ball centred on the nucleus that contains the cell).
To get a finite maximum they consider
the maximum restricted to those cells having non-empty intersection
with  some bounded window $A \subset \R^d$.
 This construction avoids dealing with delicate boundary effects,
and the limit distribution, for large intensity,
is determined in \cite{Calka} using 
results from \cite{Janson}.

It seems at least as natural to consider Voronoi cells
with respect to the Poisson sample {\em restricted to $A$}.  
A little thought (similar to arguments given in \cite{Calka})
shows that the largest circumradius of the Voronoi cells
\textcolor{\blue}{ {\em inside $A$} (i.e., the intersections 
	with $A$ of the Voronoi cells),}
with respect to the sample $\X_n$, is equal to
$R_n$, and likewise for  a Poisson sample in $A$; thus,  our results add to
those given in \cite{Calka}. 

A somewhat related topic is 
the {\em convex hull} 
of the random sample $\X_n$. For $d=2$ with $A$ convex,
 the limiting behaviour
of the Hausdorff distance from this convex hull to $A$ is obtained in
\cite{BSB}.
The limiting behaviour of the Hausdorff distance from 
$\X_n$ itself to $A$ (which is our $R_n$) is {\em not} the same
 as for the convex hull. 

 \section{Definitions and notation}
 \label{secdefs}


Throughout this paper, we work within the following mathematical framework.  Let $d \in \N$. Suppose we have the following ingredients:

\begin{itemize}
	\item A compact, Riemann measurable set
 $A \subset \R^d$ (Riemann measurability of a bounded set
in $\R^d$ amounts to its boundary having zero Lebesgue measure). 
\item
	A Borel probability measure $\mu$ on $A$ with
		probability density function $f$.
	\item
		A specified set $B \subset A$ (possibly $B=A$).
	\item On a common  probability space
		$(\bbS,\cF,\Pr)$, a sequence  
$X_1,X_2,\ldots$ of  independent identically distributed
random $d$-vectors with common probability distribution $\mu$,
and also
		a unit rate Poisson counting
process   $(Z_t,t\geq 0)$, independent of $(X_1,X_2,\ldots)$
(so $Z_t$ is Poisson distributed with mean $t$ for each $t >0$).
\end{itemize}

For $n \in \N$, $t >0$,  let $\X_n:= \{X_1,\ldots,X_n\}$, and let
  $\Po_t:= \{X_1,\ldots,X_{Z_t}\}$.
These are the point processes that concern us here. Observe that $\Po_t$
is a Poisson point process in $\R^d $ with intensity measure $t \mu$
 (see e.g. \cite{LP}).

For $x \in \R^d$ and $r>0$ set $B(x,r):= \{y \in \R^d:\|y-x\| \leq r\}$
where $\|\cdot\|$ denotes the Euclidean norm.
\textcolor{\blue}{(We write $B_{(d)} (x,r)$ for this if we wish 
to emphasise the dimension.)}
For $r>0$, let $A^{(r)}:=\{ x \in A: B(x,r) \subset A^o\}$, the
`$r$-interior' of $A$.

 \textcolor{\blue}{Also, define the set  $A^{[r]}$
 to be
 the interior of the union of all hypercubes of the
 form $\prod_{i=1}^d[n_ir,(n_i+1)r] $, with $n_1,\ldots,n_d \in \Z$,
 that are contained in $A^o$ (the set $A^{[r]}$ resembles
  $A^{(r)}$ but is guaranteed
 to be Riemann measurable).}

For any point set $\X \subset \R^d$ and any $D \subset \R^d$ we write
$\X(D)$ for the number of points of $\X$ in $D$,
 and we use below the convention $\inf\{\} := +\infty$.

Given $n, k \in \N$, and 
 $t \in (0,\infty)$,
 define the $k$-coverage thresholds $R_{n,k}$
 and $R'_{t,k}$ 
 by
\bea
R_{n,k} : =
 \inf \left\{ r >0: \X_n   (B(x,r)) \geq k 
~~~~ \forall x \in B \right\};
\label{Rnkdef}
\\
R'_{t,k} : = R_{Z_t,k} :=
 \inf \left\{ r >0: \Po_t
(B(x,r)) \geq k 
~~~~ \forall x \in B \right\},
\label{Rdashdef}
\eea
and  define also the {\em interior $k$-coverage thresholds}
\bea
\tilde{R}_{n,k} : = \inf \left\{ r >0: 
\X_n (B(x,r)) \geq k ~~~ \forall x \in B \cap A^{(r)}
\right\};
\label{eqmaxspac}
\\
\tilde{R}_{Z_t,k} : = \inf \left\{ r >0: 
\Po_t (B(x,r)) \geq k ~~~ \forall x \in B \cap A^{(r)}
\right\}.
\label{eqmaxspacPo}
\eea
Set $R_n : = R_{n,1}$,
and  $R'_t:= R'_{t,1}$, and
 $\tR_n: = \tR_{n,1}$. Then
$R_n$ is the coverage threshold.  
Observe that
$R_n  = \inf \{ r >0: B \subset \cup_{i=1}^n B(X_i,r) \}$.
In the case $B=A$, this agrees with our earlier definition of $R_n$.

We are chiefly interested in the asymptotic behaviour of $R_n$ 
 for large $n$.  More generally, we consider $R_{n,k}$ where $k$ may
vary with $n$. 
We are especially  interested in the case 
with $B=A$.

Observe that $\tilde{R}_{n,k} $
is the smallest $r$ such that $ B \cap A^{(r)} $
is covered $k$ times by the balls
of radius $r$ centred on the points of $\X_n$. 
%
It can be seen that when $B=A$,
the maximal $k$-spacing of the sample $\X_n$ (defined earlier) is 
equal to $\theta_d \tR_n^d$,
where 
$\theta_d := \pi^{d/2}/\Gamma(1 + d/2)$,
the volume of the unit ball in $\R^d$.


We use the Poissonized $k$-coverage threshold
$R'_{t,k}$, and the interior $k$-coverage thresholds
$\tR_{n,k}$ and $\tR_{Z_t,k}$, mainly as 
stepping stones towards deriving results for $R_{n,k}$ and
$\tR_{n,k}$ respectively,
but they are also of interest in their own right. Indeed, some of the literature
\cite{HallBk,HallZW,BB,Lan} is concerned more with $R'_t$ than with $R_n$,
and we have already mentioned the literature on the maximal spacing.

We now give  some further notation used throughout. For
$D \subset \R^d$, let $\overline{D}$ 
and $D^o$
denote the closure of $D$.
and interior of $D$, respectively. 
Let $|D|$ denote the Lebesgue  measure (volume) of $D$, and
$|\partial D|$ the perimeter of $D$, i.e. the
$(d-1)$-dimensional Hausdorff measure of
$\partial D$, when these are defined.
\textcolor{\blue}{
	Write $\lglg t$ for $\log (\log t)$,
	$t >1$. 
Let $o$ denote
the origin in $\R^d$.
Set $\bH := \R^{d-1}\times [0,\infty)$ and 
 $\partial \bH := \R^{d-1}\times \{0\}$.}
 
Given two  sets $\X,\Y \subset \R^d$, we
 set  $ \X \triangle \Y := (\X \setminus \Y) 
\cup (\Y \setminus \X)$, the symmetric difference between $\X$ and $\Y$.
Also, we write $\X \oplus \Y$ for the set $\{x+y: x \in \X, y \in \Y\}$.
 Given also $x \in \R^d$ we write $x+\Y$ for $\{x\} \oplus \Y$.

Given $x,y \in \R^d$, we denote by $[x,y]$ the line segment from
$x$ to $y$, that  is, the convex hull of the set $\{x,y\}$.
We write $a \wedge b$ (respectively $a \vee b$) for the minimum
(resp. maximum) of any two numbers $a,b \in \R$.

\textcolor{\blue}{
Given  $m \in \N$ and functions
$f: \N \cap [m,\infty) \to \R$ and
$g: \N \cap [m,\infty) \to (0,\infty)$,
we write $f(n) = O(g(n))$ as $n \to \infty$ 
if  
 $\limsup_{n \to \infty }|f(n)|/g(n) < \infty$.
 We write $f(n)= o(g(n))$ as $n \to \infty$
 if $\lim_{n \to \infty} f(n)/g(n) =0$.
 We write $f(n)= \Theta(g(n))$ as $n \to \infty$
 if both $f(n)= O(g(n))$ and $g(n)= O(f(n))$.
Given $s >0$ and  functions
$f: (0,s) \to \R$
and $g:(0,s) \to (0,\infty)$,
we write $f(r) = O(g(r))$ as $r \downarrow 0$,
or $g(r) = \Omega(f(r))$ as $ r \downarrow 0$,
if 
 $\limsup_{r \downarrow 0} |f(r) |/g(r) < \infty$.
 We write $f(r)= o(g(r))$ as $r \downarrow 0$
 if $\lim_{r \downarrow 0 } f(r)/g(r) =0$,
 and $f(r) \sim g(r)$ as $r \downarrow 0$
 if this limit is 1.}

\section{Convergence in distribution}
\label{secweak}
\allco

The main results of this section are concerned with 
weak convergence for $R_{n,k}$ (defined at (\ref{Rnkdef}))
as $n \to \infty$ with $k$ fixed,
in cases where $f$ is uniform on $A$ and $B=A$.

Our first result concerns the case where
$A$ has a smooth boundary in the following sense.
We say that $A$ {\em has $C^2$ boundary} if 
for each $x \in \partial A$
there exists a neighbourhood $U$ of $x$ and a real-valued function $f$ that
is  defined on an open set in $\R^{d-1}$
and twice continuously differentiable, such
that  $\partial A \cap U$, after a rotation, is the graph of the
function $f$.

Given $d \in \N$, define the constant
\bea
c_d := \frac{1}{d!} \left( \frac{\sqrt{\pi} \; \Gamma((d/2)+1) }{ \Gamma((d+1)/2) }
\right)^{d-1}.
\label{cdef}
\eea
Note that $c_1=c_2=1$, and  $c_3 =3 \pi^2/32$.
\textcolor{\blue}{
Moreover, using Stirling's formula one can show that
$c_d^{1/d}  \sim e (\pi/(2d))^{1/2}$ as $d \to \infty$.
}
Given also $k \in \N$, for $d \geq 2$ set
\bea
c_{d,k} := \left( \frac{c_{d-1}}{(k-1)!} \right)
\theta_d^{2-d-1/d} \theta_{d-1}^{2d-3}
\theta_{d-2}^{1-d} 
(1- 1/d)^{d+k-3 + 1/d} 2^{-1+1/d}.
\label{cdkdef}
\eea
\textcolor{\blue}{
By some tedious algebra, one can simplify this to
$$
c_{d,1} =  (d!)^{-1}2^{1-d} \pi^{(d/2)-1} (d-1)^{d+(1/d)-2}\Gamma
\left( \frac{d+1}{2} \right)^{1-d} \Gamma \left( \frac{d}{2} \right)^{d+ 1/d -1},
$$
with $c_{d,k} = c_{d,1}(1-1/d)^{k-1}/(k-1)!$ for $k >1$.}
Note that $c_{2,k}
= 2^{1-k} \pi^{-1/2}/(k-1)! $ and $c_{3,k}= 2^{k-5} 3^{1-k} \pi^{5/3}/(k-1)!$,
and $c_{d,1}^{1/d} \sim e/(2d)^{1/2}$ as $d \to \infty$. 

In all  limiting statements in the sequel,
$n$ takes values in $\N$ while $t$ takes values in $(0,\infty)$.

\begin{theo}
\label{thsmoothgen}
Suppose that $d\geq 2$  and
$f=f_0{\bf 1}_A$, where $A \subset \R^d$ is compact, and
has $C^2$ boundary,
	and $f_0 := |A|^{-1}$. Assume $B=A$.
	Let $k \in \N$,
	$\zeta \in \R$.
Then
\bea
	\lim_{n \to \infty} \Pr \left[ \left( \frac{n \theta_d f_0 R_{n,k}^d}{2}
	\right)- 
	\left(\frac{d-1}{d}\right) \log (n f_0) -
\left( d+k-3+1/d \right)
	\lglg n \leq \zeta \right]
\nonumber \\
=
	\lim_{t \to \infty} \Pr \left[ \left( \frac{t \theta_d 
	f_0 (R'_{t,k})^d}{2}
	\right)- 
	\left(\frac{d-1}{d}\right) \log (t f_0) -
\left( d+k-3+1/d \right) 
	\lglg t \leq \zeta \right]
\nonumber \\
=
	\begin{cases}
		\exp \left( - |A| e^{-2 \zeta}
		- c_{2,1}  | \partial A| e^{- \zeta}  
		\right) 
		&{\rm if} ~ d=2, k =1 \\
		\exp \left( - c_{d,k}  | \partial A| e^{- \zeta}  
		\right) 
		&{\rm otherwise.}
	\end{cases}
		~~~~~~~~~~~~~~~
\label{eqmain3}
\eea
\end{theo}

When $d=2, k=1 $
the exponent in (\ref{eqmain3}) has two terms.
This is because
the location in $A$ furthest from the sample $\X_n$ might lie
either in the interior of $A$, or on the boundary. 

When $d \geq 3$ or $k \geq 2$, 
 the exponent of (\ref{eqmain3}) 
has only one term, because
 the location in $A$ with furthest $k$-th nearest point of $\X_n$
is located, with probability tending to 1, on the boundary of $A$, 
and likewise for  $\Po_t$.  Increasing either $d$ or $k$
makes it more likely that this location lies on the boundary,
and the exceptional nature of the limit 
in the case $(d,k)=(2,1)$  reflects this.

We also consider the case where $A$ is a polytope, only for
$d=2$ or $d=3$. All polytopes in this paper are assumed to
be bounded, connected, and finite (i.e. have finitely many
faces).

\begin{theo}
\label{thmwksq}
Suppose $d=2$, 
$B=A$, and 
 $f = f_0  {\bf 1}_A$, where $f_0 : = |A|^{-1}$.
	\textcolor{\blue}{Assume that $A$ is compact and polygonal.}
Let $|\partial A|$ denote the length of the boundary of $A$.
Let $k \in \N$, $\zeta \in \R$.
 Then 
\bea
\lim_{n \to \infty} 
	\Pr[ n (\pi/2)
	f_0 R_{n,k}^2 -(1/2) \log (n f_0) - (k-1/2) \lglg n \leq \zeta ]
\nonumber \\
	= \lim_{t \to \infty} \Pr[ t (\pi/2) f_0 (R'_{t,k})^2 -
	(1/2) \log (t f_0) - (k-1/2) \lglg t \leq \zeta ]
\nonumber \\
= \begin{cases} 
\exp( - |A|  e^{-2 \zeta}  - |\partial A| \pi^{-1/2} e^{-\zeta} )
& \mbox{\rm if } k=1,   
	\\ \exp( - c_{2,k} |\partial A|
	e^{-\zeta} )  
& \mbox{\rm if } k\geq 2.  
\end{cases}
\label{0322b}
\eea
\end{theo}

\textcolor{\blue}{
One might seek to extend Theorem \ref{thmwksq}
to a more general class of sets $A$ including
both polygons (covered by Theorem \ref{thmwksq})
and sets with $C^2$ boundary (covered by Theorem
\ref{thsmoothgen}).}
One could take sets $A$ having piecewise
$C^2$ boundary, with the extra condition that the
corners of $A$ are {\em not too pointy}, in the sense
that for each corner $q$, there exists a triangle with vertex
at $q$ that is contained in $A$. We would expect that it is 
possible to extend the result to this more
general class.

When $d=3$ and $A$ is  polyhedral,
there are several cases to consider, depending on the value
of the angle $\alpha_1$ subtended by the `sharpest edge' of $\partial A$.
\textcolor{\blue}{The angle subtended
by an edge $e$ is defined as follows. Denote
the two faces meeting at $e$ by $F_1,F_2$. Let $p$ be a point
in the interior of the edge, and for $i=1,2$ let
$\ell_i$ be a line segment starting from $p$
that lies within $F_i$ and is perpendicular to the edge $e$.
Let $\theta$ denote the angle between the
 line segments $\ell_1$ and $\ell_2$ (so $0 < \theta < \pi$).
The angle subtended by the edge $e$ is $\theta$ 
 if there is a neighbourhood $U$ of $p$ such that 
 $A \cap U$ is convex, and is $2 \pi - \theta$ if there is no
 such neighbourhood of $p$.}

If $\alpha_1 < \pi/2$ then the location in $A$ furthest from the sample
$\X_n$
is likely to be on a 1-dimensional edge of $\partial A$, while
if $\alpha_1 > \pi/2$ the furthest location from the sample is likely
to be on a 2-dimensional face of $\partial A$, in the limit $n \to \infty$.
If $\alpha_1 = \pi/2$ (for example, for a cube), both of these possibilities
have non-vanishing probability in the limit.

Since  there are several cases to consider, to make the statement of the result
more compact we put it in terms of $\Pr[R_{n,k} \leq r_n]$
for a sequence of constants $(r_n)$. 

\begin{theo}
\label{thwkpol3}
Suppose $d=3$, $A $ is polyhedral, $B=A$ and $f = f_0 {\bf 1}_A$, where
$f_0 := |A|^{-1}$. Denote the $1$-dimensional edges of $A$ by 
$e_1,\ldots,e_\kappa $.
For each $i \in \{1,\ldots, \kappa \}$, let $\alpha_i$ denote 
the angle that $A$ subtends at edge $e_i$ (with $0 < \alpha_i < 2 \pi$),
and write $|e_i|$ for the length of $e_i$. Assume the edges are listed in
order so that  $\alpha_1 \leq \alpha_2 \leq \cdots \leq \alpha_\kappa$. 
Let $|\partial_1 A|$ denote the total area (i.e., $2$-dimensional Hausdorff
 measure) of all faces of $A$, and let $|\partial_2 A| $
denote the total length of those edges $e_i$ for which $\alpha_i= \alpha_1$. 
Let $\beta \in \R$ and $k \in \N$.  
Let $(r_t)_{t >0}$ be a family  of real numbers satisfying
(as $t \to \infty$)
\bean
 (2 \alpha_1  \wedge \pi ) f_0 t r_t^3
-  \log (t f_0) - \beta
  =  \begin{cases}
  (3k -1)  \lglg t  + o(1)
  & \mbox{ if  }  \alpha_1 \leq \pi/2
\\
(\frac{1+ 3k}{2}) \lglg t + o(1)
  & \mbox{ if  }  \alpha_1 > \pi/2.
\end{cases}
\eean
Then
\bea
\lim_{n \to \infty} \Pr[ R_{n,k} \leq r_n ] 
= \lim_{t \to \infty} \Pr[ R'_{t,k} \leq r_t ] 
~~~~~~~~~~~~~~~~~~~~~
~~~~~~~~~~~~~~~~~~~~~
\nonumber \\
= \begin{cases}
\exp \left( - \frac{3^{1-k} \alpha_1^{1/3}|\partial_2 A|
e^{-\beta/3}
}{ (k-1)! (32)^{1/3}}
\right)  & \mbox{\rm if } 
\alpha_1  < \pi/2,  ~ \mbox{\rm or } (\alpha_1 = \pi/2, k >1)
 \\
 \exp \left( - \frac{ 3 \pi^{5/3} 2^k  |\partial_1 A|
e^{-2 \beta/3} }{(k-1)! 3^k 32  }  \right) & \mbox{\rm if }  
\alpha_1 > \pi/2
\\
\exp \left( - \frac{
\pi^{1/3} 
|\partial_2 A|
e^{-\beta/3}
}{4} 
- \frac{ \pi^{5/3} |\partial_1 A|e^{- 2 \beta/3}}{16} \right) & \mbox{\rm if }  
\alpha_1 = \pi/2, k=1.  
 \end{cases}
\label{0529a}
\eea
\end{theo}

We now give a result in general $d$ for $\tR_{n,k}$,
and  for $R_{n,k}$ in the case with $\overline{B} \subset A^o$
(now we  no longer require $B=A$).
These cases are simpler
because boundary effects are avoided. In fact, the result stated below
has some overlap with  already known results; it is convenient to state it
here too for comparison  with the results just given, and because we
shall be using it to prove those results.

\begin{prop}
\label{Hallthm}
Suppose $A$ is compact with $ |A| > 0$
	\textcolor{\blue}{and $f = f_0 {\bf 1}_A$, where
	$f_0 := |A|^{-1}$,}
	and $B \subset A$ is Riemann measurable.
	Let $k \in \N$ and 
	$\beta \in \R$.
Then
\bea
\lim_{n \to \infty} \Pr[ n \theta_d f_0  \tR_{n,k}^d   - \log (n f_0)  -
(d+k-2) \lglg n \leq \beta] 
\nonumber \\
=
\lim_{t \to \infty} \Pr[ t \theta_d f_0 (\tR_{Z_t,k})^d   - \log (t f_0)  -
(d+k -2) \lglg t \leq \beta]
\nonumber \\
= \exp(-(c_d / (k-1)!) |B|e^{-\beta} ).
\lbl{1228a}
\eea  
	If moreover $B$ is
closed 
	with $B \subset A^o$, then
\bea
\lim_{n \to \infty} \Pr[ n \theta_d f_0  R_{n,k}^d   - \log (n f_0)  -
(d+k-2) \lglg n \leq \beta] 
\nonumber \\
=
\lim_{t \to \infty} \Pr[ t \theta_d f_0 (R'_{t,k})^d   - \log (t f_0)  -
(d+k -2) \lglg t \leq \beta]
\nonumber \\
= \exp(-(c_d / (k-1)!) |B|e^{-\beta} ).
\lbl{0114a}
\lbl{0114b}
\eea  
\end{prop}
The case $k=1$  of  the second equality of (\ref{0114a})
can be found in \cite{Calka}.
It provides a stronger asymptotic result
 than the one in \cite{Lan}.
A similar statement to the case $k=1$ of
(\ref{1228a})
can be found in \cite{Janson2}.

\begin{remk}
	{\rm	Let  $k \in \N$.
	The  definition  (\ref{Rnkdef}) of
	the coverage threshold $R_{n,k}$ suggests 
	  we think of the number and (random) locations of points
	as being given, and consider the smallest radius of balls
	around those
	points needed to cover $B$ $k$ times.
	Alternatively, as in 
	\cite{Janson},  one may
	 think of the radius  of the
	balls as being given, and consider the smallest number of
	balls (with locations generated sequentially at random) needed to
	cover $B$ $k$ times.
	That is,
	given $r>0$, define the random variable
	$$
	N(r,k) := \inf \{ n \in \N: \X_n (B(x,r)) \geq k ~~~ \forall x \in B\},
	$$
	and note that $N(r,k) \leq n$ if and only if $R_{n,k} \leq r$.
	In the setting of  Theorem \ref{thsmoothgen}, \ref{thmwksq} 
	or \ref{thwkpol3}, one may 
	obtain a limiting distribution
	for $N(r,k)$ (suitably scaled and centred)
	as $r \downarrow 0$ by using those
	results together with the following (we write $\tod$ for
	convergence in distribution):}
\end{remk}
	
	\begin{prop}
		\label{lemrewrite}
	Let $k \in \N$. 
		Suppose $Z$ is a random variable with a continuous cumulative
	distribution function, and $a, b, c, c' \in \R$ with $a >0, b >0$ are
		such that
	$
		a n R_{n,k}^d - b \log n -c \lglg n \textcolor{\blue}{- c'
		} \tod Z
	$
	as $n \to \infty$.
	Then as $r \downarrow 0$,
	$$
	a r^d  N(r,k) - b \log \left( (b/a) r^{-d} \right)
		- (c+ b)  \lglg (r^{-d}) \textcolor{\blue}{- c'} \tod Z.
	$$
	\end{prop}
	\textcolor{\blue}{
	For example, using Theorem \ref{thsmoothgen}, and applying
	Proposition \ref{lemrewrite} with 
	$a = \theta_d f_0/2$, $b= (d-1)/d$,
	$c = d+ k -3 +1/d$, and $c' = ((d-1)/d) \log f_0$,
	we obtain the following:
	\begin{coro}
		\label{corosmooth}
Suppose that $d\geq 2$  and
$f=f_0{\bf 1}_A$, where $A \subset \R^d$ is compact, and
has $C^2$ boundary,
	and $f_0 := |A|^{-1}$. Assume $B=A$.
	Let $k \in \N$,
	$\zeta \in \R$.
Then
\bea
		\lim_{r \downarrow 0}
		\Pr \Big[ \left( \frac{ \theta_d 
		f_0 r^d N(r,k)}{2}
	\right)- 
		\left(\frac{d-1}{d}\right) \log 
		\left( \left(\frac{2 (d-1)}{d 
		\theta_d} \right) r^{-d} \right)
\nonumber \\
		- \left( d+k-2 \right)
		\lglg r^{-d} \leq \zeta \Big]
\nonumber \\
=
	\begin{cases}
		\exp \left( - |A| e^{-2 \zeta}
		- c_{2,1}  | \partial A| e^{- \zeta}  
		\right) 
		&{\rm if} ~ d=2, k =1 \\
		\exp \left( - c_{d,k}  | \partial A| e^{- \zeta}  
		\right) 
		&{\rm otherwise.}
	\end{cases}
		~~~~~~~~~~~~~~~
\label{eqmain30}
\eea
	\end{coro}
	}

\section{Strong laws of large numbers}
\label{secLLN}
\allco
The results in this section  provide
  strong laws of large numbers (SLLNs) for $R_{n}$. For these
  results we relax the condition that $f$ be uniform on $A$.
We give strong laws for $R_n$ when $A=B$ and $A $ is either
smoothly bounded or a polytope.
Also for general $A$ we
give strong laws for $R_n$ when $\overline{B} \subset A^o$,
and for $\tR_n$ 
for general $B$.

More generally, we consider $R_{n,k}$, allowing $k$ to vary with $n$.
	Throughout this section, assume we are given 
	a constant $\beta \in [0,\infty]$
	and a sequence $k: \N \to \N$ with
	\bea
	\lim_{n \to \infty} \left( k(n)/\log n \right) = \beta;
	~~~~~ \lim_{n \to \infty} \left( k(n)/n \right) = 0.
	\label{kcond}
	\eea
We make use of the following notation throughout:
\bea
f_0 := \essinf_{x \in B} f(x); ~~~~~~~ f_1:= \inf_{x \in \partial A}
f(x);
\label{f0def}
\\
 H(t) := \begin{cases}  1-t + t \log t, ~~~ & {\rm if} ~ t >0 \\
	 1 ,  & {\rm if} ~ t =0.
 \end{cases}
	 \label{Hdef}
\eea
	Observe that $-H(\cdot)$ is unimodal with a maximum
value of $0$ at $t=1$. Given $a \in [0,\infty)$, we define
the
function $\hH_a: [0,\infty) \to [a,\infty)$ by
\bean
y = \hH_a(x)  \Longleftrightarrow y H(a/y) =x,~ y \geq a,
\eean
with $\hH_0(0):=0$.
Note that 
$\hH_a(x)$  is  increasing in  $x$,
and that 
$\hH_0(x)=x$.

	Throughout this paper, the phrase `almost surely' or 
	`a.s.' means `except on a set of $\Pr$-measure zero'.
 We write $f|_A$ for the restriction of $f$ to $A$.
 \textcolor{\blue}{If $f_0=0,$ $ b>0$ we interpret $b/f_0$ as $+\infty$ in
 the following limiting statements, and likewise for $f_1$.}

\begin{theo}
\label{thm2}
Suppose that $d \geq 2$ and $A \subset \R^d$ is compact with
$C^2$ boundary,
	and that $f|_A$ is continuous at $x$ for all
$x \in \partial A$. Assume also that $B=A$ and 
	  (\ref{kcond}) holds.
Then as $n \to \infty$, almost surely
\bea
(n \theta_d R_{n,k(n)}^d)/k(n) \to \max \left( 1/f_0, 2 /f_1 \right)
	~~~~~{\rm if}~ \beta =\infty;
\label{th2eq1}
\\
(n \theta_d R_{n,k(n)}^d)/\log n \to 
	\max \left( \hH_\beta(1)/f_0, 2 \hH_{\beta}(1-1/d)/f_1  
\right),
	~~~~~{\rm if}~ \beta < \infty.
\label{th2eq}
\eea
In particular, if $k \in \N$ is a constant, then
as $n \to \infty$, almost surely 
\bea
	(n \theta_d R_{n,k}^d)/\log n \to 
	\max \left( 1/f_0,  (2-2/d)/f_1 \right).
\label{th2eq0}
\eea
\end{theo}

We now consider the case where 
$A$ is a polytope. Assume the polytope $A$ is compact
and finite,
that is, has finitely many faces. 
Let $\Phi(A)$ denote the set of all faces of $A$ (of all dimensions).
Given a  face $\varphi \in \Phi(A)$, denote the
dimension of this face by  $D(\varphi)$.
Then $0 \leq D(\varphi) \leq d-1$, and $\varphi$ is a 
$D(\varphi)$-dimensional polytope embedded in $\R^d$.
 Set $f_\varphi := \inf_{x \in \varphi}f(x)$,
 let $\varphi^o$ denote the relative interior of $\varphi$,
and set $\partial \varphi := \varphi \setminus \varphi^o$.
Then there is a cone $\cK_\varphi$ in $\R^d$ such that every $x \in \varphi^o$
has a neighbourhood $U_x$ such that $A \cap U_x = (x+ \cK_\varphi) \cap U_x$.
Define the angular volume $\rho_\varphi$ of $\varphi$ 
to be the $d$-dimensional Lebesgue measure of
$\cK_\varphi \cap B(o,1)$. 

For example, if
$D(\varphi)=d-1$ then $\rho_\varphi = \theta_d/2$.
If $D(\varphi) = 0$ then $\varphi = \{v\}$ for some vertex $v \in  \partial A$,
and $\rho_\varphi$ equals 
\textcolor{\blue}{
the volume of $B(v,r) \cap A$, }
	divided by $r^d$,
for all sufficiently small $r$.
If $d=2$, $D(\varphi)=0$ and $\omega_\varphi$ denotes
the angle subtended by $A$ at the vertex $\varphi$, 
then $\rho_{\varphi} = \omega_\varphi/2$. If $d=3$ and
$D(\varphi)=1$, and  $\alpha_\varphi$ denotes
the angle subtended by $A$ at the edge $\varphi$ (which is either the angle
between the two boundary  planes of $A$ meeting at $\varphi$, or
$2 \pi$ minus this angle), then
$\rho_{\varphi} = 
2 \alpha_\varphi/3$.

For $d \geq 4$, our result for $A$ a polytope includes 
a condition that the polytope be convex; we conjecture
that this condition is not needed. We include connectivity
in the definition of a polytope, so for $d=1$ 
a polytope is defined to be an interval.

\begin{theo}
	\label{thmpolytope}
	Suppose $A$ is a compact finite polytope in $\R^d$. 
	If $d \geq 4$, assume  moreover that $A$ is convex.
	Assume that
	$f|_A$ is continuous at $x $ for all $x \in \partial A$, and
	set $B=A$. Assume $k(\cdot)$ satisfies (\ref{kcond}).
	Then,
	almost surely,
\bea
\lim_{n \to \infty} nR_{n,k(n)}^d/ k( n) & = & \max
	\left( \frac{1}{f_0 \theta_d}, \max_{\varphi \in \Phi(A)} \left(
	\frac{1}{f_\varphi \rho_\varphi }   
\right) \right), 
~~~~~
~~
	{\rm if} ~\beta = \infty;
	\label{0717a}
\\
\lim_{n \to \infty} nR_{n,k(n)}^d/ \log n & = & \max 
	\left( \frac{ \hH_\beta(1) }{f_0 \theta_d} ,
	\max_{\varphi \in \Phi(A)} \left( \frac{
		\hH_\beta( D(\varphi)/d)  }{ 
f_\varphi \rho_\varphi}  
	\right) \right), ~~~{\rm if} ~\beta < \infty.
	\nonumber \\
	\label{0717b}
\eea
\end{theo}

In the next three results,
we spell out some special cases of Theorem \ref{thmpolytope}.
\begin{coro}
\label{thpolygon}
Suppose that $d=2$,  $A$ is polygonal with
 $B=A $, and  $f|_A$ is continuous at $x$ for all $x \in 
	\partial A$. Let $V$ denote the set of vertices of $A$, and
	for $v \in V$ let $\omega_v$ denote the angle
	subtended by $A$ at vertex $v$.
	Assume (\ref{kcond}) with $\beta < \infty$.
Then, almost surely,
\bea
\lim_{n \to \infty} \left( \frac{ n  R_{n,k(n)}^2}{\log n}\right) =
	\max \left( \frac{\hH_\beta(1)}{\pi f_0}, 
	\frac{2 \hH_\beta(1/2)}{\pi f_1},
	\max_{v \in V}  \left( 
\frac{2  \beta }{  \omega_{v} f(v) } \right) 
	\right) .
	\label{polystrong2}
\eea
In particular, for any constant $k \in \N$,
$
\lim_{n \to \infty} \left( \frac{ n \pi R_{n,k}^2}{\log n}\right) =
 \max \left( \frac{1}{f_0}, \frac{1}{ f_1} \right) .
 $
\end{coro}

\begin{coro}
\label{thpoly}
	Suppose $d=3$ (so $\theta_d = 4\pi/3$),  $A$ is polyhedral with
 $B=A $, and $f|_A$ is continuous at $x$ for all $x \in 
\partial A$. 
	Let $V$ denote the set of vertices of $A$, and
	 $E$  the set of edges  of $A$. For $e\in E$,
let $\alpha_e$ denote the angle subtended by $A$ at edge $e$,
and  $f_e$  the infimum of
$f $ over $e$.  
	For $v \in V$ let $\rho_v$ denote the
	angular volume of vertex $v$.
	Suppose (\ref{kcond}) holds 
with $\beta <  \infty$. Then, almost surely,
\bean
\lim_{n \to \infty} \left( \frac{ n  R_{n,k(n)}^3}{\log n}\right) =
	\max \left( \frac{\hH_\beta(1)}{\theta_3 f_0}, 
	\frac{2 \hH_\beta(2/3)}{\theta_3 f_1},
	\frac{3 \hH_\beta(1/3) }{ 2 \min_{e \in E} (\alpha_e f_e ) } 
	, 
\max_{v \in V}  \left( 
\frac{\beta }{ \rho_{v} f(v) } \right) 
 \right) .
\eean
In particular, if $\beta =0$ the above limit comes to
 $\max \left( \frac{3}{ 4 \pi f_0}, \frac{1}{\pi f_1},
\max_{e \in E}  \left( 
\frac{1 }{2 \alpha_e f_e } \right)
 \right) $.
\end{coro}
\begin{coro}
\label{thmbox}
Suppose $A=B = [0,1]^d$, and $f|_A$ is continuous at $x$ for all $x \in 
\partial A$. For $1 \leq j \leq d$ let $\partial_j$ denote
the union of all $(d-j)$-dimensional faces of $A$, 
	and let $f_j$ denote the infimum of
$f $ over $\partial_j$. Assume 
	(\ref{kcond})  with $\beta < \infty$. 
  Then
\bea
\lim_{n \to \infty} \left( \frac{ n  R_{n,k(n)}^d}{\log n}\right) =
	\max_{0 \leq j \leq d} \left( \frac{2^j \hH_\beta(1-j/d)  }{ \theta_d f_j } \right),
	~~~~~ a.s.
\label{0505b}
\eea
\end{coro}

It is perhaps worth spelling out what the preceding results mean in the
special case where
$\beta =0$ (for example, if $k(n)$ is a constant)  and also
$\mu$ is the uniform distribution on
$A$ (i.e. $f(x) \equiv f_0$  on $A$).
In this case,
the right hand
side of (\ref{th2eq0}) comes to $(2-2/d)/f_0$,
while the right hand
side of (\ref{0717b}) comes to 
$f_0^{-1}\max(1/ \theta_d,
\max_{\varphi \in \Phi(A)} \frac{D(\varphi)}{(d \rho_\varphi)})$.
The limit in (\ref{polystrong2}) comes to $1/(\pi f_0)$,
while the limit in Corollary \ref{thpoly} comes to 
\linebreak
$f_0^{-1} \max[ 1/\pi, \max_e (1 /(2 \alpha_e))]$,  
and
the right hand side of (\ref{0505b}) comes to $2^{d-1}/(\theta_d d)$.

\begin{remk}
	{\rm 
The notion of coverage threshold is analogous to that of
{\em connectivity threshold} in the theory of random geometric
graphs \cite{RGG}. 
Our  results show that the threshold
for full coverage by the balls $B(X_i,r)$, is asymptotically twice
the threshold for the union of these balls to be connected, if
$A^o$ is connected,
at least when $A$ has a smooth boundary or $A =[0,1]^d$. 
	This can be seen from comparison of Theorem \ref{thm2} above
	with  \cite[Theorem 13.7]{RGG}, and comparison of
	Corollary \ref{thmbox} above with \cite[Theorem 13.2]{RGG}.
	More general polytopes were not considered in \cite{RGG}.}
\end{remk}

\begin{remk}
	{\rm
	We compare our results
	with \cite{Cuevas}.
	In the setting
	of our Theorem \ref{thm2}, 
	\cite[Theorem 3]{Cuevas} and \cite[Remark 1]{Cuevas} give an
	upper bound on 
	$\limsup_{n \to \infty}(n \theta_d R_n^d/\log n)$ of
	$\max (f_0^{-1}, 2f_1^{-1})$ in probability or 
	$\max (2f^{-1}_0, 4f_1^{-1})$ almost surely. In the
	setting of our Theorem \ref{thmpolytope}, they give
	an upper bound of $f_0^{-1} \vee
	\max_{\varphi} (\theta_d/(f_\varphi \rho_\varphi))$ in
	probability, or twice this almost surely.
	Our (\ref{th2eq0}) and (\ref{0717b}) improve significantly
	on those results.
	}
\end{remk}

\begin{remk}
\textcolor{\blue}{
	{\rm 
	In Theorem \ref{thmpolytope} we do not consider the case 
	where $A$ is a non-convex polytope in dimension $d \geq 4$.
	Even the definition of non-convex polytope when $d \geq 4$
	seems not to be universally agreed on, and to generalize
	the proof of Lemma \ref{lemangle} to this case would seem to
	require considerably more work.}
}
\end{remk}

Our final result is a law of large numbers for
$\tR_{n,k}$, no longer requiring $B=A$. In the case
where $B$ is contained in the interior of $A$, 
 this easily yields a law of large numbers for
the $k$-coverage threshold $R_{n,k}$.  
\textcolor{\blue}{Recall from Section \ref{secdefs} that
$\mu$ denotes the probability measure on $A$
with density $f$.}

\begin{prop}
\label{thm1}
 Suppose that either (i) 
	$B$ is compact and Riemann measurable with $\mu(B) >0$,
	and $B \subset A^o$,
	and $f$ is continuous on $A$; or
 (ii) $B=A$.
	Assume (\ref{kcond}) holds.
	Then, almost surely,
\bea
\lim_{n \to \infty} ( n \theta_d \tR_{n,k(n)}^d/k(n)) = f_0^{-1}
	~~~~~{\rm if}~ \beta = \infty;
\label{0617a}
\\
	\lim_{n \to \infty} (n \theta_d  \tR_{n,k(n)}^d /\log n) =  
	\hH_\beta(1)/f_0,
	~~~~~{\rm if}~ \beta < \infty.
\label{0315a}
\eea
In particular, if $k \in \N$ is a constant then
\bea
 \Pr[ \lim_{n \to \infty} ( n \theta_d f_0 \tR_{n,k}^d /\log n) = 1 ] =1.
\label{Strong1}
\eea
In Case (i), all of the above almost sure limiting
 statements hold for $R_{n,k(n)}$ as well as for $\tR_{n,k(n)}$.
\end{prop}

Proposition \ref{thm1} has some overlap 
 with known results;
 the uniform case  with $A=B = [0,1]^d$ and $f \equiv 1$ on $A$ is
covered by  
\cite[Theorem 1]{Deheuvels}. 
Taking ${\bf C}$ of that paper to be the class of 
Euclidean balls centred on the origin, we see 
that the quantity denoted $M_{1,m}$ in \cite{Deheuvels} equals $\tilde{R}_n$.
In \cite[Example 3]{Deheuvels} it is stated that the Euclidean balls
satisfy the conditions of \cite[Theorem 3]{Deheuvels}. 
See also \cite{Janson2}. Note also that \cite{Iyer}
has a result similar to the case of Proposition \ref{thm1} where $d=2$, 
$A=B= [0,1]^2$ and $f$ is uniform over $A$.

\section{Strategy of proofs}
\label{secstrategy}

\textcolor{\blue}{
We first give an overview of the strategy for the proofs, in Section
\ref{secLLNpfs}, of the strong laws of large numbers
that were stated in Section \ref{secLLN}.}

For $n \in \N$ and $p \in [0,1]$ let $\Bin(n,p)$ denote
  a binomial random variable with parameters $n,p$.
  Recall that $H(\cdot)$ was defined at (\ref{Hdef}), \textcolor{\blue}{and
  $Z_t$ is  a Poisson$(t)$ variable for $t>0$.}
  The proofs in Section \ref{secLLNpfs} rely heavily on the
  following lemma.
\begin{lemm}[Chernoff bounds]
\label{lemChern}
Suppose  $n \in \N$, $p \in (0,1)$, $t >0$ and $0 < k < n$. 

	(a) If $k \geq np$ then $\Pr[ \Bin(n,p) \geq k ] \leq \exp\left( 
	- np H(k/(np) )
\right)$. 

	(b) If $k \leq np$ then $\Pr[ \Bin(n,p) \leq k ] \leq \exp\left( - np H(k/(np))
\right)$. 

	(c) If $k \geq e^2 np$ then 
 $\Pr[ \Bin(n,p) \geq k ] \leq \exp\left( - (k/2)
	\log(k/(np))\right) \leq e^{-k}$.

	\textcolor{\blue}{
	(d) If $k < t$ then $\Pr[Z_t \leq k ] \leq
	\exp(- t H(k/t))$.}

	\textcolor{\blue}{
	(e) If $k \in \N$ then $\Pr[Z_t = k ] \geq
	(2 \pi k)^{-1/2}e^{-1/(12k)}\exp(- t H(k/t))$.
	}
\end{lemm}
\begin{proof}
	See e.g. \cite[Lemmas 1.1, 1.2 and 1.3]{RGG}.
	\qed
\end{proof}


\textcolor{\blue}{
Recall that $R_{n,k}$ is defined at (\ref{Rnkdef}),  that
we assume $(k(n))_{n \geq 1}$
satisfies (\ref{kcond}) for some $\beta \in [0,\infty]$, and
that $\hH_\beta(x)$ is defined to be the $ y \geq \beta $
such that $y H(\beta/y) =x$, where $H(\cdot)$ was defined
at (\ref{Hdef}).}

\textcolor{\blue}{
If $f \equiv f_0$ on $A$ for a constant $f_0$, and $r>0$, then
for $x \in A^{(r)}$,   $\X_n(B(x,r))$ is binomial with mean
$n f_0 \theta_d r^d =: M$, say.
Hence if $M > k(n)$, then
parts (b) and (e) of
Lemma \ref{lemChern}
suggest
that we should have
$$
\Pr[\X_n(B(x,r)) < k(n) ] \approx \exp(- M H(k(n)/M)) .
$$
If 
$\beta < \infty$, and 
we choose $r=r_n$ so that
$M = a \log n$ with $a = \hH_\beta(1)$ 
this probability approximates
to $\exp( -  a (\log n) H( (\beta \log n) / (a \log n)))$,
which comes to $n^{-1}$.
Since we can find 
$n^{1+o(1)}$
disjoint balls of radius
$r_n$ where this might happen, this suggests $a = \hH_\beta(1)$
is the critical  value of $a := M/\log n$,
below  which the interior region $A^{(r_n)}$
is not covered ($k(n)$ times)
with high probability, and above which it is covered.
We can then improve the `with high probability' statement
to an `almost surely for large enough $n$' statement  using Lemma
\ref{lemtrick}.
If $f$ is continuous but not constant on $A$,
the value of $f_0$ defined at (\ref{f0def}) still
determines the critical choice of $a$. If $\beta = \infty$
instead, taking  $M= a' k(n)$ the critical value of $a'$ is now 1.
These considerations lead to Proposition \ref{thm1}.
}

\textcolor{\blue}{
Now consider the boundary region $A \setminus A^{(r)}$,  in the case where
$\partial A$ is smooth.
We can argue similarly to before, except that for  $x \in \partial A$ 
the approximate mean of $\X_n(B(x,r))$
changes to $nf_1 \theta r^d/2 =: M'$. 
If $\beta < \infty$, and we now choose $r = r_n$ so $M' = a' \log n$
with $a' = \hH_\beta(1-1/d)$, then $\Pr[\X_n(B(x,r_n) ) < k(n) ]
\approx n^{-(1-1/d)}$.
Since we can find 
$n^{1-1/d + o(1)}$ 
disjoint balls of radius
$r_n$ centred in $A \setminus A^{(r_n)}$, this suggests 
$a' := \hH_\beta(1-1/d)$ 
 is the critical  choice of $a'$ for covering $A \setminus A^{(r_n)}$.
 }

 \textcolor{\blue}{
For polytopal $A$ we consider covering the regions 
near to each of the
lower dimensional faces of $\partial A$ in an analogous way;
the dimension of a face affects both the $\mu$-content
of a ball centred on that face, and the number of
disjoint balls that can packed into the region near the face.
}

 \textcolor{\blue}{
Next, we describe the strategy for the proof, in Section
\ref{secpfwk}, of
the weak convergence results that were stated in Section \ref{secweak}.
}

 \textcolor{\blue}{
First, we shall provide a general `De-Poissonization' lemma 
(Lemma \ref{depolem}), as a result of which for each of the
Theorems in Section \ref{secweak} it suffices to
prove the results for a Poisson process rather than
a binomial point process (i.e., for $R'_{t,k}$ rather
than for $R_{n,k}$).
}

 \textcolor{\blue}{
Next, we shall provide a general lemma (Lemma \ref{lemHall}) 
giving the limiting probability of covering ($k$ times, with
$k$ now fixed)
a  bounded region
of $\R^d$ by a spherical Poisson Boolean model (SPBM) on
the whole of $\R^d$, in the limit of 
high intensity and small balls. This is based on results from
\cite{Janson} (or \cite{HallZW} if $k=1$).
Applying Lemma \ref{lemHall} for an SPBM with all balls of
the same radius yields a proof of Proposition \ref{Hallthm}.
}

 \textcolor{\blue}{
Next, we shall consider the SPBM with balls of equal radius $r_t$,
centred on a homogeneous Poisson process of intensity 
$t f_0 $ in a $d$-dimensional half-space.  In the large-$t$
limit, with $r_t$ shrinking in an appropriate way,
in Lemma \ref{lemhalf3a} we determine the limiting probability that the
a given  bounded 
set within the surface of the half-space is covered, by applying Lemma 
\ref{lemHall} in $d-1$ dimensions, now with balls of
varying radii. Moreover we will show
that the probability  that a region in the half-space  within distance $r_t$
of  that set  is covered with the same limiting probability.
}

 \textcolor{\blue}{
We can then complete the proof of 
Theorem \ref{thmwksq} by applying Lemma \ref{lemhalf3a}
to determine the limiting probability that  the
region near the boundaries of a polygonal set $A$ is covered,
and Proposition  \ref{Hallthm} to determine the limiting
probability that the interior region is covered,
along with a separate argument to show the regions
near the corners  of polygonal $A$ are  covered with high probability.
}

 \textcolor{\blue}{
For Theorem \ref{thwkpol3}
we again use Proposition \ref{Hallthm} to determine the limiting
probability that the interior region is covered,
and Lemma \ref{lemhalf3a} to  determine the 
limiting probability that the region near the faces
(but not the edges) of polyhedral $A$ are covered.
To deal with  the region near the edges of $A$,
we also require a separate lemma (Lemma \ref{lemwedge})
determining the limiting probability that
a bounded region near the edge of an infinite
wedge-shaped region in $\R^3$ is covered by
an SPBM restricted to that wedge-shaped region.}

 \textcolor{\blue}{
The proof of Theorem \ref{thsmoothgen} requires further ideas.
}
Let $\gamma \in (1/(2d),1/d)$. We shall
work simultaneously on two length scales,
namely the radius $r_t$ satisfying (\ref{rt3c})  
(and hence  satisfying $r_t = \Theta (((\log t)/t)^{1/d})$),
and a coarser length-scale given by
$t^{-\gamma}$. If $d=2$, we approximate to $\partial A$ by
a polygon 
of side-length
that is
$\Theta(t^{-\gamma})$. We approximate to $\Po_t$
by a Poisson process inside the polygon, 
and can determine the asymptotic probability
of  complete coverage of this approximating polygon by considering
a lower-dimensional Boolean model on each of  the edges. In
fact we shall line up these edges by means rigid motions, into a line
segment embedded in the plane; in the limit we obtain the limiting
probability of covering this line segment with balls centred on
a Poisson process in the half-space to one side of this line segment,
which we know about by Lemma \ref{lemhalf3a}.
By separate estimates we can show that the error
terms either 
from the corners of the polygon, or from the approximation
of Poisson processes, are negligible in the large-$t$ limit.

For $d \geq 3$
we would like to approximate
to $A$ by a polyhedral set $A_t$ obtained by taking its surface to be
a triangulation
of the surface of $A$ with side-lengths $\Theta(t^{-\gamma})$.
However, 
two obstacles to this strategy present themselves.

The first obstacle is that in 3 or more dimensions, it is harder to
be globally explicit about the set $A_t$ and the set difference 
$A \triangle A_t$. We deal with this by triangulating
$\partial A$  locally rather than globally; we break $\partial A$ into
finitely many pieces, each of which lies within a single chart within
which $\partial A$, after a rotation,
can be expressed as the graph of a $C^2$ function
on a region $U$ in $\R^{d-1}$. Then we triangulate $U$ (in
the sense of tessellating it into simplices) explicitly and  use this
to determine an explicit local triangulation (in the sense
of approximating the curved surface by a union of $(d-1)$-dimensional
simplices) of $\partial A$.

The second obstacle is that the simplices in the triangulation
cannot in general be reassembled into a $(d-1)$-dimensional cube.
To get around this, we shall pick $\gamma' \in (\gamma,1/d)$
and subdivide these simplices 
into smaller $(d-1)$-dimensional cubes
of side $t^{-\gamma'}$; we can
reassemble these smaller $(d-1)$-dimensional cubes 
into a cube
in $\R^{d-1}$, and
control the boundary region near the boundaries of the
smaller $(d-1)$-dimensional cubes, or near the boundary of the simplices,
by separate estimates.

\section{Proof of strong laws of large numbers}
\label{secLLNpfs}
\allco

In this section we prove the results stated in Section \ref{secLLN}.
Throughout this section we are assuming we are  given a 
constant $\beta \in [0,\infty]$
and a sequence $(k(n))_{n \in \N}$ satisfying (\ref{kcond}). 
Recall that $\mu$ denotes the distribution of $X_1$,
and this has a density $f$ with support $A$, and that $B \subset A$
is fixed, and $R_{n,k}$ is defined at (\ref{Rnkdef}).

We shall repeatedly use the following lemma. It 
is based on what in \cite{RGG} was called
the `subsequence trick'. \textcolor{\blue}{This result says 
that if an array of random variables $U_{n,k}$ is
monotone in $n$ and $k$, and $U_{n,k(n)}$, properly scaled,
converges in probability to a constant at rate $n^{-\eps}$,
one may be able to improve 
this to almost sure convergence.} 
\begin{lemm}[Subsequence trick]
\label{lemtrick}
Suppose $(U_{n,k},(n,k) \in \N \times \N)$ is an array of random 
variables on a common probability space such that $U_{n,k} $ is
nonincreasing in $n$ and nondecreasing in $k$, that is, 
$U_{n+1,k} \leq U_{n,k} \leq U_{n,k+1} $ almost surely,
for all $(n,k) \in \N \times \N$. Let $\beta \in [0,\infty), \eps >0, c  >0$,
and suppose $(k(n),n \in \N)$
is an $\N$-valued sequence such that 
 $k(n)/\log n \to \beta$ as $n \to \infty$. 

	(a) Suppose
$
\Pr[n U_{n, \lfloor (\beta + \eps) \log n \rfloor } > \log n] \leq 
c n^{-\eps},
$
	for all but finitely many $n$. 
Then $\Pr[ \limsup_{n \to \infty} n U_{n,k(n)} /\log n \leq 1] =1$.

	(b)
	Suppose  $\eps < \beta $ and
$
\Pr[n U_{n, \lceil (\beta - \eps) \log n \rceil } \leq \log n] \leq 
c n^{-\eps},
$
	for all but finitely many $n$. 
Then $\Pr[ \liminf_{n \to \infty} n U_{n,k(n)} /\log n \geq 1] =1$.
\end{lemm}
\begin{proof} (a)
For each $n $ set $k'(n) := \lfloor (\beta + \eps) \log n \rfloor$.
Pick $K \in \N$ with $K \eps >1$.
Then 
	by our assumption, we have for all large
 enough  $n$ that $\Pr[n^K U_{n^K, k'(n^K)} >
\log (n^K) ]\leq
c  n^{-K \eps}$,
which is summable in $n$. Therefore by the Borel-Cantelli lemma,
 there exists a random but almost surely finite $N$
such that for all $n \geq N$ we have 
$$
n^K U_{n^K,k'(n^K)} \leq  \log (n^K),
$$
and also  $k(m) \leq (\beta + \eps/2) \log m$ for all $m \geq N^K$,
and moreover $(\beta + \eps/2) \log ((n+1)^K) 
\leq (\beta + \eps)  \log (n^K) $ for 
all $n \geq N$.
Now for $m \in \N$ with $m \geq N^K$,
 choose $n \in \N$ such that $n^K \leq m < (n+1)^K$. Then
$$
k(m) \leq  (\beta +\eps/2)\log ((n+1)^K) 
\leq  (\beta +\eps)\log (n^K),
$$
so $k(m) \leq k'(n^K)$.
Since $U_{m,k}$ is nonincreasing in $m$ and nondecreasing in $k$,
\bean
m U_{m,k(m)}/\log m \leq (n+1)^K U_{n^K,k'(n^K)}/\log (n^K) 
\leq (1+ n^{-1})^K , 
\eean
which gives us the result asserted.

	(b) The proof is similar to that of (a), and is omitted.
	\qed
\end{proof}


\subsection{{\bf General lower and upper bounds}}
\label{subsecgenbds}

\textcolor{\blue}{
In this subsection 
we present 
asymptotic lower and upper
bounds on $R_{n,k(n)}$, not requiring 
any extra
assumptions on $A$ and $B$. In fact, $A$ here can be any metric space
endowed with a probability  measure $\mu$, and $B$ can
be any subset of $A$. The definition of $R_{n,k}$ at (\ref{Rnkdef})
carries over in an obvious way to this general setting.
}
%

\textcolor{\blue}{
Later, we shall derive the results
stated in Section \ref{secLLN}
by  applying the results of this subsection to the different 
regions within $A$ (namely interior, boundary, and
lower-dimensional faces).
}

\textcolor{\blue}{
Given $r >0, a>0$,   define the `packing number' $ \nu(B,r,a) $
 be the largest number $m$ such that there exists a collection of $m$ disjoint closed balls of radius $r$ centred on points of $B$,
each with $\mu$-measure at most $a$.
The proof of the following lemma implements,
for a general metric space,
the strategy outlined in Section \ref{secstrategy}. 
}

\begin{lemm}[General lower bound]
\textcolor{\blue}{
	\label{gammalem}
	Let $a >0, b \geq 0$. Suppose
	$\nu(B,r,a r^d) = \Omega (r^{-b})$ as $r \downarrow 0$.
	Then, almost surely, if $\beta = \infty$ then
	$\liminf_{n \to \infty} \left(n  R_{n,k(n)}^d/k(n) \right) \geq
	1/a$. If $\beta < \infty$ then
	$\liminf_{n \to \infty} \left(n  R_{n,k(n)}^d/\log n \right) 
	\geq a^{-1} \hH_\beta(b/d)$, almost surely.
}
\end{lemm}
\begin{proof}
	First suppose $\beta = \infty$.
	\textcolor{\blue}{Let $\nalpha \in (0,1/a )$.}
	Set $r_n := \left( \nalpha k(n)/n \right)^{1/d}$, $n \in \N$.
	By (\ref{kcond}), $r_n \to 0$ as $n \to \infty$.
	Then, given $n$ sufficiently large, we have
	$\nu(B,r_n,ar_n^d ) >0$ so  we can find $x \in B$ such
	that $ \mu(B(x,r_n)) \leq a r_n^d$, and hence
	$
	n \mu (B(x,r_n)) \leq a \nalpha k(n) .
	$
	If $k(n) \leq e^2 n \mu(B(x,r_n))$
	\textcolor{\blue}{(and hence $n \mu(B(x,r_n)) \geq e^{-2} k(n)$),
	then since $\X_n(B(x,r_n))$ is binomial with parameters
	$n$ and $\mu(B(x,r_n))$,}
	by Lemma \ref{lemChern}(a) we have that
	\bean
	\Pr[ R_{n,k(n)} \leq r_n] \leq
	\Pr[ \X_n(B(x,r_n) ) \geq k(n) ]
	& \leq & \textcolor{\blue}{  \exp \left( - n \mu(B(x,r_n)) H \left(\frac{k(n)}{n \mu( B(x,r_n))} \right) \right)}
	\\
	& \leq & \exp \left( - 
	e^{-2} k(n)
	H \left(
	(a \nalpha)^{-1}
	\right) \right),
	\eean
	while  if $k (n) > e^2 n \mu(B(x,r_n))$ then
	by Lemma \ref{lemChern}(c),
	$\Pr[ R_{n,k(n)} \leq r_n] \leq e^{-k(n)}$.
	Therefore
	$\Pr[ R_{n,k(n)} \leq r_n]$
	is summable in $n$ because we assume here that 
	$k(n)/\log n \to \infty$
	as $n \to \infty$. Thus by the Borel-Cantelli lemma,
	almost surely
	$R_{n,k(n)} > r_n$ for all but finitely many $n$, and hence
	$
	\liminf n  R_{n,k(n)}^d /k(n)  \geq \nalpha$. This gives the result
	for $\beta = \infty$. 

	Now suppose instead that $\beta < \infty$ and
	$b=0$, so that $\hH_\beta(b/d) = \beta$.
	Assume that   $\beta >0$ 
	(otherwise the result is trivial).
	Let $\beta' \in (0, \beta)$.  Let $\delta > 0$
with $\beta' < \beta- \delta$ and with
$
\beta'
H \left( \frac{ \beta - \delta}{\beta' } 
\right) > \delta.
$
	For $n \in \N$, set $r_n := (  
	(\beta' \log n)/ (a n) )^{1/d}$ and set
$k'(n)= \lceil (\beta - \delta) \log n \rceil$.
	By assumption $\nu(B,r_n,a r_n^d) = \Omega(1)$, so
	for all $n$ large enough, 
	we can (and do) choose $x_n \in B$ such that
$
	n \mu(B(x_{n},r_n)) \leq n a r_n^d = \beta' \log n. 
$
Then by a simple coupling, and Lemma \ref{lemChern}(a),
\bean
\Pr[
\X_{n} ( B(x_{n},r_n) ) \geq k'(n)
	]
	& \leq & 
	\textcolor{\blue}{
	\Pr \left[ \Bin\left(n,
	(\beta' \log n)/n)
	\right) \geq k'(n) \right]}
\\
	& \leq & \exp \left( - \left( 
	\beta' \log n
	\right) 
H \left( \frac{\beta- \delta }{\beta' } \right)  
\right) 
\leq n^{- \delta}.
\eean
Hence
$\Pr[ n R_{n,k'(n)}^d \leq (\beta'/a) \log n ]
= \Pr[ R_{n,k'(n)} \leq r_n] \leq n^{-\delta}$. Then
 by   the subsequence trick (Lemma \ref{lemtrick}(b)),
	we may deduce that
 $
	\liminf
( n  R_{n,k(n)}^d /\log n) 
\geq \beta'/a ,
$
	almost surely, which gives us the result for this case.

	Now suppose instead that $\beta < \infty$ and $b > 0$.
Let $\nalpha \in( 0,  a^{-1} \hH_\beta(b/d))$, so that
$\nalpha a  H(\beta/(\nalpha a)) < b/d$.
	Choose $\eps >0$ such that $(1+ \eps) \nalpha a H(\beta/(\nalpha a))
	< (b/d)-9 \eps$.

For each $n \in \N$ set $r_n= (\nalpha (\log n)/n)^{1/d}$.
Let $m_n := \nu(B,r_n,a r_n^d)$, and choose $x_{n,1},\ldots,$ $x_{n,m_n} \in B$
such that the balls $B(x_{n,1},r_n),\ldots,B(x_{n,m_n},r_n)$ are
pairwise disjoint and each have $\mu$-measure at most $ar_n^d$.
Set $\lambda(n):= n+ n^{3/4}$.
	For $1 \leq i \leq m_n$,
if $k(n) > 1$ then \textcolor{\blue}{by a simple coupling, 
and Lemma \ref{lemChern}(e)}, 
\bean
	\Pr[ \Po_{\lambda(n)}( B(x_{n,i},r_n) ) \leq k(n) -1]
	\geq \Pr[ Z_{\lambda(n) ar_n^d} \leq k(n)-1]
	\\
\geq 
 \left( \frac{e^{-1/(12(k(n)-1))}}{ \sqrt{2 \pi (k(n)-1)}} \right) 
\exp 
\left( - \lambda(n) a r_n^d  
	H \left( \frac{k(n)-1}{\lambda(n) a r_n^d } \right)  
\right).
\eean
	\textcolor{\blue}{
		Now $\lambda(n)r_n^d/\log n \to \nalpha$ so by (\ref{kcond}),
	$(k(n)-1)/(\lambda(n) a r_n^d) \to \beta/(\nalpha a)$
	as $n \to \infty$. Thus by the continuity of
	$H(\cdot)$,
	provided $n$ is large enough and $k(n) > 1$,
	for $1 \leq i \leq m_{n}$,}
\bea
\Pr[\Po_{\lambda(n)}(B(x_{n,i},r_n))  \leq  k(n)-1]  
~~~~~~~~~~~~~~~~~~~~~~~~~~~  
\nonumber \\
\geq 
 \left( \frac{e^{-1/12} } {\sqrt{2 \pi (\beta +1 ) \log n }} \right)
  \exp 
\left( - (1+ \eps)  a \nalpha   
H \left( \frac{\beta}{a \nalpha   } \right) \log n  
\right) .
\label{0319b5}
\eea
If $k(n)=1$ for infinitely many $n$, then $\beta =0$ and (\ref{0319b5})
still holds for large enough $n$.

By  (\ref{0319b5}) and our choice of $\eps$,
  there is a constant $c >0 $ such that for  all
large enough $n$ and all $i \in \{1,\ldots,m_n\}$ we have
\bean
\Pr[\Po_{\lambda(n)}(B(x_{n,i},r_n))  \leq  k(n)-1]  \geq c (\log n)^{-1/2}
	n^{   9 \eps - b/d  } \geq n^{8\eps -b/d}. 
\eean
	Hence, setting $E_n:= \cap_{i=1}^{m_n}
	\{
\Po_{\lambda(n)}(B(x_{n,i},r_n))  \geq  k(n)
		\}$,
	for all large enough $n$ we have
$$
	\Pr[E_n] \leq ( 1- n^{8\eps-b/d})^{m_n} \leq \exp( - m_n
	n^{8\eps -b/d} 
).
$$ 
By assumption $m_n = \nu(B,r_n , a r_n^d) = \Omega (r_n^{-b})$ so that 
for large enough $n$ we have $m_n \geq n^{(b/d) - \eps}$,
and therefore $\Pr[E_n]$ is
 is summable in $n$.

 \textcolor{\blue}{
	 By Lemma \ref{lemChern}(d), 
and Taylor expansion of $H(x)$ about $x=1$ 
(see the print version of  \cite[Lemma 1.4]{RGG} for details; there may be a typo in the electronic
version),}
for all $n $ large enough
$
\Pr[Z_{\lambda(n)} < n] \leq
 \exp( - \frac19  n^{1/2}), 
$
which is summable in $n$. Since $R_{m,k}$ is nonincreasing in $m$,
by the union bound
$$
\Pr[R_{n,k(n)} \leq r_n] \leq \Pr[R_{Z_{\lambda(n)},k(n)} \leq r_n ]
+ \Pr[ Z_{\lambda(n) } < n ]  \leq 
 \Pr[ E_n]+
\Pr[ Z_{\lambda(n)} < n],
$$
which is summable in $n$ by the preceding estimates. Therefore
by the Borel-Cantelli lemma,
\bean
 \Pr[  \liminf ( n  R_{n,k(n)}^d /\log n) \geq \nalpha  ] =1,
~~~~~ \nalpha < a^{-1} \hH_\beta(b/d),
\eean
so the result follows for this case too.
\qed
\end{proof}

\textcolor{\blue}{
Given  $r>0$, and $D \subset A$,
define the `covering number' 
\bea
\kappa(D,r): = \min \{m  \in \N: \exists x_1,\ldots,x_m \in D
~{\rm with} ~ D \subset \cup_{i=1}^m B(x_i,r) 
\}.
\label{covnumdef}
\eea
We need a complementary upper  bound to go with
the preceding asymptotic lower bound on $R_{n,k(n)}$. 
For this,  we shall require a condition on the `covering number'
that is roughly dual to the condition on `packing number'
 used in Lemma \ref{gammalem}.
Also, instead of stating the lemma in terms  
of $R_{n,k}$ directly, it is more convenient
to state it in terms of 
the `fully covered' region $F_{n,k,r}$, 
defined for 
$n,k \in \N$ and $r>0$ by
}
\bea
F_{n,k,r}:= \{x \in A: \X_n(B(x,r)) \geq k\}.
\label{F3def}
\eea
\textcolor{\blue}{
 We can characterise the event
 $\{\tR_{n,k} \le r\} $ in terms of the set $F_{n,k,r}$
 as follows:
\bea
\tR_{n,k } \leq r {\rm~~~ if ~and~only~if~~~}
(B \cap A^{(r)} )
\subset F_{n,k,r}.
\label{Fnequiv}
\eea
Indeed, the `if' part of this statement is clear from (\ref{eqmaxspac}).  
For the `only if' part, note that  if there exists $x \in (B \cap A^{(r)}) 
\setminus F_{n,k,r}$, then there exists $s >r$ with $x \in B \cap 
A^{(s)} \setminus F_{n,k,s}$. Then for all $s' <s$ we have
$x \in B \cap A^{(s')} \setminus F_{n,k,s}$, and therefore
$\tR_{n,k} \geq s > r$.}

\begin{lemm}[General upper bound]
\label{lemmeta}
\textcolor{\blue}{
Suppose $r_0, a, b \in (0,\infty)$,  and
 a family of sets  $A_r \subset A,$ defined for $0 < r < r_0$, 
	are
	 such that  for all $r \in (0,r_0)$, $x \in A_r$ and
 $s \in (0,r)$ we have $\mu(B(x,s)) \geq a s^d$, and
moreover $\kappa(A_r,r) = O(r^{-b})$ as $r \downarrow 0$.
}

\textcolor{\blue}{
If $\beta = \infty$ then let $\nalpha > 1/a$
and set $r_n = (\nalpha k(n)/n)^{1/d}$, $ n \in \N$.
Then with probability one,  $A_{r_n} \subset F_{n,k(n),r_n}$ 
for all large enough $n$.
}

\textcolor{\blue}{
If $\beta < \infty$, let $\nalpha >a^{-1}
\hH_\beta(b/d)$ and set $r_n = (\nalpha (\log n)/n)^{1/d}$.
 Then there exists $\eps >0$ such that
$
\Pr[\{A_{r_n} \subset F_{n,\lfloor (\beta + \eps) \log n\rfloor,r_n} \}^c ]
 = O(n^{-\eps})
$
as $n \to \infty$.
}
\end{lemm}
\begin{proof}
	Let $\eps \in (0,1)$; if
$\beta = \infty$, assume  $a (1-\eps)^d \nalpha > 1+ \eps$.
If $\beta < \infty$, assume
$$
a(1-\eps)^d \nalpha H((\beta+ \eps)/(a (1-\eps)^d \nalpha))
> ( b/d ) + \eps.
$$
 This can be achieved because
$a \nalpha H (\beta /(a \nalpha)) > b/d$ in this case.
Set $m_n = \kappa(A_{r_n},\eps r_n)$. Then $m_n= O(r_n^{-b})
= O(n^{b/d})$ (in either case).
Let $x_{n,1},\ldots,x_{n,m_n} \in A_{r_n}$ with $A_{r_n}
\subset \cup_{i=1}^{m_n} B(x_{n,i}, \eps r_n)$. Then
for $1 \leq i \leq m_n$, 
if $\X_n(B(x_{n,i}, (1-\eps)r_n) \geq k(n)$ then 
$B(x_{n,i},\eps r_n) \subset F_{n,k(n),r_n}$.
Therefore
\bea
	\Pr[ \{A_{r_n} \subset F_{n,k(n),r_n} \}^c ]
 \leq 
\Pr[ \cup_{i=1}^{m_n} \{ \X_n(B(x_{n,i},(1-\eps)r_n)  
< k(n) \}]. 
\label{210920a}
\eea

Suppose $\beta = \infty$. Then for $1 \leq i \leq m_n$, 
$$
n \mu(B(x_{n,i},(1-\eps)r_n))
\geq n a ((1-\eps)r_n)^d
\geq (1+ \eps) k(n), 
$$
so that by (\ref{210920a}), the union bound, and Lemma \ref{lemChern}(b),
\bean
	\Pr[ \{A_{r_n} \subset F_{n,k(n),r_n} \}^c ]
		& \leq & m_n \Pr[\Bin(n, (1+\eps) k(n)/n) <k(n) ] 
	\\
	& \leq & m_n \exp ( - (1+\eps) k(n) H((1+\eps)^{-1}) ).
\eean
This is summable in $n$, 
since $m_n = O(n^{b/d})$ and $k(n)/\log n \to \infty$.
Therefore by the first Borel-Cantelli lemma,
we obtain the desired conclusion for this case.

Now suppose instead that $\beta < \infty$. Then
$$
n \mu(B(x_{n,i},(1-\eps)r_n))
\geq n a ((1-\eps)r_n)^d
\geq  a (1- \eps)^d  \nalpha \log n . 
$$
Therefore setting
$k'(n)  := \lfloor (\beta + \eps)  \log n \rfloor$, 
	by (\ref{210920a})
we have
\bean
	\Pr[ \{A_{r_n} \subset F_{n,k'(n),r_n} \}^c ]
& \leq & m_n \exp \left( -a (1-\eps)^d \nalpha   H
\left( \frac{(\beta +\eps) }{a (1- \eps)^d \nalpha} \right) 
\log n
\right)
\\
& \leq & m_n n^{-(b/d) -\eps},
\eean
which yields the desired conclusion for this case.
\qed
\end{proof}

\subsection{{\bf Proof of Proposition \ref{thm1}}} 
\label{subsecpfprop1}

Throughout this subsection we assume
that either:
(i)  $B$ is compact and Riemann measurable with $\mu(B)>0$ and
 $B \subset A^o$, 
 or
 (ii) $B=A$. \textcolor{\blue}{We do not (yet) assume $f$ is continuous on $A$.}
 Recall from (\ref{f0def}) that
 $f_0: = \essinf_{x \in B} f(x)$.


 \textcolor{\blue}{
 We shall prove Proposition \ref{thm1} by applying 
   (\ref{Fnequiv}) and
  Lemma 
  \ref{lemmeta} 
   to derive an
 upper bound on $\tR_{n,k(n)}$ (Lemma \ref{lemlimsup}), and Lemma
  \ref{gammalem} to derive a lower bound (Lemma \ref{lemliminf}).
 For the lower bound we also require the following lemma (recall
 that $\nu(B,r,a)$ was defined just before Lemma \ref{gammalem}).}

\begin{lemm}
\label{packlem}
Let $\alpha > f_0$.
	Then $\liminf_{r \downarrow 0} r^d \nu(B,r, \alpha \theta_d r^d) > 0$.
\end{lemm}
\begin{proof}
	By the measurability assumption,
$\essinf_{B^{(\eps)}} (f) \downarrow f_0$ as $\eps \downarrow 0$. 
Therefore we can and do
	choose $\delta >0$ with $\mu(B^{(\delta)} ) >0$ and with
	$\essinf_{B^{(\delta)}}(f) < \alpha$.
	For $ r >0$,
	let $\sigma'(r)$ be the maximum number of disjoint closed balls $B_i$
	of radius $r$
	that can be found such that $0 < \mu( B_i \cap B^{(\delta)} ) \leq
	\alpha \theta_d r^d$. Then by applying \cite[Lemma 5.2]{RGG},
	taking the measure $F$ there to be the restriction of $\mu$ to 
	$B^{(\delta)}$,
	we have
	that $\liminf_{r \downarrow 0} (r^d \sigma'(r)) >0$.  
	Each of the balls $B_i$ satisfies $B_i \cap B^{(\delta)} 
	\neq \emptyset $, so if $r< \delta$ the centre of $B_i$ lies in
	$B$.    Hence $\sigma'(r) \leq \nu(B,r,\alpha \theta_d r^d)$ 
	for $r < \delta$, and the result follows.
	\qed
\end{proof}

\begin{lemm} 
\label{lemliminf}
	It is the case that
\bea
\Pr[ \liminf(n \theta_d \tR_{n,k(n)}^d/k(n) ) \geq 1/f_0]=1
	~~~~~{\rm if} ~\beta = \infty;
\label{0328a}
\\
 \Pr[  \liminf 
	( n \theta_d \tR_{n,k(n)}^d /\log n) \geq \hH_\beta(1)/f_0  ] =1
	~~~~~{\rm if} ~\beta < \infty.
\label{liminf1}
\eea
\end{lemm}
\begin{proof}
	Let $\alpha' > \alpha > f_0$.
	Set $r_n := (k(n)/(n \theta_d \alpha'))^{1/d}$ if $\beta = \infty$, and
	set
	\linebreak
	$r_n:= (\hH_\beta(1) (\log n)/(n \theta_d \alpha'))^{1/d}$ if $\beta < \infty$. 

	Assume for now that $B$ is compact with
	$B \subset A^o$, so that there exists $\delta >0$
	such that $B \subset A^{(\delta)}$.
	Then, even if $f$ is not continuous on $A$,
	we can find $x_0 \in B$ with $f(x_0) < \alpha$,
	such that $x_0$ is a Lebesgue point of $f$.
	Then for all small enough $r>0$ we have $\mu(B(x_0,r)) < 
	\alpha \theta_d r^d$,
	so that $\nu(B,r,\alpha \theta r^d) = \Omega(1)$ as $r
	\downarrow 0$.

	If $\beta = \infty$,
	then by Lemma \ref{gammalem} (taking $b=0$),
	 $\liminf_{n \to \infty} nR_{n,k(n)}^d/k(n)
	\geq (\theta_d \alpha)^{-1}$, almost surely. Hence for all large
	enough $n$ we have
	$R_{n,k(n)} > r_n$;
	  provided  $n$ is also large enough  so that $r_n < \delta$
	   we also have $\tR_{n,k(n)} > r_n$, and (\ref{0328a})
	follows.

	Suppose instead that $\beta < \infty$.
	Then by Lemma \ref{packlem},
	$\nu(B,r, \alpha \theta_d r^{d}) = \Omega(r^{-d})$ as
	$r \downarrow 0$.
	Hence by Lemma \ref{gammalem}, almost surely
	$\liminf_{n \to \infty}
	\left( n R_{n,k(n)}^d /\log n \right) \geq (\alpha \theta_d)^{-1} 
	\hH_\beta(1)$, and hence for large enough $n$ we have
	$R_{n,k(n)} >  r_n $ 
	and also 
	$\tR_{n,k(n)} >  r_n $, 
	which
	yields (\ref{liminf1}). 

	Finally, suppose instead that $B=A$.  Then
	by using e.g. \cite[Lemma 11.12]{RGG}
	we can find compact, Riemann measurable $B' \subset A^o$ with
	$\mu(B') >0$ and $\essinf_{x \in B'} f(x) < \alpha$.
	Define $S_{n,k}$ to be the smallest $r \geq 0$ such that every
	point in $B'$ is covered at least $k$ times by balls of
	radius $r$ centred on points of $\X_n$. 
	By the argument already given we have
	almost surely 
	for all large enough $n$
	\textcolor{\blue}{ 
	that $S_{n,k(n)} > r_n $
	and also $B' \subset A^{(r_n)}$. 
	 For such $n$, there exists $x \in B' \subset B 
	\cap A^{(r_n)}$ with 
	$\X_n(B(x,r_n)) < k(n)$,
	and hence by (\ref{Fnequiv}),}
	$\tR_{n,k(n)} > r_n$,
%
	which gives us (\ref{0328a}) and
	(\ref{liminf1}) in this case too.
	\qed
\end{proof}

\textcolor{\blue}{Now and for the rest of this subsection, we do
assume in case (i) (with $B \subset A^o$) that
$f$ is continuous on $A$.}

\begin{lemm}
\label{lemlimsup}
Suppose  that
	$f_0 >0$.
	Then almost surely
\bea
  \limsup ( n \theta_d \tR_{n,k(n)}^d /k(n)) \leq 1/f_0, ~~~~~ 
	{\rm if~} \beta = \infty;
	\label{eqsubseq3}
\\
	\limsup ( n \theta_d \tR_{n,k(n)}^d /\log n) \leq \hH_\beta(1)/f_0 , ~~~~~
	{\rm if~}  \beta < \infty.
\label{eqsubseq2}
	\eea 

\end{lemm}
\begin{proof}
	\textcolor{\blue}{We shall apply Lemma \ref{lemmeta},
	here taking $A_r = B \cap A^{(r)}$.} 
To start, we claim that
\bea
\liminf_{r \downarrow 0} \inf_{x \in B \cap A^{(r)} , s \in (0,r)} 
	\left(
	\frac{\mu(B(x,s))}{\theta_d s^d } \right)
	\geq f_0.
\label{ctyclaim}
\eea 
This follows from the definition (\ref{f0def}) of $f_0$
	when $B=A$. In the other case (with $B \subset A^o$) it
follows from (\ref{f0def}) and the assumed continuity of $f$ on $A$.

Suppose $f_0 < \infty$ and let $\delta \in (0,f_0)$.
It is easy to see that $\kappa(B \cap A^{(r)},r) = O(r^{-d})$
as $r \downarrow 0$. Together with (\ref{ctyclaim}), this 
shows that  the hypotheses of Lemma \ref{lemmeta} (taking 
$A_r = B \cap A^{(r)}$)
apply with $a= \theta_d(f_0- \delta)$ and $b= d$. 
	Hence 
	\textcolor{\blue}{by (\ref{Fnequiv}) and Lemma \ref{lemmeta},}
	if $\beta = \infty$ then for  any
 $ u > (\theta_d(f_0-\delta))^{-1}$,
we have almost surely for large enough $n$ that 
$\tR_{n,k(n)} \leq (u k(n)/n)^{1/d}$,
and (\ref{eqsubseq3}) follows.

If $\beta < \infty$, then 
	\textcolor{\blue}{by (\ref{Fnequiv}) and Lemma \ref{lemmeta},} given
$u > \hH_\beta(1)/
(\theta_d (f_0-\delta))
$,
there exists $\eps >0$ such that, setting $k'(n):=
\lfloor (\beta + \eps) \log n \rfloor$ and
	$r_n:= (u (\log n)/n)^{1/d}$, 
	we have
$$
\Pr[n \tR_{n, k'(n)}^d > u \log n]
=
\Pr[ \tR_{n,k'(n)}
 > r_n
	]
	= \Pr[\{B \cap A^{(r_n)} \subset F_{n, k'(n),r_n}\}^c]
= O(  n^{-\eps}).
$$
 Therefore by Lemma \ref{lemtrick},
which is applicable since $\tR_{n,k}^d/u$ is nonincreasing
in $n$ and nondecreasing in $k$, we obtain that
$	 \limsup ( n  \tR_{n,k(n)}^d /\log n) \leq u$, 
	almost surely.
Since $u > \hH_\beta(1)/(\theta_d ( f_0 - \delta))$
and $\delta \in (0,f_0)$ are 
 arbitrary, we therefore obtain (\ref{eqsubseq2}).
	\qed
\end{proof}

\begin{proof}[Proof of Proposition \ref{thm1}]
	Under either hypothesis ((i) or (ii)),
it is immediate from Lemmas  \ref{lemliminf} and
 \ref{lemlimsup} that (\ref{0617a}) holds if $\beta = \infty$ and
(\ref{0315a}) holds if $\beta < \infty$.

It follows that almost surely
 $\tR_{n,k(n)} \to 0$ as $n \to \infty$, 
and therefore if we are in Case (i) (with $B \subset
	A^o$) we have $\tR_{n,k(n)} = R_{n,k(n)}$ for all
large enough $n$.  Therefore in this case (\ref{0617a}) (if $\beta =\infty$)
or (\ref{0315a}) (if $\beta < \infty)$ still holds with $\tR_{n,k(n)}$
replaced by $R_{n,k(n)}$.
\end{proof}

\subsection{{\bf Proof of  Theorem \ref{thm2}}}
\label{secproofstrong}

In this section and again later on,
we shall use certain results from \cite{RGG}, which
rely on an alternative characterization of $A$ having
a $C^2$ boundary, given in the next lemma.
Recall that $S \subset \R^d$ is called a $(d-1)$-dimensional 
{\em submanifold} of $\R^d$ 
if there exists a collection of pairs $(U_i,\phi_i)$, where
$\{U_i\}$ is a collection
of open sets in $\R^d$ whose union contains $S$, and
$\phi_i$ is a $C^2$ diffeomorphism of $U_i$ onto an open
set in $\R^d$ with the property that
$\phi_i(U_i \cap S) = \phi_i(U_i) \cap (\R^{d-1} \times \{0\})$.
The pairs $(U_i,\phi_i)$ are called {\em charts}.
We shall sometimes also refer to the sets $U_i$ as charts here.

\begin{lemm}
	\label{lemsubmfld}
	Suppose  $A \subset \R^d$  has a $C^2$ boundary. Then
 $\partial A$ is a $(d-1)$-dimensional $C^2$ submanifold of $\R^d$.
\end{lemm}

\begin{proof}
	Let $x \in \partial A$. Let $U$ be an  open neighbourhood of
	$x$, $V \subset \R^{d-1} $ 
	an open set, and $\rho$ a rotation on $\R^d$ about $x$ 
	such that $ \rho( \partial A \cap U)  = \{(w,f(w)): w \in V\}$,
	and moreover $\rho(U) \subset V \times \R$.
	Then for $(w,z) \in U$ (with $w \in V $ and $z \in \R$),
	take $\psi(w,z) = (w,z-f(w))$. Then $\psi \circ \rho$
	is a $C^2$ diffeomorphism from $U$ to $\psi \circ \rho(U)$,
	with the property that $\psi \circ \rho(U \cap \partial A) 
	=  \psi \circ \rho (U) \cap (\R^{d-1} \times \{0\})$,
	as required.
	\qed
\end{proof}
\begin{remk}
	{\rm
The converse to Lemma \ref{lemsubmfld}  also holds:
if $\partial A$ is a $(d-1)$-dimensional submanifold of $\R^d$
then $A$ has a $C^2$ boundary in the sense that we have defined it.
The proof of this this implication is more involved, and not needed in
	the sequel, so  we omit the argument.}
\end{remk}


We shall use the following lemma  here and again
later on.

\begin{lemm}
	\label{lemMyBk}
	Suppose $A \subset \R^d$ is compact, and has
	$C^2$ boundary.
	Given $\eps >0$,
	there exists $\delta >0$ such that for all $x \in A$
	and $s \in (0,\delta)$, we have $|B(x,s) \cap A|
	> (1- \eps) \theta_d s^d/2$.
\end{lemm}
\begin{proof} 
	Immediate from applying first
	Lemma \ref{lemsubmfld}, and then
 \cite[Lemma 5.7]{RGG}.
	\qed
	\end{proof}

	\textcolor{\blue}{
	Recall that $F_{n,k,r}$ was defined at (\ref{F3def}).
	}
We introduce 
a new variable
$R_{n,k,1}$, which is the smallest radius $r$ of balls required to
cover $k$ times the
boundary region $A \setminus A^{(r)} $:
\bea
R_{n,k,1} := \inf \{r >0 : 
 A \setminus A^{(r)} \subset F_{n,k,r} 
\}, ~~~~n,k \in \N.
\label{Rnk1def}
\eea
 Loosely speaking, the 1 in the subscript refers to the fact
that this boundary region is in some sense $(d-1)$-dimensional. 
For all $n,k$, we claim that
\bea
R_{n,k} = \max( \tR_{n,k},R_{n,k,1}), ~~~~{\rm  if} ~ B=A.
\label{mineq}
\label{mineq2}
\eea
\textcolor{\blue}{Indeed, if $r > R_{n,k}$ then $A \subset F_{n,k,r}$
so that $A^{(r)} \subset F_{n,k,r}$ and $A \setminus A^{(r)}
\subset F_{n,k,r}$, and hence $r \geq \max(\tR_{n,k},R_{n,k,1})$;
hence, $R_{n,k} \geq \max(\tR_{n,k},R_{n,k,1})$.
For an inequality the other way,
suppose $r > \max(\tR_{n,k},R_{n,k,1})$; then
by (\ref{eqmaxspac}),
there exists $r' <r$ with $A^{(r')} \subset F_{n,k,r'}$,
and hence also $A^{(r)} \subset F_{n,k,r}$. 
Moreover by  (\ref{Rnk1def}) there exists $s < r$ with
$A \setminus A^{(s)} \subset F_{n,k,s}$. 
Now suppose $x \in A^{(s)} \setminus A^{(r)}$. 
Let $z \in \partial A$ with $\|z-x\|= \dist(x,\partial A) \in (s,r]$.
Let $y \in [x,z]$ with $\|y-z\|= s$.
Then 
$y \in A \setminus A^{(s)}$,
so that $y \in F_{n,k,s}$, and also $\|x-y\| \leq r-s$.
This implies that $x \in F_{n,k,r}$.
Therefore $A^{(s)} \setminus A^{(r)} \subset F_{n,k,r}$. Combined
with the earlier set inclusions this shows that $A \subset F_{n,k,r}$
and hence 
$R_{n,k} \leq r$. Thus $R_{n,k} \leq \min(\tR_{n,k},R_{n,k,1})$,
and hence (\ref{mineq}) as claimed.
}

Recall that 
we are assuming that
$k(n)/\log n \to \beta \in [0,\infty]$ and $k(n) /n \to 0$,
 as $n \to \infty$,
and
$f_1 := \inf_{\partial A} f$. \textcolor{\blue}{We shall derive
Theorem \ref{thm2} 
using the next two lemmas.}

\begin{lemm}
\label{lemliminfb}
	Suppose the assumptions of Theorem \ref{thm2} apply.  Then
\bea
\Pr[\liminf_{n \to \infty}
\left( n \theta_d R_{n,k(n)}^d /k( n) \right)
	\geq
 2/ f_1  ] =1, ~~~~{\rm if}~
 \beta = \infty;
\label{liminfeq0}
	\\
\Pr[\liminf_{n \to \infty}
\left( n \theta_d R_{n,k(n)}^d /\log n \right) 
\geq
	2 \hH_\beta(1-1/d)/ f_1 ] =1, ~~~~~ {\rm if} ~
\beta < \infty .
\label{liminfeq}
\eea
\end{lemm}

\begin{proof}
Let $\eps >0$.
By \cite[Lemma 5.8]{RGG}, 
for each $r >0 $  we can and do take $\ell_r \in \N \cup \{0\}$ and
	points $y_{r,1},\ldots,y_{r,\ell_r} \in \partial A$ 
such that the balls $B(y_{r,i},r), 1 \leq i \leq \ell_r$, are disjoint
and each satisfy $\mu(B(y_{r,i}, r)) \leq (f_1+\eps) \theta_d r^d/2$,
	with $\liminf_{r \downarrow 0} (r^{d-1} \ell_r) >0$.
	In other words, $ \liminf_{r \downarrow 0} 
	r^{d-1} \nu(B,r,(f_1+\eps)\theta_d r^d/2) >0$.

	Hence, if $\beta = \infty$ then by
	Lemma \ref{gammalem} we have that $\liminf_{n \to \infty}
	\left(  n R_{n,k(n)}^d/k(n) \right) \geq 2/(\theta_d (f_1+ \eps))$,
	almost surely, and this yields (\ref{liminfeq0}).

	Now suppose $\beta < \infty$; also we are assuming $d \geq 2$. By
	taking $a= (f_1+ \eps) \theta_d/2$ in Lemma \ref{gammalem},
	we obtain that, almost surely,
	$$
	\liminf_{n \to \infty} \left( nR_{n,k(n)}^d / \log n \right)
	\geq \left( 
	2/((f_1+ \eps) \theta_d)
	\right)
	\hH_\beta \left( 1 - 1/d \right),
	$$
	and hence (\ref{liminfeq}).
	\qed
\end{proof}

\begin{lemm}
\label{lemlimsupb}
	Under the assumptions of Theorem \ref{thm2},
\bea
\Pr[\limsup_{n \to \infty}
\left( n \theta_d R_{n,k(n),1}^d /k( n) \right) \leq 
 2/ f_1  ] =1, ~~~~{\rm if}~
 \beta = \infty;
\label{limsupeq0}
\\
\Pr[\limsup_{n \to \infty}
\left( n \theta_d R_{n,k(n),1}^d /\log n \right) \leq 
	2 \hH_\beta(1- 1/d) / f_1 ] =1, ~~~~~ {\rm if} ~ 
\beta < \infty .
\label{limsupeq}
\eea
\end{lemm}
\begin{proof}
We shall apply Lemma \ref{lemmeta}, here taking $A_r = A \setminus A^{(r)}$. 
Observe that by (\ref{Rnk1def}),
	event $\{ A \setminus A^{(r)} \subset F_{n,k,r} \}$
	implies
 $R_{n,k,1} \leq r$,
	 for all  $r>0$, $n, k \in \N$. 

We claim that
\bea
	\kappa(A \setminus A^{(r)},r) = O(r^{1-d}) ~~{\rm as} ~~ r \downarrow
	0.
\label{0320b2}
\eea
	To see this,  let $r>0$, 
	and  let $x_1,\ldots,x_m \in \partial A$ with 
	$\partial A \subset \cup_{i=1}^m B(x_i,r)$, and
	with $m = \kappa(\partial A,r)$.
	Then $A \setminus A^{(r)}  \subset \cup_{i=1}^{m}
	B(x_{i}, 2r)$. Setting $c := \kappa (B(o,4),1)$,
	we can cover each ball $B(x_i,2r)$ by $c$ balls of
	radius $r/2$, and therefore can cover $A \setminus A^{(r)}$
	by  $cm$ balls of radius $r/2$, denoted $B_1,\ldots, B_{cm}$
	say. Set $ {\cal I}:= \{i \in 
	\{1,\ldots,cm\}:
	A\setminus A^{(r)} \cap B_i \neq \emptyset \}$.
	For each $i \in {\cal I}$, 
	select a point $y_i \in A \setminus A^{(r)} \cap B_i$.
	Then
	$A \setminus A^{(r)} \subset \cup_{i \in{\cal I}} B(y_i,r) $,
	and hence $\kappa(A \setminus A^{(r)}) \leq c \kappa( \partial A,r)$.
	By \cite[Lemma 5.4]{RGG}, $\kappa(\partial A,r) = O(r^{1-d})$,
	and (\ref{0320b2}) follows.

	Let $\eps_1 \in (0,1)$.
Since we assume $f|_A$ is continuous at $x $ for all $x \in \partial A$,
there exists $\delta >0$ such that
$f(x) > (1 - \eps_1)f_1$ for all $x \in A$ distant less than $\delta$
from $\partial A$. Then, under the hypotheses of Theorem
\ref{thm2}, by Lemma \ref{lemMyBk},
	\textcolor{\blue}{there is a further constant $\delta' \in (0,\delta/2)$ such
	that for all 
	$r \in (0, \delta')$}
	and all $x \in A \setminus A^{(r)}$,
	$s \in (0,r]$ we have
$
\mu(B(x,s)) \geq 
(1- \eps_1  )^{2} 
f_1 (\theta_d/2) s^d.
$

	Therefore taking $A_r = A \setminus A^{(r)}$,
	the hypotheses of Lemma \ref{lemmeta} hold
	with $a = (1-\eps_1)^2 f_1 \theta_d/2$ and
	$b= d-1$.

	Thus if $\beta = \infty$, then taking $r_n = (u k(n)/n)^{1/d}$
	with $u > 2 (1-\eps_1)^{-2}/(f_1  \theta_d)$,
	by Lemma \ref{lemmeta}
	we have almost surely that for all $n$ large enough,
	$A \setminus A^{(r_n)} \subset F_{n,k(n),r_n}$ and
	hence $R_{n,k(n),1} \leq r_n$. That is,
	$n R_{n,k(n),1}^d/k(n) \leq u$, and (\ref{limsupeq0}) follows. 
%

	If $\beta < \infty$, take $u > (2 (1-\eps_1)^{-2}/(f_1 \theta_d))
	\hH_\beta(1-1/d)$.
	By Lemma \ref{lemmeta}, if we set $r_n= (u (\log n)/n)^{1/d}$,
	then there exist $\eps >0$ such that
	if we take $k'(n)= \lfloor (\beta + \eps) \log n \rfloor$, then
	$$
	\Pr[n R_{n,k'(n),1}^d >  u \log n]  =
	\Pr[ R_{n,k'(n),1} > r_n] \leq
	\Pr[ \{ (A \setminus A^{(r_n)} ) \subset F_{n,k'(n),r_n} \}^c]
	= O(n^{-\eps}).
	$$
	Also $R_{n,k,1}$ is nonincreasing in $n$
	and nondecreasing in $k$, so by Lemma \ref{lemtrick},
	we have almost surely that $\limsup_{n \to \infty}
	n R^d_{n,k(n),1} /\log n \leq u$, and hence
	(\ref{limsupeq}).
	\qed
\end{proof}

\begin{proof}[Proof of Theorem \ref{thm2}]
	\textcolor{\blue}{
	By (\ref{mineq}), Proposition \ref{thm1} and Lemma \ref{lemlimsupb},
	$$
	\limsup_{n \to \infty} \left(
	n \theta_d R_{n,k(n)}^d/k(n) \right) \leq \begin{cases}
		\max(1/f_0,2/f_1) & 
		{\rm if}~ \beta = \infty \\
		\max(\hH_\beta(1)/f_0, 2 \hH_\beta(1-1/d)/f_1) & 
		{\rm if}~ \beta < \infty, 
	\end{cases}
	$$
	almost surely.
Moreover
by (\ref{mineq}), Proposition \ref{thm1} and  Lemma \ref{lemliminfb}, 
	$$
	\liminf_{n \to \infty} \left(
	n \theta_d R_{n,k(n)}^d/k(n) \right) \geq \begin{cases}
		\max(1/f_0,2/f_1) & 
		{\rm if}~ \beta = \infty \\
		\max(\hH_\beta(1)/f_0, 2 \hH_\beta(1-1/d)/f_1) & 
		{\rm if}~ \beta < \infty, 
	\end{cases}
	$$
	almost surely, and the result follows.}
\end{proof}

\subsection{{\bf Polytopes: Proof of Theorem \ref{thmpolytope}}}
\label{secpfpolytope}

Throughout this subsection we assume, as in Theorem \ref{thmpolytope}, that
$A$ is a compact finite  polytope in $\R^d$,
and if $d \geq 4 $ then also $A$ is convex.
We also assume that  $B=A$, and $f|_A$ is continuous at $x$
for all $x \in \partial A$, and  (\ref{kcond})
holds for some $\beta \in [0,\infty]$.

Given any $x \in \R^d$ and nonempty $S \subset \R^d$
we set  $\dist(x,S):= \inf_{y \in S} \|x-y\|$.

\begin{lemm}
	\label{lemangle}
	Suppose $\varphi, \varphi'$ are faces of $A$ with $D(\varphi)>0$
	and $D(\varphi') = d-1$, and with
	$\varphi \setminus \varphi' \neq \emptyset$.
	Then $\varphi^o \cap \varphi' = \emptyset$, and
	$K(\varphi,\varphi') < \infty$, where we set
	\bea
	K(\varphi,\varphi') := \sup_{x \in \varphi^o}
	\left( \frac{\dist(x,\partial \varphi)}{\dist(x,\varphi')}\right).
	\label{Kphi2def}
	\eea
\end{lemm}

\begin{proof}
	If $\varphi \cap \varphi' = \emptyset$ then $K(\varphi,\varphi')
	< \infty$
	by an easy compactness argument, so
	assume $\varphi \cap \varphi' \neq \emptyset$.
Without loss of generality we may then assume $o \in \varphi \cap
\varphi'$.

	\textcolor{\blue}{
	If $d=3$, $A$ is convex and $D(\varphi)= D(\varphi')=2$,
	$D(\varphi \cap \varphi') = 1$ and
	moreover $\varphi, \varphi'$ are rectangular with angle
	$\alpha $ between them and $0 < \alpha < \pi$, then
	$K(\varphi,\varphi') = \sec \alpha$, as illustrated in 
	Figure 1. However to generalize to all $d$,
	and to non-convex $A$ when $d=2$ or $d=3$,
	takes some care.}

\begin{figure}[!h]
\label{figangle}
\center
\includegraphics[width=8cm]{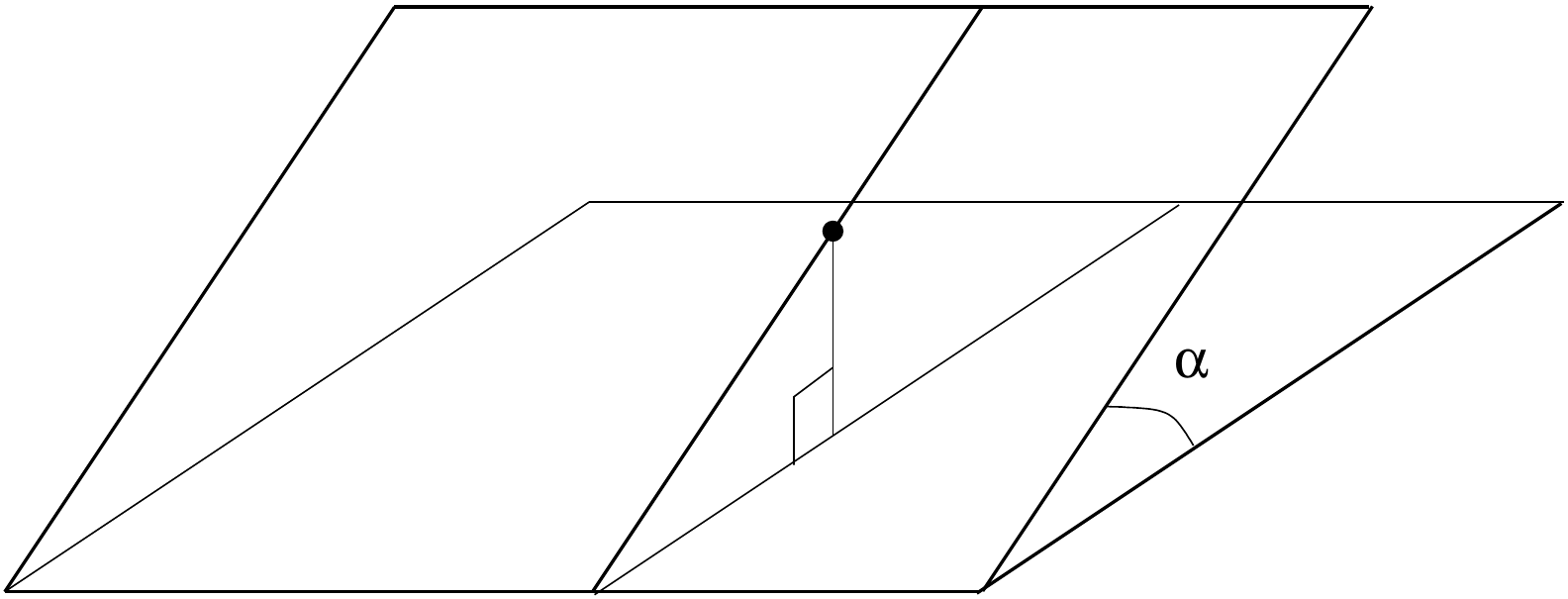}
	\caption{Illustration of Lemma \ref{lemangle} in  in $d=3$
	with $D(\varphi) = D(\varphi')= 2$. The dot represents a
	point in $\varphi^o$.}
\end{figure}

Let $\langle \varphi \rangle$, respectively
 $\langle \varphi' \rangle$,  be the linear subspace of $\R^d$ generated 
 by $\varphi$, respectively by $\varphi'$. Set $\psi := \langle \varphi
 \rangle \cap \langle \varphi' \rangle$.

	Assume first that $A$ is convex. Then $A \cap \langle \varphi' 
	\rangle = \varphi'$, and $\langle \varphi' \rangle$ is a supporting 
	hyperplane of $A$.

 We claim that $\varphi \cap \langle \varphi' \rangle \subset \partial
 \varphi$. Indeed, let $z \in \varphi \cap \langle \varphi'\rangle$ and
	$y \in \varphi \setminus  \varphi' $.
	Then 
	$y \in \varphi \setminus \langle \varphi' \rangle$,
	and for all $\eps >0$
	the vector $y + (1+ \eps) (z- y) $ lies in the affine hull
	of $\varphi$ but  not in $A$, since
	it is on the wrong side of the supporting hyperplane
	$\langle \varphi' \rangle$, and therefore not in $\varphi$.
	This shows that $z \in \partial \varphi$, and hence the claim.

	\textcolor{\blue}{Since
	$\varphi \cap \psi \subset
	\varphi \cap \langle \varphi' \rangle$,}
	it follows
	 from the preceding claim that
	$\varphi \cap  \psi  \subset
 \partial \varphi$.
 
 Now let $x \in \varphi^o$.
	Then
$x \in  \langle \varphi \rangle \setminus \psi$.
 Let $\pi_\psi(x)$ denote the point in $\psi$ closest to $x$, 
 and set $a:= \| x - \pi_\psi(x) \| = \dist(x, \psi)$. Then $a >0$. 

 Set $w:= a^{-1}(x - \pi_\psi(x))$.
 Then $\|w \| =1,$ $ w \in \langle \varphi \rangle$ and $w \perp \psi$
	(i.e., the Euclidean inner product of $w$ and $z$ 
	is zero for all $z \in \psi$),
 so $\dist(w,\langle \varphi' \rangle) \geq \delta$, where we set
 $$
 \delta := \inf \{ \dist(y,\langle \varphi' \rangle): y \in \langle \varphi \rangle,
 \| y \| =1, y \perp \psi \}.
 $$
	If $y \in \langle \varphi \rangle \setminus \psi$
 then $y \notin \langle \varphi' \rangle $ so 
 $\dist( y, \langle \varphi' \rangle ) >0$. Therefore $\delta$
 is the infimum of a continuous, strictly positive  function 
 defined on a non-empty compact set of vectors $y$,
	and hence $0 < \delta < \infty$.
	Thus for $x \in \varphi^o$, with $w,a$ as given above, we have
 \bea
 \dist(x, \varphi') \geq \dist(x,\langle \varphi' \rangle)
 = \dist(\pi_\psi(x) + a w, \langle \varphi' \rangle)
 \nonumber \\
 = \dist(aw, \langle \varphi' \rangle)
 \nonumber \\
 \geq \delta a 
 = \delta \dist(x,\psi).
 \label{1228b}
 \eea
 If $\pi_\psi(x) \notin \varphi$,
 then there is a point in
 $[x, \pi_\psi(x)] \cap \partial \varphi$,
 while 
 if $\pi_\psi(x) \in \varphi$, then $\pi_\psi(x) \in \partial \varphi$.
 Either way $\dist(x,\psi) \geq \dist(x, \partial \varphi)$,
 and hence by (\ref{1228b}),
 $\dist (x,\varphi') \geq \delta \dist(x, \partial \varphi)$.
 Therefore $K(\varphi,\varphi')\leq \delta^{-1} < \infty$ as required.


 Now we drop the assumption that $K$ is convex. Suppose
 first that the dimension of the face  $\varphi \cap \varphi'$ is  
 zero. Without loss of generality we may assume $\varphi \cap \varphi' = \{o\}$.
 Then there exists a neighbourhood $U$ of $o$, and cones $\cK,\cK'$, such
 that $\varphi \cap U = \cK \cap U$ and
 $\varphi' \cap U = \cK' \cap U$.
 
 Then $\cK \cap \cK' = \{o\}$; indeed, if $z \in \cK \cap \cK'$ then
 for small enough $\eps >0$ we have $\eps z \in \varphi \cap \varphi'$
 so  $z=o$. Hence, defining
 $$
 \delta' := \min \left( 1, \inf \{ \dist(z,\cK') : z \in \cK, \|z\| = 1\} 
 \right),
 $$
 we have by a compactness argument that $\delta' > 0$. 
 
 Choose $r >0$ such that $B(o,2r) \subset U$. Let $x \in
 B(o,r) \cap  \varphi \setminus \varphi'$, and let $y \in \varphi'$.
 Then $x \neq o$. If 
 $y \in   B(o,2r)$, then  $y \in \cK'$ so that
 \bean
 \|x-y\| \geq \dist (x, \cK') = \|x\| \dist (\|x\|^{-1}x,\cK')
 \\
 \geq \delta' \|x\| \geq \delta' \dist(x, \partial \varphi),
 \eean
 while if $y \notin  B(o,2r)$ then $\|x-y \| \geq r \geq \|x\|$.
 Therefore taking the infimum over all $y \in \varphi'$, we have
 that $\dist(x,\varphi') \geq \delta' \dist(x,\partial \varphi)$.

 Moreover, by a compactness argument 
 $\dist (x,\varphi')$ is bounded away from $0$ on $x \in \varphi
 \setminus B(o,r)$, while $\dist(x,\partial \varphi)$ is bounded away from
 infinity. Combining this with the preceding bound we obtain that
 $K(\varphi,\varphi') < \infty$ whenever $D(\varphi \cap \varphi')=0$.

 If $d=2$ then the preceding case with $D(\varphi \cap \varphi')=0$,
 is the only non-trivial case to consider so we have the result in
 this case.  

 Now suppose $d=3$.
 Then we also need to consider the case with
 $D(\varphi \cap \varphi')=1$. In this case $D(\varphi)= D(\varphi')=2$.
 Then we may write  $\varphi \cap \varphi' = \cup_{j=1}^k e_j$,
 where $e_1,\ldots,e_k$ are edges of $A$, all contained in a single 
 line (for example $k$ could be 2 if $A$ is a cube with a small
 polyhedral notch removed in a neighbourhood of the middle of one 
 of its edges). 
 Without loss of generality we assume that this line is the
 $x$-axis. Pick $j \in \{1,\ldots,k\}$ and assume
 without loss of generality 
 that the endpoints
 of the line segment $e_j$ are at 
 $o $ and at $(a,0,0)$ for some $a >0$, and that
 $ \varphi \subset \R^2 \times \{0\}$.

 Let $\alpha $ be the angle subtended by $A$ at the edge
 $e_j$.
 Then $0 \leq \alpha < 2 \pi $ with $\alpha \neq \pi$. 
 Pick $b>0$ such that none of
 the edges $e_j, j \in \{1,\ldots,k\}$, intersect
 with the open line segment from
 $(-b,0,0)$ to $o$.

 For $x \in \varphi^o$ with $x $ close to $e_j$
 we need to find a lower bound
 for $\dist(x,\varphi')/ \dist(x, \partial \varphi)$.
 We argue separately depending on whether $x$ lies in
 the `left' region $L:= (-\infty,0] \times \R^2$, the 
 `middle' region $M := (0,a) \times \R^2$, or the `right' region
 $R:= [a, \infty) \times \R^2$. 

 Let $\cK_0, \cK'_0$ be cones such that there is a neighbourhood
 $U_0$  of $o$ with $\varphi \cap U_0 = \cK_0 \cap U_0$ and $\varphi' \cap U_0
 = \cK'_0 \cap U_0$. For $x \in \cK_0 \cap L \cap \cK'_0$ we
 have for sufficiently small $\eps >0$ that $\eps x \in \varphi \cap \varphi' 
 \cap L \cap ((-b,0] \times \R^2) = \{o\}$, and
 hence $\cK_0 \cap L \cap \cK'_0 = \{o\}$.
 Then setting 
 $$
 \delta_j := \min \left(1,  \inf 
 \{\dist(w,\cK'_0): w \in \cK_0 \cap L, \|w\|=1\} \right),
 $$
 we have  $\delta_j >0$ by a compactness argument. Next,  pick
 $r_j >0$ such that $B(o,2r_j) \subset U_0$.  For $x \in L \cap \varphi 
 \cap B(o,r_j)$, if
  $y \in   \varphi' \cap B(o,2r_j)$
 we have
 \bean
 \|y - x\| \geq \dist(x,\cK'_0) \geq \|x\| \delta_j \geq \delta_j
 \dist(x, \partial \varphi),
 \eean
 while if $y \in \varphi' \setminus B(o,2r_j)$ then $\| y-x\|
 \geq r_j \geq \|x\|$, so that $\dist(x,\varphi') \geq  
 \delta_j \dist(x, \partial \varphi)$.

 Similarly, we can find $\eps_j >0$ and $s_j >0$
 such that if $x \in R \cap
 \varphi \cap B((a,0,0),s_j)$ we have $\dist(x,\varphi' ) \geq
 \eps_j \dist(x,\partial \varphi)$.

 Define $\langle \varphi \rangle , \langle \varphi' \rangle$
 as before.
 The planes $\langle \varphi \rangle$ and $\langle
 \varphi' \rangle$ are at an angle 
 $\alpha' := \min(\alpha, 2 \pi - \alpha)$ to each other. 
 Given  $x \in \varphi \cap M $, let $\pi(x)$ be the closest point
 in $\langle \varphi \rangle  \cap \langle \varphi' \rangle$ to $x$.
 Then, since $x \in M$ we have $\pi(x) \in \varphi \cap \varphi' \subset
 \partial \varphi$.
 Hence $\|x - \pi(x)\| \geq \dist(x,\partial \varphi)$.
 Then
 \bean
 \dist(x,\varphi') \geq \dist(x,\langle \varphi' \rangle)
 = \|x - \pi(x)\| \sin (\alpha') \geq \dist(x,\partial \varphi) \sin (\alpha'),
 \eean
 for all $x \in \varphi \cap  M$. Combined with the  preceding
 estimates for $x \in L$ and for $x \in R $, we have 
 \bean
 \inf \left\{ \frac{\dist(x,\varphi')}{\dist(x,\partial \varphi) }:
 x \in \varphi^o , \dist(x,e_j) \leq \min(r_j,s_j)
 \right\} > \min(\delta_j, \eps_j, \sin \alpha').
 \eean
 By a  compactness argument,
  for $x \in \varphi^o$ with $\dist(x,\varphi \cap \varphi')
 \geq \min(r_1,s_1, \ldots, r_k,s_k)$ the ratio
 $\dist(x, \varphi') /\dist(x,\partial \varphi)$
 is also bounded away from zero.
 This shows that
 $K(\varphi,\varphi')$ is finite in this case too.
 \qed
\end{proof}

Recall that we are assuming (\ref{kcond}).
Also, 
recall  that
for each face $\varphi$ of $A$
we denote the angular  volume of $A$ at $\varphi$ by
$\rho_{\varphi}$, and set $f_\varphi := \inf_{\varphi} f(\cdot)$.

\begin{lemm}
	\label{lemtopelb}
Let $\varphi$ be a face of $A$.
Then, almost surely:
\bea
	\liminf_{n \to \infty} \left( n  R^d_{n,k(n)}/ k(n) \right)
	\geq ( \rho_\varphi f_\varphi)^{-1}  ~~~ & {\rm if} ~ \beta = \infty; 
	\label{0704e3}
	\\
\liminf_{n \to \infty} \left( n  R^d_{n,k(n)}/ \log n \right)
	\geq (\rho_\varphi f_{\varphi} )^{-1} \hH_\beta(D(\varphi)/d)  ~~~ & {\rm if} ~ \beta < \infty 
	.
	\label{0704f3}
\eea
\end{lemm}
\begin{proof}
	Let $a > f_{\varphi}$.
	Take $x_0 \in \varphi$ 
	such that $f (x_0) <a$. If $D(\varphi) >0$, assume
	also that $x_0 \in \varphi^o$.
	By the assumed continuity of $f$ at $x_0$, for all
	small enough $r >0$ we have
	$\mu(B(x_0,r)) \leq
	a  \rho_\varphi r^d$, so that
	$\nu(B,r,a \rho_\varphi r^d) = \Omega(1)$ as
	$r \downarrow 0$.
	Hence by Lemma \ref{gammalem} (taking $b=0$),
if  $\beta = \infty$ then almost surely   
	$\liminf_{n \to \infty} n R_{n,k(n)}^d/k(n) \geq
	1/(a \rho_\varphi)$, and (\ref{0704e3}) follows.
	Also, if $\beta < \infty$ and 
	$D(\varphi) =0$, then by Lemma \ref{gammalem}
	(again with $b=0$),
	almost
	surely
	$\liminf_{n \to \infty} (n R_{n,k(n)}^d/\log n)
	\geq \hH_\beta(0)/(a \rho_\varphi)$, and hence
	(\ref{0704f3}) in this
	case.

	Now suppose $\beta < \infty$ and $D(\varphi)>0$. Take
	$\delta >0 $ such that $f(x) < a$ for all
 $x \in B(x_0, 2 \delta) \cap A$, and such that moreover
	$B(x_0,2 \delta) \cap A = B(x_0,2 \delta) \cap (x_0+
	\cK_\varphi)$ (the cone $\cK_\varphi$ was defined
	in Section \ref{secLLN}).
	Then for all $x \in B(x_0,\delta) \cap \varphi$
	and all $r \in (0,\delta)$,
	we have $\mu(B(x,r)) \leq  a \rho_\varphi   r^d$.

	There is a constant $c >0$ such that
	for small enough $r >0$ we can find at least $cr^{-D(\varphi)}$
	points $x_i \in
	B(x_0,\delta) \cap \varphi$ that are all at a
	distance more than $2 r$ from each other, and therefore
	 $\nu(B, r,a \rho_\varphi r^d ) 
	 = \Omega( r^{-D(\varphi)})$ 
	as $r \downarrow 0 $.
	Thus by
	Lemma \ref{gammalem} we have
	$$
	\liminf_{n \to \infty} \left( nR_{n,k(n)}^d/k(n) \right)
	\geq  ( a \rho_\varphi)^{-1} \hH_{\beta}(D(\varphi)/d),
	$$
	almost surely,
	and (\ref{0704f3}) follows.
	\qed
\end{proof}

We now define a sequence of positive constants $K_1,K_2,\ldots$ depending on $A$
as follows. With $K(\varphi,\varphi') $ defined at (\ref{Kphi2def}), set 
\bea
K_A : =
\max \{ K(\varphi,\varphi'): \varphi ,  \varphi' \in
\Phi(A), D(\varphi') = d-1, \varphi \setminus \varphi' \neq \emptyset \},
\label{KAdef}
\eea
which is finite by Lemma \ref{lemangle}.
Then for $j=1,2,\ldots$ set $K_j := j(K_A+1)^{j-1}$. Then
$K_1=1$ and
for each $j \geq 1 $ we have $K_{j+1} \geq (K_A+1)(K_j+1)$. 

For each face $\varphi$ of $A$ and each $r >0$,
define the sets $\varphi_{r} : = \cup_{x \in \varphi} B(x,r) \cap A$,
and also $(\partial \varphi)_{r} := \cup_{x \in \partial \varphi} 
B(x,r) \cap A$ (so if $D(\varphi)=0$ then 
 $(\partial \varphi)_{r} = \partial \varphi = \emptyset$).
Given also $r>0$,
define for each $n,k  \in \N$ 
the event $G_{n,k,r,\varphi}$ as follows:

If $D(\varphi) = d-j$ with $1 \leq j \leq d$, let  
$G_{n,k,r,\varphi} := \{
	(\varphi_{K_j r} \setminus (\partial \varphi)_{K_{j+1} r})
	\setminus F_{n,k,r} \neq \emptyset \}$,
	 the event that there exists 
$x \in \varphi_{K_j r} \setminus (\partial \varphi)_{K_{j+1} r}$
such that $\X_n ( B(x,r)) < k$. 

Let $R_{n,k,1}$
 be the smallest radius $r$ of balls required to
cover $k$ times the
boundary region $A \setminus A^{(r)} $,
as  defined at (\ref{Rnk1def}).

\begin{lemm}
	\label{Flem}
	Given $r>0$ and $n, k \in \N$,
	$ \{R_{n,k,1} > r\} \subset 
	 \cup_{\varphi \in \Phi(A)} G_{n,k,r,\varphi}. 
	$
\end{lemm}
\begin{proof}
	Suppose $R_{n,k,1} > r$. Then we can and do choose a point
	 $x \in (A \setminus A^{(r)} ) \setminus F_{n,k,r}$,
	and a face
	$\varphi_1 \in \Phi(A) $ with $D(\varphi_1)
	= d-1$, and with
	$x \in (\varphi_1)_{r} = (\varphi_1)_{K_1 r}$. 

	If $x \notin (\partial \varphi_1)_{K_2 r}$ then
		$G_{n,k,r,\varphi_1}$
	occurs. Otherwise, we can and do choose $\varphi_2 \in \Phi(A)$ with
	$D(\varphi_2) = d-2$ and $x \in (\varphi_2)_{K_2 r}$. 

	If $x \notin (\partial \varphi_2)_{K_3 r}$ then 
	\textcolor{\blue}{
	$G_{n,k,r,\varphi_2}$
	occurs.} Otherwise, we can choose $\varphi_3 \in \Phi(A)$ with
	$D(\varphi_3) = d-3$ and $x \in (\varphi_3)_{K_3 r}$. 

	Continuing in this way,
	we obtain a terminating
	sequence of faces $\varphi_1 \supset
	\varphi_2 \supset \cdots \varphi_m$, with $m \leq d$, such that
	for $j =1,2,\ldots,m$ we have $D(\varphi_j) = d-j$ and
	$x \in (\varphi_j)_{K_j r}$, and $x \notin
	(\partial \varphi_m)_{K_{m+1} r}$
	(the sequence must terminate because if $D(\varphi)=0$
	then $(\partial \varphi)_s =\emptyset$ for all $s>0$
	by definition). But then $G_{n,k,r,\varphi_m}$ occurs,
	completing the proof.
\end{proof}

\begin{lemm}
	\label{lemfromang}
	Let $r>0$
	and $j \in \N$.
	Suppose $\varphi \in \Phi(A)$ and
	$x \in \varphi_{K_{j} r} \setminus 
	(\partial \varphi)_{K_{j+1} r}$. Then  
	for all $\varphi' \in \Phi(A)$ with $D(\varphi') = d-1$
	and  $\varphi \setminus \varphi' \neq \emptyset$,
	we have $\dist(x, \varphi') \geq r$.
\end{lemm}
\begin{proof}
	Suppose some such $\varphi'$ exists
	with $\dist(x,\varphi') < r$. Then 
	there exist points $z \in \varphi$ and $z' \in \varphi'$
	with $\|z-x\| \leq K_j r$, and $\|z'-x\|  < r$,
	so that by the triangle inequality $\| z'-z\| < (K_j+1)r$.
	By (\ref{Kphi2def}) and (\ref{KAdef}),
	this implies that $\dist(z, \partial \varphi) < K_A(K_j+1) r$.
On the other hand, 
	we also have that 
	$$
	\dist(z, \partial \varphi) \geq \dist (x,\partial \varphi)
	- \|z-x\|  \geq (K_{j+1} - K_j)r,
	$$
	and combining these inequalities  shows
	that $K_{j+1} - K_j < K_A(K_j+1)$, that is,
	$K_{j+1} < K_j(K_A+1) +K_A < (K_j+1)(K_A+1)$. However, we earlier
	defined the sequence $(K_j)$ in such a way that
	$K_{j+1} \geq (K_j+1)(K_A+1)$,
	so we have a contradiction.
	\qed
\end{proof}

\begin{lemm}
\label{lemedge2}
	Let $\varphi$ be a face of $A$.
	If $\beta = \infty $ then let $u > 1/(f_{\varphi} \rho_\varphi)$.
	If $\beta < \infty$, let $u >  \hH_\beta(D(\varphi)/d) / (f_{\varphi} \rho_\varphi) $. 
For each $n \in \N$,
	set $r_n = (u k(n)/n )^{1/d}$ if $\beta = \infty$, and
	set $r_n = (u (\log n)/n )^{1/d}$ if $\beta < \infty$.
	Then:

	(i) if $\beta = \infty$ then a.s. the events 
	$G_{n,k(n),r_n,\varphi}$ occur for only finitely many $n$;

	(ii) if $\beta < \infty$ then
	 there exists $\eps>0$ such that, setting $k'(n):=
	 \lfloor (\beta + \eps) \log n\rfloor$, we have that
	 $\Pr[G_{n,k'(n),r_n,\varphi}]
	= O(n^{- \eps})$ as $n \to \infty$.
\end{lemm}

\begin{proof}
	Set $j = d-D(\varphi)$.
	We shall apply Lemma \ref{lemmeta},
	now taking 
	$A_r = \varphi_{K_j r} \setminus (\partial \varphi)_{K_{j+1} r}$.
	With this choice of $A_r$, observe first that
	$\kappa(A_r,r) = O(r^{-D(\varphi)})$ as $r \downarrow 0$.

	Let $\delta \in (0,1)$. Assume $u > 1/(f_\varphi \rho_\varphi
	(1-\delta))$ if $\beta = \infty$ and
	assume that
	$u > \hH_\beta(D(\varphi)/d)/(f_\varphi \rho_\varphi (1-\delta))$ 
	if $\beta < \infty$.

        By Lemma \ref{lemfromang}, for all small enough $r$, and
	for all $x \in \phi_{K_j r} \setminus (\partial \phi)_{K_{j+1} r}$,
	and all $s \in (0,r]$,
the ball $B(x,s)$ does not intersect any of
the faces of dimension $d-1$,
        other than those
        which meet at
        $\varphi$ (i.e., which contain $\varphi$).
        Also $f(y) \geq (1-  \delta) f_\varphi$ for all $y \in A $ sufficiently
	close to $\varphi$.  Hence 
	$$
	\mu(B(x,s)) \geq (1- \delta) f_\varphi \rho_\varphi s^d.
	$$
	Therefore we can apply Lemma \ref{lemmeta}, now with
	 $A_r = \varphi_{K_j r} \setminus (\partial \varphi)_{K_{j+1} r}$,
	 taking $a = (1-\delta) f_\varphi \rho_\varphi$ and
	 $b = D(\varphi)$.

	 If $\beta = \infty$, taking $u > 1/(f_\varphi 
	 \rho_\varphi (1-\delta))$ and $r_n = (u k(n)/n)^{1/d}$,
	 by Lemma \ref{lemmeta} we have 
	 with probability 1 that 
	$ \varphi_{K_j r_n} \setminus (\partial \varphi)_{K_{j+1} r_n}
	\subset F_{n,k(n),r_n}$ for all large enough $n$,
	which gives part (i).

	If $\beta < \infty$, taking $u > 
	\hH_\beta(D(\varphi)/d)/(f_\varphi \rho_\varphi (1-\delta))$,
	and setting $r_n = (u (\log n)/n)^{1/d}$, 
	we have from Lemma \ref{lemmeta} that
	there exists $\eps >0 $ such that setting 
	$k'(n) := \lfloor (\beta + \eps) \log n \rfloor$,
	we have $\Pr[ \{
	 \varphi_{K_j r_n} \setminus (\partial \varphi)_{K_{j+1} r_n}
	 \subset F_{n,k'(n),r_n} \}^c] = O(n^{-\eps})$,
 which gives part (ii). 
	\qed
\end{proof}

\begin{proof}[Proof of  Theorem \ref{thmpolytope}] 
	Suppose $\beta = \infty$. 
	Let
	$u > \max \left( \frac{1}{\theta_d f_0},
	\max_{\varphi \in \Phi(A)}
	\frac{1}{f_\varphi \rho_\varphi} \right)$.
	Setting $r_n = (u k(n)/n)^{1/d}$, we have from 
	Lemma \ref{lemedge2} that  $G_{n,k(n),r_n,\varphi}$
	occurs only finitely often, a.s.,  for each $\varphi \in \Phi(A)$.
	Hence by Lemma \ref{Flem}, $R_{n,k(n),1} \leq r_n$ for all large
	enough $n$, a.s. Hence, almost surely
	$
	\limsup_{n \to \infty}  \left( n R_{n,k(n),1}^d/k(n) \right) \leq
	u.
	$

	Since $u > 1/(\theta_d f_0)$, by Proposition \ref{thm1}
	  we also have
	$\limsup_{n \to \infty}  \left( n \tR_{n,k(n)}^d/k(n) \right) \leq
	u$,
	  almost surely,    
	and hence by (\ref{mineq2}), almost surely 
	$$
	\limsup_{n\to \infty}
	\left( n R_{n,k(n)}^d/k(n) \right) \leq 
	\max \left( 
	 (\theta_d  f_0)^{-1}, 
	\max_{\varphi \in \Phi(A)}
	1/(f_\varphi \rho_\varphi)
	\right).
	$$
	Moreover, by
	Lemma \ref{lemtopelb},
	Proposition \ref{thm1} and (\ref{mineq2}) 
	we also have that
	$$
	\liminf_{n\to \infty}
	\left( n R_{n,k(n)}^d/k(n) \right) \geq 
	\max \left( 
	  (\theta_d f_0)^{-1}, 
	\max_{\varphi \in \Phi(A)}
	1/(f_\varphi \rho_\varphi)
	\right),
	$$
	and thus (\ref{0717a}).

	Now suppose $\beta < \infty$.
	Let $u > \max \left(
	 \hH_\beta(1)/ (f_0\theta_d), 
	\max_{\varphi \in \Phi(A)}
	\hH_\beta(D(\varphi)/d)/(f_\varphi \rho_\varphi)
	\right)$.  
	Set $r_n := (u (\log n)/n)^{1/d}$.
	Given $\varphi \in \Phi(A)$, by
	Lemma \ref{lemedge2}  there exists $\eps >0$ such
	that, setting $k'(n):= \lfloor(\beta + \eps) \log n \rfloor$,
	we have $\Pr[G_{n,k'(n),r_n,\varphi}
	] = O(n^{-\eps})$.
	Hence  by Lemma \ref{Flem} and the union bound, 
	$$
	\Pr[ n R_{n,k'(n),1}^d/ \log n > u] = \Pr[R_{n,k'(n),1} > r_n] =
	O(n^{- \eps}).
	$$
	Thus by the subsequence trick (Lemma \ref{lemtrick} (a)),
	$
	\limsup_{n \to \infty} \left( n R_{n,k(n),1}^d/ \log n
	\right) \leq u,
	$
	almost surely.
	Since
	$u > \hH_\beta(1)/ (f_0\theta_d)$,  
	and we take $B=A$ here,
	by Proposition \ref{thm1}
	  we also have a.s. that
	$\limsup_{n \to \infty}  \left( n \tR_{n,k(n)}^d/\log n \right) \leq
	u$,
	and hence by (\ref{mineq2}), almost surely 
	$$
	\limsup_{n\to \infty}
	\left(
	n R_{n,k(n)}^d/\log n \right) \leq 
	\max \left( \frac{
		\hH_\beta(1) }{  f_0\theta_d } 
	 ,
	\max_{\varphi \in \Phi(A)} \left(
	\frac{ \hH_\beta(D(\varphi)/d)}{ f_\varphi \rho_\varphi}
	\right)
	\right).  
	$$
	Moreover, by
	Lemma \ref{lemtopelb},
	Proposition \ref{thm1} and (\ref{mineq2}), we also have a.s. that
	$$
	\liminf_{n\to \infty}
	\left( n R_{n,k(n)}^d/k(n) \right) \geq 
	\max \left( 
	\frac{\hH_\beta(1)}{\theta_d f_0}, 
	\max_{\varphi \in \Phi(A)}
	\left(
	\frac{ \hH_\beta(D(\varphi)/d) }{ f_\varphi \rho_\varphi } \right)
	\right),
	$$
	and thus (\ref{0717b}).
	\qed
\end{proof}

\section{Proof of results from Section \ref{secweak}}
\label{secpfwk}
Throughout this section, we assume $f = f_0 {\bf 1}_A$, where
$A \subset \R^d$ is compact and Riemann measurable with $|A| >0$,
and $f_0 := |A|^{-1}$.

\subsection{{\bf Preliminaries, and proof of Propositions \ref{Hallthm}
and \ref{lemrewrite}}}
\label{subsecprelims}
We start by showing that any weak convergence result  for $R'_{t,k}$ 
(in the large-$t$ limit)
of the type we seek to prove, implies the corresponding
weak convergence result for $R_{n,k}$ in the large-$n$ limit.
\textcolor{\blue}{This is needed because all of the
results in Section \ref{secweak} are stated both
for $R_{n,k}$ and for $R'_{t,k}$ (these quantities were defined at
(\ref{Rnkdef}) and
(\ref{Rdashdef})).}
\allco
\begin{lemm}[
de-Poissonization]
\label{depolem}
Suppose $\mu$ is uniform over $A$.
 Let $k \in \N$, and $a,b,c \in \R$ with $a >0$ and $b >0$.
Let  $F$ be
 a continuous  cumulative distribution function. Suppose that
\bea
\lim_{t \to \infty} \Pr[a t (R'_{t,k})^d - b \log t - c \lglg t \leq \gamma]
= F(\gamma), ~~~~ \forall \gamma \in \R. 
\label{0114a2}
\eea
Then 
\bea
\lim_{n \to \infty} \Pr[a n R_{n,k}^d - b \log n - c \lglg n \leq \gamma]
= F(\gamma), ~~~~ \forall \gamma \in \R. 
\label{0114b2}
\eea
\end{lemm}
\begin{proof}
For each $n \in \N$, set $t(n) := n- n^{3/4}$.
 Let $\gamma \in \R$. Given $n \in \N \cap (1,\infty)$,
set
$$
r_n := \left( \frac{b \log n  + c \lglg n  +
  \gamma}{
a n  } \right)^{1/d}.
$$
Then
\bea
t(n) a  r_n^d  - b \log (t(n) ) - c \lglg (t(n) ) 
\nonumber \\
= (b \log n + c \lglg n  + \gamma) \frac{t(n)}{n} 
- b \log (n- n^{3/4}) - c \lglg (n-n^{3/4})
\nonumber \\
\to   \gamma.
\label{togamma}
\eea
 Then by (\ref{0114a2}), and the continuity of $F$,
 we obtain that
\bea
\Pr[R'_{t(n),k} \leq r_n] \to F(\gamma). 
\label{0114e}
\eea
Moreover, since adding further points reduces the $k$-coverage threshold, 
\bea
\Pr[R'_{t(n),k} \leq r_n < R_{n,k}]
\leq \Pr[R'_{t(n),k}  < R_{n,k}]
\leq \Pr[Z_{t(n)} > n ],
\label{0114d}
\eea
which tends to zero by Chebyshev's inequality.

Now suppose $R'_{t(n),k} > r_n$.
 Pick a point $X$ of $B $ that is covered 
by fewer than $k$ of the closed balls of radius $r_n$ centred on points in
$\Po_{t(n)}$  (this can be done
in a measurable way). If, additionally, $R_{n,k} \leq r_n$, then 
we must have $Z_{t(n)} < n$, and at least one
of the points $X_{Z_{t(n)}+1}, X_{Z_{t(n)}+2},\ldots,X_n $ must lie
in $B(X,r_n) $.
Therefore 
\bean
\Pr[ R'_{t(n),k} > r_n \geq R_{n,k}] \leq
\Pr[ \{ n-2 n^{3/4} \leq Z_{t(n) } \leq n\}^c]
+ 2 n^{3/4} \theta_d f_0 r_n^d,
\eean
which tends to zero by Chebyshev's inequality.
Combined with (\ref{0114d}) and (\ref{0114e}) this shows
that 
$
\Pr[ R_{n,k} \leq r_n ] \to F(\gamma)
$
as $n \to \infty$, which gives us (\ref{0114b2}) as required.
\qed
\end{proof}

The {\em spherical Poisson Boolean model (SPBM)} is defined to be a collection
of Euclidean balls (referred to as {\em grains}) of i.i.d. random radii,
centred on the points of a
homogeneous Poisson process in the whole of $\R^d$. 
Often in the literature the SPBM is taken to
be the {\em union} of these balls (see e.g. \cite{LP}) but here,
following \cite{HallBk}, we
take the SPBM to be the {\em collection} of these balls, rather than
their union. This enables us to consider multiple coverage: given $k \in \N$
we say a point $x \in \R^d$ is {\em covered $k$ times} by the SPBM
if it lies in $k$ of the balls in this collection. 
The SPBM is parametrised by
the intensity of the Poisson process and the distribution of the radii.

We shall repeatedly use   
the following result, which comes from results in
Janson \cite{Janson} or (when $k=1$) 
  Hall \cite{HallZW}. Recall that $c_d$ was defined at (\ref{cdef}).
\begin{lemm}
\label{lemHall}
Let $d,k \in \N$. Suppose $Y$ is a bounded nonnegative random  
variable, and $\alpha = \theta_d \E[Y^d] $
 is the expected volume of a ball of radius $Y$.
Let $\beta \in \R$. Suppose 
$ \delta(\lambda) \in (0,\infty)$  
is defined for all $\lambda >0$, and satisfies
\bea
\lim_{\lambda \to \infty} \left( \alpha \delta(\lambda)^d \lambda -
\log \lambda  - (d+k-2) \lglg \lambda \right)  
= \beta.
\label{0315c}
\eea
Let $B \subset \R^d$  be compact and Riemann measurable, and for 
each $\lambda >0$ let
	\textcolor{\blue}{$B_\lambda \subset B$
	be Riemann measurable with the properties that  
	$B_\lambda \subset B_{\lambda'}$ 
	whenever $\lambda \leq \lambda'$, and  that
	$\cup_{\lambda >0} B_\lambda \supset B^o$}.
 Let $E_\lambda$ be the
event that every point in $B_\lambda$ is
 covered at least $k$ times by a spherical Poisson Boolean model
with intensity $\lambda$ and radii having the distribution of
$\delta(\lambda)Y$. Then 
\bea
\lim_{\lambda \to \infty} \Pr[E_\lambda] 
=
 \exp \left(-  \left( 
\frac{c_d  (\E [Y^{d-1}] )^d}{ (k-1)! (\E[Y^d ])^{d-1} } 
\right)|B| e^{-\beta}  \right). 
\label{0315b}
\eea
\end{lemm}

\textcolor{\blue}{
In the proof, and elsewhere, we use the fact that asymptotics
of iterated logarithms
are unaffected by multiplicative constants: if $a >0$ then
as $t \to \infty$ we have
\bea
\log \log(ta ) = \log ( \log t ( 1+ (\log a)/\log t)) 
= \log \log t + o(1).
\label{eqloglog}
\eea
}
\begin{proof}[Proof of Lemma \ref{lemHall}]
For $k=1$,  when 
$B_\lambda = B$ for all $\lambda >0$ the result can be obtained from
 \cite[Theorem 2]{HallZW}. 
Since \cite{HallZW} does not address multiple coverage,
we  use \cite{Janson} instead to prove the result for general $k$.
Because of
the way the result is stated in \cite{Janson}, we need to
 express $\log(1/(\alpha \delta^d))$
and $\lglg(1/(\alpha \delta^d))$
asymptotically in terms of $\lambda$ (we are now  
 writing just $\delta $ for $\delta(\lambda)$).
  By (\ref{0315c}),
\bea
\alpha \delta^d  =
\lambda^{-1}(\log \lambda) (1 + o(1)) 
\label{0426k}
\eea
so that
$
\log (1/( \alpha \delta^d)) = \log \lambda  - \lglg 
\lambda  + o(1)
\label{0426b}
$
and
$
\lglg (1/(\alpha \delta^d)) = \lglg \lambda + o(1).
$
Therefore  we have as $\lambda \to \infty$ that
\bea
\lambda \delta^d \alpha - \log (1/(\alpha \delta^d))  
- (d+ k -1) \lglg (1/(\alpha \delta^d))
\nonumber
 \\ 
=
\lambda \delta^d \alpha - \log \lambda 
- (d+ k -2) \lglg \lambda +o(1),
\label{0426c}
\eea
which tends to $\beta$ by  
 (\ref{0315c}).

Let $\alpha_J$ be the quantity denoted $\alpha$ by Janson \cite{Janson}
(our $\alpha$ is the quantity so denoted by Hall \cite{HallZW}). In the present
setting, as described in \cite[Example 4]{Janson},
\bean
\alpha_J = \frac{1}{d!} \left( \frac{\sqrt{\pi} \Gamma (1 +d/2) }{
\Gamma((1+d)/2) } \right)^{d-1} \frac{(\E[Y^{d-1}] )^d}{( \E[Y^d])^{d-1} } 
=  \frac{c_d  (\E[Y^{d-1}])^d}{(\E[Y^{d} ])^{d-1} }.
\eean

Let $\tilde{A} = B(o,r_0)$ with $r_0$ chosen large enough so that
  $B$ is contained in the  interior of $\tilde{A}$.
Let $(X_1,Y_1),(X_2,Y_2),\ldots$ be independent identically distributed random
 $(d+1)$-vectors with $X_1 $ 
uniformly distributed over $\tilde{A}$ and $Y_1 $ having the distribution of 
$Y$, independent of $X_1$.

Set $n(\lambda) := \lceil \lambda |\tilde{A} | - \lambda^{3/4} \rceil$.
Let $\tilde{E}_\lambda$ be the event that every point of  
$B$ is covered at least $k$ times by the balls 
$B(X_1,\delta Y_1),\ldots,B(X_{n(\lambda)},\delta Y_{n(\lambda)})$.
By (\ref{0426k}) we have $\lambda^{3/4} \delta^d \to 0$,
so that $n(\lambda) \delta^d \alpha/|\tilde{A}| = \lambda \delta^d \alpha 
+ o(1)$,
and hence 
 by (\ref{0426c}), \textcolor{\blue}{and (\ref{eqloglog}),}
 we have as $\lambda \to \infty$ that
\bean
\frac{n(\lambda) \delta^d \alpha}{|\tilde{A}|}  - \log \left( \frac{|B|}{\alpha \delta^d} \right)
- (d + k-1) \lglg \left( \frac{|B|}{\alpha \delta^d} \right)
+ \log \left( \frac{(k-1)!}{\alpha_J} \right)
\\
\to  \beta + \log \left(\frac{(k-1)!}{\alpha_J |B|} \right).
\eean
Then by \cite[Theorem 1.1]{Janson} 
\bea
\Pr[ \tilde{E}_\lambda ] \to \exp \left(- \left(  \frac{\alpha_J|B|}{(k-1)!} 
\right)  e^{-\beta} \right).
\label{0426l}
\eea

We can and do 
 assume that our Poisson Boolean model, restricted to grains centred in
$\tilde{A}$, is coupled to the sequence $(X_n,Y_n)_{n \geq 1}$ as follows.
Taking $Z_{\lambda |\tilde{A}|}$ to  
be Poisson distributed with mean $ \lambda |\tilde{A}|, $ independent of
$ (X_1,Y_1), (X_2,Y_2),\ldots$, assume the restricted Boolean model
consists of the balls $B(X_i,\delta Y_i), 
1 \leq i \leq Z_{\lambda |\tilde{A}|}$.

Then $\Pr[\tilde{E}_\lambda  \setminus E_\lambda] \leq
\Pr[Z_{\lambda |\tilde{A}|}
< n(\lambda)]$, which tends to zero by Chebyshev's inequality. Also,
if $\tilde{E}_\lambda$ fails to occur, we can and do choose
 (in a measurable way)
a point $V \in B$ which is covered by fewer   than $k$
of the balls $B(X_i,\delta Y_i), 1 \leq i \leq n(\lambda)$. Then
for large $\lambda $,
\textcolor{\blue}{grains centred outside $\tA$ cannot intersect $B$, and}
$$
\Pr[E_\lambda \setminus \tilde{E}_\lambda]
\leq \Pr[ Z_{\lambda |A|} - n(\lambda) > 2 \lambda^{3/4} ]
+ 2 \lambda^{3/4} \Pr[X_1 \in B(V,\delta) ]
$$
which tends to zero by Chebyshev's inequality and (\ref{0426k}).
These estimates, together with (\ref{0426l}), give us the asserted result
(\ref{0315b})
for general $k$ in the case with  $B_\lambda =B$ for all $\lambda$.
It is then straightforward to obtain (\ref{0315b})
 for  general $(B_\lambda) $ satisfying the stated conditions.
 \qed
\end{proof}

\begin{proof}[Proof of Proposition \ref{Hallthm}]
Suppose for some $\beta \in \R$ that $(r_t)_{t >0}$ satisfies
\bea
\lim_{t \to \infty} \left(
	\theta_d t f_0  r_t^d  - \log ( t f_0)  - (d+k -2) \lglg t  \right) 
= \beta. 
\label{0511a}
\eea
The point process $\Po_t = \{X_1,\ldots,X_{Z_t}\}$ is a homogeneous
Poisson process of intensity
$t f_0$ 
in $A$.
Let $\cQ_t$ be a homogeneous
 Poisson  process of intensity $t f_0$ in $\R^d \setminus A$,
independent of $\Po_t$.
Then $\Po_t \cup \cQ_t$  is a homogeneous Poisson process of intensity 
$t f_0$ in all of $\R^d$.

The balls of radius $r_t$ centred
on the points of $\Po_t \cup \cQ_t$ form a Boolean 
model in $\R^d$,  and in the notation of 
	Lemma \ref{lemHall}, here we have
	\textcolor{\blue}{ $\delta (\lambda)
	= r_t$}, 
	and $\Pr[Y=1]=1$, so that $\alpha =\theta_d$, and
	$\lambda = tf_0$,
	so that
	$\lglg \lambda = \lglg t + o(1)$ \textcolor{\blue}{by
	(\ref{eqloglog})}.
Also, by (\ref{0511a}) we have the condition (\ref{0315c})
from Lemma \ref{lemHall}. 

	First assume $B$ is compact and Riemann measurable  with
	$B \subset A^o$. 
	Then
 for all large enough $t$ we have $B \subset A^{(r_t)}$, in which
case  $R'_{n,k} \leq r_t$ if and only if all locations in $B$ are
covered at least $k$ times by the balls of radius $r_t$ centred
on points of $\Po_t \cup \cQ_t$, which is precisely the event denoted
$E_\lambda$ in Lemma \ref{lemHall}. Therefore by that result, 
 we obtain that
$\Pr[R'_{t,k} \leq r_t] \to  \exp(-(c_d/(k-1)!) |B|e^{-\beta})$.
This yields the second equality of (\ref{0114a}).
We then obtain the first equality of (\ref{0114b}) using Lemma \ref{depolem}. 


	Now consider general Riemann measurable
	$B \subset A$ (dropping the previous stronger
	assumption on $B$).
	Given $\eps >0$, by using \cite[Lemma 11.12]{RGG}
	we can find a Riemann measurable
compact set $B' \subset A^o$ with $|A \setminus B'|< \eps$.
Then $B \cap B'$ is also Riemann measurable.
	Let $S_{Z_t,k}$ be the smallest radius of balls centred on
	$\Po_t$ needed to cover $k$ times the set $B \cap B'$.
	Then $\Pr [S_{Z_t,k} \leq
	r_t] \to \exp( -(c_d/(k-1)!) |B \cap B'|e^{-\beta})$. 
	For sufficiently
	large $t$ we have
	$\Pr[\tR_{Z_t,k} \leq r_t] \leq \Pr[S_{Z_t,k} \leq r_t ] $,
	but also $\Pr [\{S_{Z_t,k} \leq r_t\} \setminus \{ \tR_{Z_t,k}
	\leq r_t\}] $ is bounded by the probability that 
	$A \setminus B'$ is not covered $k$ times by a SPBM of
	intensity $tf_0$ with radii $r_t$, which converges
	to $1 - \exp(-c_d /(k-1)!) |A \setminus B'|e^{-\beta})$.
	Using these estimates  we may deduce that 
	\bean
	\limsup_{t \to \infty} \Pr[\tR_{Z_t,k} \leq r_t]
	\leq \exp[ -(c_d/(k-1)!)(|B| -\eps) e^{-\beta}];
	\\
	\liminf_{t \to \infty} \Pr[\tR_{Z_t,k} \leq r_t]
	\geq
	 \exp[ -(\frac{c_d}{(k-1)!})|B|  e^{-\beta}]
	 - ( 1- \exp(- (\frac{c_d}{(k-1)!}) \eps e^{-\beta} ) ),
	 \eean
	 and since $\eps$ can be  arbitrarily small,   that
	$\Pr[\tR_{Z_t,k} \leq r_t] 
	\to - \exp(-(c_d /(k-1)!) |B |e^{-\beta})$. This yields
	the second equality of (\ref{1228a}),
	and then we can obtain the first
	equality of (\ref{1228a}) 
	by a similar argument to Lemma \ref{depolem}.
	\qed
\end{proof}

	\begin{proof}[Proof of Proposition \ref{lemrewrite}]
		\textcolor{\blue}{It suffices to prove this result in 
		the special case with $c'=0$ (we leave it to the reader
		to verify this). 
		Recall that (in this special case)} we assume
	$a n R_{n,k}^d - b \log n -c \lglg n \tod Z$.
		Let $(r_m)_{n \geq 1}$
		be an arbitrary real-valued  sequence satisfying
		$r_m \downarrow 0$ as $m \to \infty$.
		Let  $t \in \R$. Then for all but finitely
		many $m \in  \N$
		we can and do define $n_m \in \N$
		by
		$$
		n_m:= \lfloor a^{-1} r_m^{-d}
		\left( b \log ((b/a) r_m^{-d} ) + (c+b) \lglg (r_m^{-d})
		+t \right) \rfloor,
		$$
	and set 
		$t_m := a r_m^d n_m - b \log ( (b/a) r_m^{-d}) - (c+b) \lglg 
		(r_m^{-d}) $. 
		As $m \to \infty$,
		we have $t_m \to t$  and also
%
	$\log n_m= \log [ (b/a) r_m^{-d} \log ((b/a) r_m^{-d})] + o(1)$,
	and hence
		$\lglg n_m = \lglg (r_m^{-d}) + o(1)$. Therefore
	\bean
	a n_m r_m^d - b \log n_m - c \lglg n_m =
	a n_m r_m^d - b \log ( (b/a)r_m^{-d}) - b \lglg ((b/a) r_m^{-d})
	\\
		- c \lglg ( r_m^{-d}) + o(1),
	\eean
		which converges to $t$ \textcolor{\blue}{(using
		(\ref{eqloglog}))}.
		Also 
	as $m \to \infty$ we have 
	$n_m \to \infty$ and
	\bean
		\Pr[a r_m^d N(r_m,k) - b \log \left( (b/a) r_m^{-d} \right)
		- (c+b) \lglg( r_m^{-d})  \leq t_m] =
	\Pr[N(r_m ,k) \leq n_m ]
		\\
		= \Pr[ R_{n_m,k} \leq r_m],
	\eean
	and by the convergence in distribution 
	assumption, 
	this  converges
		to $ \Pr[Z \leq t]$.
		\qed
	\end{proof}

	We shall use the following notation throughout the sequel.
	Fix $d,k \in \N$. Suppose 
	that $r_t >0$ is defined for each $t >0$.

Given any point process $\X$ in $\R^d$, and any $t >0$,
define the `vacant' region
\bea
V_t(\X) : =\{  x \in \R^d:\X ( B(x,r_t) ) < k \},
\label{Vtdef}
\eea
which is the set of locations in $\R^d$ covered fewer than $k$ times by the
balls of radius $r_t$ centred on the points of $\X$.
Given also $D \subset \R^d$, define the event
\bea
F_t(D,\X) := \{V_t (\X) \cap D = \emptyset\} ,
\label{Ftdef00}
\eea
which is the event that every location in $D$ 
is covered at least $k$ times by the collection of balls of radius
$r_t$ centred on the points of $\X$. \textcolor{\blue}{Again the
$F$ stands for `fully covered', but
we no longer need the notation
$F_{n,k,r}$ from (\ref{F3def}).
However, we do use
the notation $\kappa(D,r)$ from (\ref{covnumdef}) in the next result, 
which will be used
to show various `exceptional' regions are covered with high probability in
the proofs that follow.}

\begin{lemm}
	\label{lem2meta}
	\textcolor{\blue}{Let $t_0 \geq 0$, and suppose $(r_t)_{t >t_0}$ satisfies
	$t r_t^d \sim c \log t$ as
	$t \to \infty$, for some constant $c >0$.
	Suppose $(\mu_t)_{t \geq t_0}$ are a family
	of Borel measures on $\R^d$, and for each $t \geq t_0$ let
	$\cR_{t}$ be a Poisson process with intensity measure $t \mu_t$
	on $\R^d$. Suppose $(W_t)_{t \geq t_0}$ are 
	Borel sets in $\R^d$, and $a > 0, b \geq 0$ are constants,
	such that (i)
	$\kappa(W_t, r_t) = O(r_t^{-b})$ as $t \to \infty$, and (ii)
	$\mu_t(B(x,s)) > a s^d$
	for all $t \geq t_0$,  $x \in W_t$, $s \in [r_t/2,r_t]$.
	Let $\eps >0$. Then $\Pr[(F_t(W_t, \cR_t))^c ]= O(t^{(b/d)-a c + \eps})$
		as $t \to \infty$.}
\end{lemm}
\begin{proof}
	\textcolor{\blue}{
This proof is similar to that of Lemma \ref{lemmeta}.
	Let $\delta \in (0,1/2)$.
	Since $\kappa(W_t,\delta r_t)$ is at most a constant
	times $\kappa(W_t,r_t)$, we can and do
	cover $W_t$ by $m_t$ balls of radius $\delta r_t$,
	with $m_t= O(r_t^{-b}) = O(t^{b/d})$ 
	as $t \to \infty$.
	Let $B_{1,t},\ldots,B_{m_t,t}$
	be balls with radius $(1-\delta)r_t$ and with the same centres
	as the balls in the covering. Then
	for $t > t_0$ and $1 \leq i \leq m_t$ we have
	$t \mu_t (B_{i,t} ) \geq a t (1-\delta)^d r_t^d$
	and  so by Lemma \ref{lemChern}(d),
	provided $ k < \delta a t (1-\delta)^d r_t^d $
	and $tr_t^d >  (1-\delta) c \log t$ 
	(which is true for large $t$) we have
	\bean
	\Pr[F_t(W_t,\cR_t)^c ] \leq \Pr[ \cup_{i=1}^{m_t}
	\{ \cR_t(B_{i,t}) < k \} ] 
	& \leq & m_t  \Pr[ Z_{at(1-\delta)^d  r_t^d} \leq
	\delta (1-\delta)^d a t r_t^d]
	\\ & \leq & m_t 
	\exp( - (1-\delta)^d a t r_t^d H(\delta) ) 
	\\
	& = & O(t^{b/d} t^{-a(1-\delta)^{d+1} H(\delta) c})
	~~~{\rm as}~ t \to \infty.
	\eean
	Since we can choose $\delta$ so that
	$a(1-\delta)^{d+1} H(\delta) c > ac - \eps$, the result follows.
		}
	\qed
\end{proof}

\subsection{{\bf \textcolor{\blue}{Coverage of a region in a hyperplane by
a Boolean model in a half-space}}}

In this subsection, assume that $d \geq 2$.  Let $\zeta \in \R $, 
 $k \in \N$, and
assume   that $(r_t)_{t >0}$ satisfies  
\bea
\frac{f_0 t \theta_d r_t^d}{2} - \left(\frac{d - 1}{d}\right) \log (t f_0) - 
\left( d+k-3+1/d \right)
\lglg t \to \zeta
~~ {\rm as} ~ t \to \infty,
~~~
\label{rt3c}
\eea
\textcolor{\blue}{
so that  for some function $h(t)$ tending to zero as $t \to \infty$,
\bea
\label{rt3d}
\exp(-\theta_d f_0 t r_t^d/2) =
(t f_0)^{-(d-1)/d} 
(\log t)^{-d-k +3- 1/d} e^{- \zeta + h(t)}.
\eea
	}

 For $t >0$ let $\cU_t$ denote a homogeneous Poisson process on
 the half-space $\bH := \R^{d-1} \times [0,\infty)$
  of intensity $t f_0$.
 Recall from (\ref{cdkdef}) the definition of $c_{d,k}$. 
 \textcolor{\blue}{The next result determines the limiting
 probability of covering a bounded region $\Omega \times \{0\} $ in
 the hyperplane $\partial \bH := \R^{d-1} \times \{0\}$ by
 balls of radius $r_t$ centred on $\cU_t$, or of covering
 the $r_t$-neighbourhood of $\Omega \times \{0\}$ in $\bH$.
 It is crucial for dealing with boundary regions 
 in the proof of Theorems 
 \ref{thsmoothgen}, 
 \ref{thmwksq}  
 and \ref{thwkpol3}.
 In it, $|\Omega|$ 
	denotes the $(d-1)$-dimensional
	Lebesgue measure of $\Omega$.}

\begin{lemm}
\label{lemhalf3a}
	Let $\Omega \subset \R^{d-1}$ and (for each $t >0$)
	$\Omega_t \subset \R^{d-1}$ be closed and Riemann measurable, with
$\Omega_t \subset \Omega$ for each $t >0$.
	Assume (\ref{rt3c}) holds for some $\zeta \in (-\infty,\infty]$
	and also $ \limsup_{t \to \infty}
	( tr_t^d/(\log t)) < \infty$. 
\bea
\lim_{t \to \infty} (\Pr[F_t(\Omega \times \{0\},\cU_t ) ] 
) = 
	\exp \left(-  c_{d,k}   |\Omega| 
e^{-  \zeta}  \right). 
\label{0517c2}
\eea
Also, given 
 $a \in (0,\infty)$ and  $\delta_t >0$ for each $t >0$,
\bea
\lim_{t \to \infty}
(\Pr[F_t((\Omega_t \times \{0\} ) 
	\cup ( ((\partial \Omega_t) \oplus B_{(d-1)}(o,\delta_t)) 
	\times [0,ar_t]),   
 \cU_t) 
	\nonumber \\
\setminus
 F_t ( \Omega_t \times [0, a r_t] , \cU_t)] ) = 0.
\label{0517b2}
\eea 
\end{lemm}
	\begin{remk}
		{\rm
Usually we shall use Lemma \ref{lemhalf3a}
 in the case where $\zeta < \infty$. In
this case the extra condition
	$ \limsup_{t \to \infty} ( tr_t^d/(\log t)) < \infty$ 
	is automatic.
When $\zeta =\infty$, in  (\ref{0517c2}) we use the convention
		$e^{-\infty} := 0$.}
	\end{remk}

	The following terminology and notation will be used in the proof
	of Lemma \ref{lemhalf3a}, and again later on.
	We use bold face for vectors in $\R^d $ here.
Given $\bx \in \R^d$, we let $\pi_d(\bx)$ denote
the $d$-th co-ordinate of $\bx$, and refer to $\pi_d(\bx)$ as the
as  {\em height} of $\bx$.
Given $\bx_1 \in \R^d, \ldots,\bx_d \in \R^d$, and $r >0$,
if $\cap_{i=1}^d \partial B(\bx_i,r)$
consists of exactly two points, we refer to these as
$\bp_r(\bx_1,\ldots,\bx_d)$ and $\bq_r(\bx_1,\ldots,\bx_d)$
with $\bp_r(\bx_1,\ldots,\bx_d)$ at a smaller height
than 
 $\bq_r(\bx_1,\ldots,\bx_d)$ (or if they are at the same height, take 
$\bp_r(\bx_1,\ldots,\bx_d) <  
\bq_r(\bx_1,\ldots,\bx_d)$
in the lexicographic ordering).
Define the indicator function
\bea
h_r(\bx_1,\ldots,\bx_d) := {\bf 1} \{
\pi_d(\bx_1) \leq \min(\pi_d(\bx_2), \ldots, \pi_d(\bx_d)) \}
\nonumber \\
\times {\bf 1} \{ \#( \cap_{i=1}^d \partial B(\bx_i,r)  
) = 2 \} {\bf 1} \{ \pi_d(\bx_1) <
 \pi_d(\bq_r(\bx_1,\ldots,\bx_d))\} .  
\label{hrdef2}
\eea

\begin{proof}[Proof of Lemma \ref{lemhalf3a}]
Assume for now that $\zeta < \infty$.
Considering the slices of  balls of radius $r_t$ centred on points of
	$\cU_t$ that
	intersect the hyperplane $\R^{d-1} \times \{0\}$,
we have a $(d-1)$-dimensional Boolean model with (in the notation of Lemma
	\ref{lemHall})
$$
\delta = r_t, ~~~ \lambda = t f_0 r_t, ~~~~ \alpha = \frac{\theta_d}{2}
	, ~~~~~ \E[ Y^{d-1}] = \frac{\theta_d}{2 \theta_{d-1}} 
	, ~~~~~
	\E[Y^{d-2}]= \frac{\theta_{d-1}}{2 \theta_{d-2}}.
$$ 
	To see the moment assertions here, note that here $Y= (1-U^2)^{1/2}$
	with $U$ uniformly distributed over $[0,1]$, so that
	$\theta_{d-1}Y^{d-1}$
	is the $(d-1)$-dimensional Lebesgue measure of a $(d-1)$-dimensional
	affine slice through the unit
	ball at distance $U$ from its centre, and an application
	of Fubini's theorem gives the above assertions
	for  $\alpha:= \theta_{d-1} \E[Y^{d-1}]$
	and hence for $\E[Y^{d-1}]$; 
	the same argument in a lower dimension gives the assertion
	regarding $\E[Y^{d-2}]$. (We take $\theta_0=1$ so the assertion
	for $\E[Y^{d-2}]$ is valid for $d=2$ as well.)

By (\ref{rt3c}),  as $t \to \infty$ we have
	$r_t \sim (2 -2/d)^{1/d} 
	(\theta_d t f_0)^{-1/d} (\log t)^{1/d}$,  and therefore
$
\lambda \sim (2 - 2/d)^{1/d}
	\theta_d^{-1/d}
 (t f_0)^{1-1/d} (\log t )^{1/d}
$
so that 
$$
	\log \lambda =  
	(1-1/d) \log ( t f_0) + d^{-1} \lglg t +
	d^{-1} \log  \left( \frac{2-2/d}{\theta_d} \right) +
	o(1).
$$ 
	\textcolor{\blue}{Hence, $\log \lambda = (\log t)(1- 1/d)(1+ g(t))$ for 
	some function $g(t)$ tending to zero. 
	Therefore
	 $\lglg \lambda = \lglg t + \log (1-1/d) + o(1)$.} Checking 
	(\ref{0315c})
	here, we have
\bea
	\alpha \delta^{d-1} \lambda - \log \lambda - (d+k -3) \lglg \lambda
	\nonumber \\
= (\theta_d/2)  f_0 tr_t^d 
- (1-1/d) \log (t f_0) 
- (d+k -3 +1/d) \lglg t
\nonumber
\\
	- d^{-1} \log  ((2-2/d)/\theta_d) 
	- (d+k-3 ) \log (1-1/d) + o(1), 
\label{0518a2}
\eea
so by  using (\ref{rt3c}) again we obtain that
\bean
\lim_{t \to \infty} (
	\alpha \delta^{d-1} \lambda - \log \lambda - (d+k -3) \lglg \lambda )
~~~~~~~~~~~~~~~~~~~~~~~~~~~~~~~~~
~~~~~~~~~~~~~
\\
= 
\zeta - d^{-1} \log (2/\theta_d) - (d+k -3 + d^{-1}) \log (1-1/d).
\eean
Also $(\E[Y^{d-2}])^{d-1}/(\E[Y^{d-1}])^{d-2} = \theta_d^{2-d}
\theta_{d-1}^{2d-3} \theta_{d-2}^{1-d}/2$, so
(\ref{0517c2}) follows by Lemma \ref{lemHall}
and (\ref{cdkdef}). 

Having now verified (\ref{0517c2}) in the case where $\zeta < \infty$,
	we can easily deduce (\ref{0517c2}) in the other case too.

It remains to  prove (\ref{0517b2}); we now consider general $\zeta  \in 
(0,\infty]$. 
Let $E_t$ be the (exceptional)
 event that there exist $d$ distinct points $\bx_1,\ldots,\bx_d$
of $\cU_t$
such that $ \cap_{i=1}^d \partial B(\bx_i,r_t) $
has non-empty intersection with the hyperplane $\R^{d-1} \times \{0\}$. 
Then $\Pr[E_t]=0$.

Suppose that the event displayed in (\ref{0517b2}) occurs, and that $E_t$
 does not. 
Let $\bw$ be a location  of minimal height (i.e., $d$-coordinate) 
in the closure of
$V_t(\cU_t) \cap( \Omega_t \times [0,a r_t])$.
Since we assume $F_t(((\partial \Omega_t) \oplus B_{(d-1)}(o, \delta_t))
\times [0,a r_t], \cU_t)$ occurs, 
$\bw$ must lie in $\Omega_t^o \times
[0,ar_t]$. Also 
\textcolor{\blue}{we claim that}
$\bw$ must
  be  a `corner' given by the
meeting point of the boundaries of exactly $d$ balls of radius $r_t$ centred at
 points of $\cU_t $,
located at $\bx_1,\ldots,\bx_d$ say, with $\bx_1$  the lowest
of these $d$ points,
and with
 $\#(\cap_{i=1}^d \partial B(\bx_i,r_t) 
) =2$, 
and $\bw \in V_t( \cU_t \setminus \{\bx_1,\ldots,\bx_d\})$.

\textcolor{\blue}{
Indeed, if $\bw$ is not at the boundary of any such ball,
then for some $\delta >0$ we have $B(\bw, \delta) \subset
V_t(\cU_t)$,
and then we could find a location in $V_t(\cU_t) \cap (\Omega_t \times
[0,ar_t])$ lower than $\bw$,
a contradiction.
Next, suppose instead that $\bw$ lies at the boundary of fewer than $d$ such balls.
Then denoting by $L$ the intersection of the supporting hyperplanes
at $\bw$ of each of these balls, we have that $L$ is an
affine subspace of $\R^d$, of
dimension at least 1. Take $\delta  >0$ small enough
so that $B(\bw,\delta)$ does not intersect any of the boundaries
of balls of radius $r_t$ centred at points of $\cU_t$,
other than those which meet at $\bw$.
Taking $\bw' \in L \cap B(\bw, \delta) \setminus \{\bw\}$ 
such that $\bw'$ is at least as low as $\bw$, we have
that $\bw'$ lies
in the interior of $V_t(\cU_t)$. Hence for some $\delta' >0$,
$B(\bw',\delta') \subset V_t(\cU_t)$ and we can
find a location in $B(\bw',\delta')$ that is lower than $\bw$, 
yielding a contradiction for this case too. 
Finally, with probability 1 there is no set of
more than $d$ points of $\cU_t$ such that the boundaries
of balls of radius $r_t$ centred on these points have non-empty
intersection, so $\bw$ is not at the boundary of more than
$d$ such balls.
Thus we have justified the claim.
}

Moreover $\bw $ must be the point 
\textcolor{\blue}{$ \bq_{r_t}(\bx_1,\ldots,\bx_d)$ rather than $\bp_{r_t}(\bx_1,\ldots,\bx_d)$,}
 because otherwise by extending the line segment from 
$ \bq_{r_t}(\bx_1,\ldots,\bx_d)$ to  $\bp_{r_t}(\bx_1,\ldots,\bx_d)$ slightly 
beyond  $\bp_{r_t}(\bx_1,\ldots,\bx_d)$ we could find  a  point in
$V_t(\cU_t) \cap (\Omega_t \times [0,a r_t])$
lower than $\bw$, contradicting the statement that $\bw$ is
a location of minimal height in the closure  of
$V_t(\cU_t) \cap (\Omega_t \times [0,a r_t])$.
Moreover, $\bw$ must be strictly higher than $\bx_1$, since if
$\pi_d(\bw) \leq  \min(\pi_d(\bx_1), \ldots,\pi_d(\bx_d))$,
then locations  just below $\bw$ would lie
in $V_t(\cU_t) \cap (\Omega_t \times [0,a r_t])$, contradicting
the statement that $\bw$ is
a  point of minimal height in the closure of $V_t(\cU_t \cap (\Omega_t \times 
[0,ar_t]))$.
Hence, $h_{r_t}(\bx_1,\ldots,\bx_d) = 1$, where $h_r(\cdot)$ was
defined at (\ref{hrdef2}).

Note that there is a constant
$c' := c'(d,a) >0$ such
that for any $\bx \in \bH$ with $0 \leq \pi_d(\bx) \leq a$ we
have that $|B(\bx,1) \cap  \bH | \geq (\theta_d/2) + c' \pi_d(\bx)$.
Hence for any $r>0$, and
any $\bx \in \bH$ with
$0 \leq \pi_d (\bx) \leq a r$,
\bean
|B(\bx,r) \cap \bH| = r^d |B(r^{-1}\bx,1) \cap \bH | \geq (\theta_d/2) 
r^d + c'
r^{d-1} \pi_d(\bx). 
\eean

Thus if the event  displayed in (\ref{0517b2}) holds, then almost surely
there exists at least
one $d$-tuple of points $\bx_1,\ldots,\bx_d \in \cU_t$, such that
$h_{r_t}(\bx_1,\ldots,\bx_d)=1$,
 and moreover
 $\bq_{r_t}(\bx_1,\ldots,\bx_d) \in V_t( \cU_t \setminus
 \{\bx_1,\ldots,\bx_d\}) \cap ( \Omega_t \times [0,a r_t])$.
By the Mecke formula (see e.g. \cite{LP}), 
there is a constant $c$ such that
 the expected number of such $d$-tuples is bounded above by
\bea
c t^d (tr_t^d)^{k-1} \int_{\bH} d\bx_1
 \cdots
 \int_{\bH} d\bx_d h_{r_t}(\bx_1,\ldots,\bx_d)
 {\bf 1}\{  \bq_{r_t}(\bx_1,\ldots,\bx_d) \in \Omega_t \times [0,ar_t] \}
\nonumber \\
\times
 \exp(- (\theta_d/2) f_0 t  r_t^d - c' f_0 t r_t^{d-1}
 \pi_d( \bq_{r_t}(\bx_1,\ldots,\bx_d)
 )).  ~~~
\label{0522a2}
\eea

Now we change variables to $\bfy_i= r_t^{-1}(\bx_i-\bx_1)$ for
$2 \leq i \leq d$,
noting that 
\bean
\pi_{d}(\bq_{r_t}(\bx_1,\ldots,\bx_d)) = \pi_d(\bx_1) +
\pi_d( \bq_{r_t}(o,\bx_2 -\bx_1, \ldots,  \bx_d-\bx_1) ) )
\\
= \pi_d(\bx_1) + r_t \pi_d( \bq_1(o,\bfy_2, \ldots, \bfy_d)).
\eean
Hence by (\ref{rt3d}), there is a constant $c''$ such that
 the expression  in (\ref{0522a2}) is at most
\bea
c'' t^d
 (\log t)^{k-1}
 t^{-1 + 1/d} (\log t)^{3-(1/d) -d - k}
 r_t^{d(d-1)}  
\int_0^{a r_t} e^{-c' f_0 t r^{d-1}_t u} du 
\nonumber \\
\times
\int_\bH \cdots \int_\bH h_{1}(o,\bfy_2,\ldots,\bfy_d) 
\exp( - c' f_0 t r_t^d \pi_d( \bq_1(o,\bfy_2,\ldots,\bfy_d)) ) 
d\bfy_2 \ldots d\bfy_d.
\label{0512b}
\eea
In the last expression the first line is bounded by a constant times
the expression
\bean
t^{d-1 + 1/d} r_t^{d(d-1)} (\log t)^{2-d -1/d} (t r_t^{d-1})^{-1} 
=  (t r_t^d)^{d-2 + 1/d} (\log t)^{2-d -1/d}, 
\eean
which is bounded because we assume $\limsup_{t \to \infty}
(tr_t^d/(\log t))< \infty$. 
The second line of  (\ref{0512b}) tends to zero by dominated convergence
because $tr_t^d \to \infty$ by (\ref{rt3c}) and
because the
indicator
function $ (\bz_2,\ldots,\bz_d) \mapsto h_1(o,\bz_2,\ldots,
\bz_d) $ has bounded support
and  is zero when $\pi_d(\bq_1(o,\bz_2,\ldots,\bz_d)) \leq 0$. Therefore by
Markov's inequality we obtain (\ref{0517b2}).
\qed
\end{proof}

\subsection{{\bf Polygons: proof of Theorem \ref{thmwksq}}}

\textcolor{\blue}{
Of the
three theorems in Section \ref{secweak},
Theorem \ref{thmwksq} has the simplest proof, so we give this one first.} 
Thus,
in this subsection, we
set $d=2$, take
$A$ to be polygonal with $B=A$, and take
$f \equiv f_0 {\bf 1}_A$, with $f_0 := |A|^{-1}$. 
Denote the vertices of $A$ by $q_1,\ldots,q_\kappa$, and the angles subtended
at these vertices by $\alpha_1,\ldots,\alpha_\kappa$ respectively. Choose
$K \in (2, \infty)$ such that $K \sin \alpha_i > 9$ for
$i=1,\ldots, \kappa$. 

For the duration of this subsection (and the next) we
fix $k \in \N$ and $\beta \in \R$.  Assume we are given 
real numbers
$(r_t)_{t >0}$ satisfying
 \bea
\lim_{t \to \infty} (\pi t f_0 r_t^2 - \log ( t f_0) + (1 -  2k ) \lglg t ) =  \beta. 
\label{0220a}
\eea
Setting $\zeta = \beta/2$, this is the same  as
the condition (\ref{rt3c}) for $d=2$.

Given any $D \subset \R^2$,
 and any point process $\X$ in $\R^2$,
 and $t >0$,  define  the `vacant region' $V_t(\X)$ and the event
 $F_t(D,\X)$ as at (\ref{Vtdef}) and (\ref{Ftdef00}).

For $t >0$, in this subsection we define the `corner regions' $Q_t$ and
$Q_t^-$ 
by  
$$
Q_t:= \cup_{j=1}^\kappa B(q_j,(K+9)r_t);
~~~~~
Q^-_t:= \cup_{j=1}^\kappa B(q_j,Kr_t).
$$
\begin{lemm}
\label{lemQ}
It is the case that $\Pr[F_t(Q_t \cap A,\Po_t)] \to 1$ as $t \to \infty$.
\end{lemm}
\begin{proof}
	We have
	 $\kappa(Q_t \cap A,r_t) = O(1)$ as $t \to \infty$.
	Also there exists $a>0$ such that
	for all large enough $t$, all $s \in [r_t/2,r_t]$
	and $x \in Q_t \cap A$,
	we have $f_0 |B(x,s) \cap A| \geq a r_t^d$. 
	Thus we can apply Lemma \ref{lem2meta},
	taking $W_t = Q_t \cap A$, and $\mu_t (\cdot) = f_0 |\cdot \cap A|$,
	with $b = 0$,
	to deduce that $\Pr[(F_t(W_t, \Po_t))^c ]  $ is $ O(t^{-\delta})$
	for some $\delta >0$, and hence tends to 0.
%
\qed
\end{proof}
\begin{lemm}
\label{LemRpp}
It is the case that
\bea
\lim_{t \to \infty} (
\Pr[F_t(\partial A \setminus Q^-_t,\Po_t)
])  =
	\exp(- c_{2,k} |\partial A| 
	e^{-\beta/2}).
\label{0429a}
\eea
Also,
\bea
\lim_{t \to \infty} (
	\Pr[F_t (A^{(3r_t)} \cup \partial A \cup (Q_t \cap A),\Po_t)
 \setminus F_t(A,\Po_t)] )   = 0.
\label{eqlem0220}
\eea
\end{lemm}
\begin{proof}
Denote the line segments making up $\partial A$ by $I_1,\ldots,I_{\kappa}$,
and for $ t >0$ and $1 \leq i \leq \kappa$ set $I_{t,i} := I_i \setminus
Q^-_t$.

	Let $i,j,k \in \{1,\ldots,\kappa\}$ be such that
	$i\neq j$ and the edges $I_i$ and $I_j$ are both incident
	to $q_k$. If $x \in I_{t,i}$ and
	$y \in I_{t,j}$, then $\|x-y \| \ge (Kr_t) \sin \alpha_k \geq 9r_t$. 
	Hence 
	the events $F_t(I_{t,1}, \Po_t),\ldots,F_t(I_{t,\kappa}, \Po_t)$
are mutually independent. Therefore
\bean
\Pr[ F_t(\partial A \setminus Q^-_t, \Po_t) ] =
\prod_{i=1}^\kappa \Pr[ F_t(I_{t,i}, \Po_t) ],
\eean
and by Lemma \ref{lemhalf3a} this converges to the right hand
side of (\ref{0429a}).

Now we prove (\ref{eqlem0220}).
For $t >0$, and
 $i \in \{1,2,\ldots,  \kappa \} $, let $S_{t,i}$ denote the rectangular block
of dimensions $|I_{t,i} | \times 3 r_t$, consisting of all
points in $A$ at perpendicular distance at most $3 r_t$ from $I_{t,i}$.
Let $\partial_{{\rm side}}S_{t,i}$ denote the union of the two
`short' edges of the boundary $S_{t,i}$, i.e. the two edges 
bounding $S_{t,i}$ which are perpendicular to $I_{t,i}$.

Then $A \setminus (A^{(3r_t)} \cup Q_t) \subset \cup_{i=1}^\kappa S_{t,i}$,
	and also $(\partial_{\rm side} S_{t,i} \oplus B(o,r_t))
	\subset Q_t$ for $1 \leq i \leq \kappa$,
so that 
\bean
	F_t (A^{(3r_t)} \cup \partial A \cup (Q_t \cap A),\Po_t)
 \setminus F_t(A,\Po_t) 
	~~~~~~~~~~~~~~~~~~~~~~~
	~~~~~~~~~~~~~~~~~~~~~~~
	\\
	\subset \cup_{i=1}^\kappa 
	[F_t(I_{t,i} \cup (\partial_{{\rm side}} S_{t,i} \oplus B(o,r_t)),
	\Po_t )  
\setminus F_t(S_{t,i},\Po_t)].
\eean
For  $i \in \{1,\ldots,\kappa\}$, 
let $I'_{t,i}$ denote an interval of length $|I_{i,t}|$ contained in
the $x$-axis. By a rotation one sees that
	$\Pr[F_t(I_{t,i} \cup (\partial_{{\rm side}} S_{t,i} \oplus 
	B(o,r_t)), \Po_t )  
	\setminus F_t(S_{t,i}, \Po_t)]$
is at most
 $$
 \Pr[ 
	F_t((I'_{t,i} \times \{0\}) \cup (((\partial I'_{t,i})
	\oplus [-r_t,r_t])
	\times [0,3 r_t] ),\cU_t )
\setminus 
F_t(I'_{t,i} \times [0,3r_t],\cU_t) ] ,
$$
where $\cU_t$ is as in
 Lemma \ref{lemhalf3a}. 
By (\ref{0517b2}) from that result, 
 this probability tends to zero, which gives us (\ref{eqlem0220}).
	\qed
\end{proof}

\begin{proof}[Proof of Theorem \ref{thmwksq}]
Let $\beta \in \R$ and suppose $(r_t)_{t >0}$ satisfies 
(\ref{0220a}). 
Then
$$
\lim_{t \to \infty} ( \pi f_0 t r_t^2 - \log (t f_0) - k \lglg t  ) 
= \begin{cases}
\beta & \mbox{\rm if } k=1 \\
+ \infty  & \mbox{\rm if } k \geq 2. 
 \end{cases}
$$
	Hence by (\ref{1228a}) from
	Proposition \ref{Hallthm} (taking $B=A$ there and using 
	(\ref{Fnequiv})),
\bea
\lim_{t \to \infty} \Pr[ F_t(A^{(r_t)},\Po_t ) 
]  = 
	\lim_{t \to \infty}
	\Pr [ \tR_{Z_t,k} \leq r_t]  =
   \begin{cases}
\exp( -  |A| e^{-\beta}) & \mbox{\rm if }  k=1
\\
1 & \mbox{\rm if } k \geq 2. 
\end{cases}
\label{0322a}
\eea
	Observe that $\{ F_t(A^{(r_t)},\Po_t) \subset
 F_t(A^{(3r_t)},\Po_t) $. 
	We claim that
\bea
\Pr[ F_t(A^{(3r_t)},\Po_t) \setminus F_t(A^{(r_t)},\Po_t)] \to 0
~~~
{\rm as} ~
t \to \infty.
\label{1212a}
\eea
Indeed, given $\eps >0$, for large $t $
the probability on the left side of (\ref{1212a})
	is bounded by $\Pr[ F_t(A^{(r_t)}\setminus A^{[\eps]},\Po_t)^c]  $
	(the set $A^{[\eps]}$ was defined in Section \ref{secdefs}), and 
	 by (\ref{1228a}) from
	 Proposition \ref{Hallthm}
	  (as in (\ref{0322a}) but
	now taking $B=A \setminus A^{[\eps]}$
	 in Proposition \ref{Hallthm}),
the latter
probability tends to a limit which can be made arbitrarily
small by the choice of $\eps$.

Also, by (\ref{0429a}) and Lemma \ref{lemQ}, 
\bea
\lim_{t \to \infty} (
\Pr[F_t(\partial A,\Po_t )
])  =
	\exp(- c_{2,k} |\partial A|  e^{-\beta/2}).
\label{0510a}
\eea

Using (\ref{eqlem0220})  
followed by  Lemma  \ref{lemQ}, we obtain that
\bea
\lim_{t \to \infty} \Pr[ F_t(A,\Po_t) ]  
= \lim_{t \to \infty} \Pr[ F_t(A^{(3r_t)}  \cup \partial A \cup
(Q_t \cap A),\Po_t) ]  
\nonumber \\
= \lim_{t \to \infty} \Pr[ F_t(A^{(3r_t)}  \cup \partial A,\Po_t) ],  
\label{0510b}
\eea
provided these limits exist.

	By (\ref{1212a}), (\ref{0322a}) still 
	holds with $A^{(r_t)}$ replaced by  $A^{(3r_t)}$
	on the left.
Also, the events $F_t(\partial A,\Po_t) $  and
and $F_t(A^{(3r_t)},\Po_t) $ are independent since the first of
these events is determined by the configuration of
 Poisson points
distant at most $r_t$ from $\partial A$,
 while the second event is determined by the Poisson points 
 distant more than $2r_t$ from $\partial A$. 
  Therefore the limit in (\ref{0510b}) does
indeed exist, and is the product of the limits arising  in (\ref{0322a})
and (\ref{0510a}). 

Since $F_t(A,\Po_t) = \{R'_{t,k} \leq r_t\} $, this  
 gives us the second equality of  (\ref{0322b}), and we then obtain
the first equality of  (\ref{0322b}) by using Lemma 
\ref{depolem}.
\qed
\end{proof}

\subsection{{\bf Polyhedra: Proof of Theorem \ref{thwkpol3}}}
\label{secpolypf}

Now suppose that $d=3$ and $A=B$ is a polyhedron with $f \equiv f_0 {\bf 1}_A$,
where $f_0 = 1/|A|$.
Fix $k \in \N$ and
$\beta \in \R$. Recall that $\alpha_1$ denotes the smallest edge angle of $A$.
Throughout this subsection,
we \textcolor{\blue}{assume that we 
are given} $(r_t)_{ t >0}$ with the following limiting behaviour
as $t \to \infty$:
\bea
 ( \pi \wedge 2\alpha_1   ) f_0 t r_t^3
-  \log (t f_0) - \beta
  =  \begin{cases}
  (3k -1)  \lglg t  + o(1)
  & \mbox{ if  }  \alpha_1 \leq \pi/2
\\
(\frac{1+ 3k}{2}) \lglg t + o(1) 
  & \mbox{ if  }  \alpha_1 > \pi/2.
	  ~~~~
\end{cases}
\label{rt3a} 
\label{rt3} 
\eea
For $D \subset \R^3$ and  $\X$ a point set in $\R^3$,
define the region $V_t(\X)$ and
event $F_t(D,\X)$  by (\ref{Vtdef}) and
(\ref{Ftdef00}) as before, now with $(r_t)_{t >0}$ satisfying (\ref{rt3a}).

To prove Theorem \ref{thwkpol3}, our main task is to determine
$\lim_{t \to \infty} \Pr[F_t(A,\Po_t)]$. Our strategy for this goes
as follows. 
We shall consider {\em reduced faces} and {\em reduced edges} of $A$,
obtained by removing a  region within distance $Kr_t$
of the boundary of each face/edge, for a suitable constant $K$. 
Then the events that different reduced faces and reduced edges are covered
are mutually independent.
Observing that the intersection of $\Po_t $ with a
face or edge is a lower-dimensional
SPBM, and using Lemma \ref{lemHall}, 
we shall  determine the limiting probability that the reduced faces
and edges are covered.

If $F_t(A,\Po_t)$ does not occur but all reduced faces and 
all reduced edges are covered, then the uncovered region 
must either meet the region  $A^{(3r_t)}$, or go  near (but not
intersect) one of the reduced faces, or go near (but not intersect)
one of the reduced edges, or go near one of the vertices.
We shall show in turn that each of these possibilities
has vanishing probability.

The following lemma will \textcolor{\blue}{help us deal}
with  the regions near the faces 
of $A$.

\begin{lemm}
\label{lemhalf3}
Let $\Omega \subset \R^2$ and (for each $t >0$)
$\Omega_t \subset \R^2$ be closed and Riemann measurable, with
$\Omega_t \subset \Omega^o$ for each $t >0$. Let $\cU_t$ be
a homogeneous Poisson process of intensity $tf_0$ in $\bH$.
Assume (\ref{rt3}) holds. Then 
\bea
\lim_{t \to \infty} (\Pr[F_t(\Omega \times \{0\},\cU_t ) ] 
) = \begin{cases}
	\exp \left(- c_{3,k}  
	|\Omega| 
e^{- 2 \beta/3}
	\right), 
& \text{\rm if }  \: 
\alpha_1 > \pi/2
\\
&
\mbox{\rm or } \: 
\alpha_1= \pi/2, k=1
\\
1 & \mbox {\rm otherwise.} 
\end{cases}
	~~~
\label{0517c}
\eea
Also in all cases, given $a \in (0,\infty)$ and
	$\delta_t >0$ for each $t >0$,
\bea
\lim_{t \to \infty}
(\Pr[F_t((\Omega_t \times \{0\} ) 
	\cup ( (\partial 
	\Omega_t \oplus B_{(2)}(o,\delta_t)) \times [0,ar_t]),   
 \cU_t) \nonumber \\
\setminus
 F_t ( \Omega_t \times [0, a r_t] , \cU_t)] ) = 0.
\label{0517b}
\eea 
\end{lemm}
\begin{proof}
	In the case where $\alpha_1 > \pi/2$ or $\alpha_1= \pi/2,k=1$ 
	the condition (\ref{rt3}) implies that (\ref{rt3c}) holds
	(for $d=3$)
	with $\zeta:= 2 \beta/3$.
	In the other case (with $\alpha_1 < \pi/2 $ or $\alpha_1 =
	\pi/2, k>1$) the condition (\ref{rt3}) implies that
	(\ref{rt3c}) holds (for $d=3$) with $\zeta := + \infty$, and
	also $\limsup_{t \to \infty}(tr_t^3 / \log t) < \infty$.
	Therefore we can deduce the result (in both cases) from
	Lemma \ref{lemhalf3a}.
	\qed
\end{proof}

In this subsection we use bold face to indicate vectors in $\R^3$.
For $i=1,2,3$, we
let $\pi_i: \R^3 \to \R$ denote projection onto the $i$th coordinate.
For $\bx \in \R^3$ we shall sometimes refer to $\pi_2(\bx)$ and
$\pi_3(\bx)$ as the `depth' and `height' of $\bx$, respectively.

\textcolor{\blue}{
For $r>0$ and  for triples
$(\bx,\bfy,\bz) \in (\R^3)^3$,
we define the function $h_r(\bx,\bfy,\bz)$ and 
(when $h_r(\bx,\bfy,\bz)=1$) the points
$\bp_r(\bx, \bfy,\bz)$
and $\bq_r(\bx, \bfy,\bz)$, as at (\ref{hrdef2}), now specialising to $d=3$.}


For $\alpha \in (0,2 \pi)$, let 
$\bW_\alpha  \subset \R^3 $ be
  the wedge-shaped region 
$$ 
\bW_\alpha := 
	\{(x,r\cos \theta, r\sin \theta) : x \in \R, r \geq 0, \theta \in
	[0,\alpha] \}
 .
$$
For $\alpha \in (0, 2 \pi)$ and
 $t >0$,
let $\cW_{\alpha,t}$ be a homogeneous Poisson process in $\bW_\alpha$ of
intensity $t f_0$.

\begin{lemm}[Lower bound for the volume of a ball within a wedge]
\label{wdglblem}
	Let $\alpha \in (0, \pi)$, 
and $u >0$. There is  a constant $c' = c'(\alpha,u) \in (0,\infty)$ such
that
for any $r >0$ and any $\bx \in \bW_{\alpha} \cap B(o, ur) $,
\bea
|B(\bx,r) \cap  \bW_\alpha |
 \geq (2 \alpha/3) r^3
 + c' r^2 \pi_3(\bx)  
+ c' r^2    (\pi_2(\bx) - \pi_3(\bx) \cot \alpha) .
\label{wdglb}
\eea
\end{lemm}
\begin{proof}
	As illustrated in Figure 2,
	we can find a constant
$c' >0 $ such that for any $\bx \in \bW_\alpha \cap B(o,u) $,
we have that
$$
|B(\bx,1) \cap  \bW_\alpha | \geq (2 \alpha/3) + c' (\pi_3(\bx)+  (\pi_2(\bx)
- (\pi_3(\bx)/ \tan \alpha)) ). 
$$
	 \begin{figure*}[!h]
\label{fig3}
\center
\includegraphics[width=8cm]{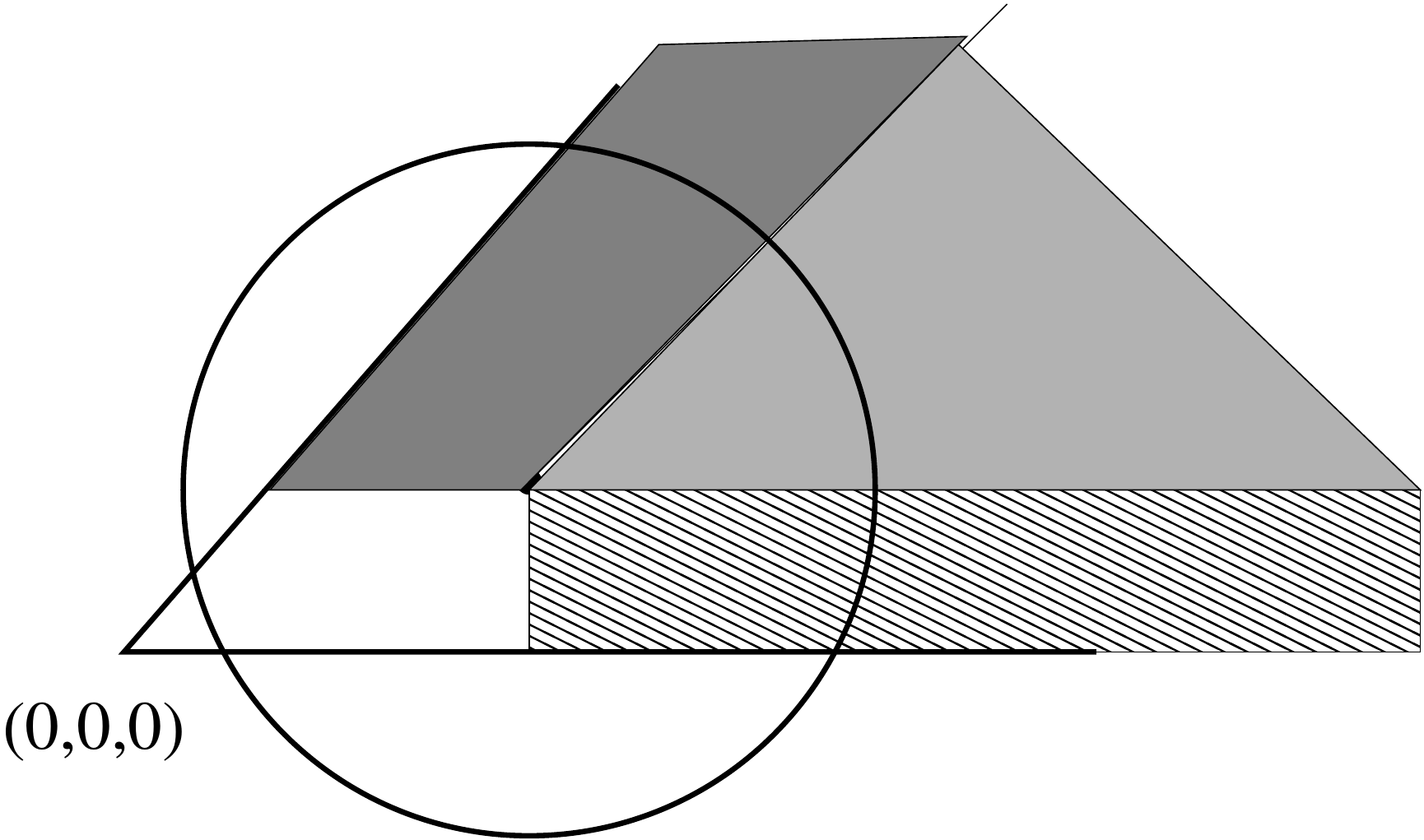}
		 \caption{We show only the 2nd and 3rd co-ordinates. 
		 In the proof of Lemma \ref{wdglblem},
		 we lower bound the volume of the intersection of
		 the circle (representing a ball of unit radius)
		 with the leftmost wedge shown, 
		 by the sum of the volumes of the
		 union of the three shaded regions, each intersected
		 with the ball/circle.}
\end{figure*}
Also 
$ |B(\bx,r) \cap  \bW_\alpha | = r^{3} |B(r^{-1} \bx,1) \cap  \bW_\alpha | $
by scaling, and it is then straightforward to deduce
(\ref{wdglb}). 
\qed
\end{proof}

In the next three lemmas, we take $I$ to be an arbitrary
fixed compact interval in $\R$, and then, given $\alpha \in (0,2\pi)$
and $b \geq 0$, set 
$$
\bW_{\alpha, b} := 
\{ (x_1,x_2,x_3)  \in \bW_\alpha: x_1 \in I,
x_2^2 + x_3^2 \leq b^2\}.
$$ 
\textcolor{\blue}{Also, writing $I = [a_1,a_2]$, let $I'_t:=
[a_1 - r_t,a_2+ r_t] $ (a slight extension of the interval $I$).
Let $\bW'_{\alpha,b,t} :=
\{ (x_1,x_2,x_3)  \in \bW_\alpha: x_1 \in I'_t,
x_2^2 + x_3^2 \leq b^2\}.
$} 

The next lemma will help us to show that if the part of $\partial A$
near to a given edge is covered, then the interior of $A$ near that
edge is also likely to be covered.

\begin{lemm}
\label{lemwedge}
Let $\alpha \in [\alpha_1,\pi)$, $a \in (1,\infty)$.
Then
\bea
\lim_{t \to \infty}
(\Pr[F_t((\partial \bW_\alpha ) \cap \bW_{\alpha,  a r_t},  \cW_{\alpha,t})   \setminus
F_t( \bW_{\alpha, a r_t}, \cW_{\alpha, t} ) ] )
= 0.
\label{0517b2a}
\eea 
\end{lemm}
\begin{proof}
	\textcolor{\blue}{First, note that 
	$\kappa(\bW'_{\alpha,ar_t,t} \setminus \bW_{\alpha,ar_t},r_t)
	= O(1)$ as $t \to \infty$, and hence by
	taking the measure $\mu_t$ 
	to be 
	$f_0 | \cdot \cap \bW_{\alpha}|$
	in Lemma} \ref{lem2meta},
	\bea
	\textcolor{\blue}{
	\lim_{t \to \infty} 
	\Pr[F_t (
	\bW'_{\alpha,ar_t,t} \setminus \bW_{\alpha,ar_t},\cW_{\alpha,t})]
	=1.}
	\label{eqwedcov}
	\eea

Let $E_t$ here  be the (exceptional)
 event that there exist three distinct points $\bx,\bfy,\bz$
of $\cW_{\alpha,t}$
such that $\partial B(\bx,r_t) \cap \partial B(\bfy,r_t) \cap \partial 
B(\bz,r_t)$
	has non-empty intersection with the plane $\R^2 \times \{0\} $. 
Then $\Pr[E_t]=0$.

Suppose that event
$F_t((\partial \bW_\alpha) \cap \bW_{\alpha, a r_t},  \cW_{\alpha,t})   
\setminus F_t( \bW_{\alpha,a r_t}, \cW_{\alpha,t} ) $
occurs, 
	\textcolor{\blue}{and that 
	$F_t ( \bW'_{\alpha,ar_t,t} \setminus \bW_{\alpha,ar_t})$ occurs,}
	and that $E_t$ does 
	not.
Let $\bw$ be a point of minimal height
	in the closure of $V_t(\cW_{\alpha,t}) \cap \bW_{\alpha,ar_t}$.
Then $\bw \in \bW_{\alpha,a r_t} \setminus \partial \bW_\alpha$, and 
$\bw =\bq_{r_t}(\bx,\bfy,\bz)$ 
for some triple 
$(\bx,\bfy,\bz)$ of
 points of $\cW_{\alpha,t} $, satisfying
$h_{r_t}(\bx,\bfy,\bz)=1$, where $h_{r_t}(\cdot)$ was defined at (\ref{hrdef2}).
 Also $\bw$ is 
 covered by fewer than $k$ of the balls centred on  the other points of
$\cW_{\alpha,t}$.
	Hence by Markov's inequality,
\bea
\Pr[F_t((\partial \bW_\alpha) \cap \bW_{\alpha,2 a r_t},  \cW_{\alpha,t})   
	\cap F_t(\bW'_{\alpha,ar_t,t}\setminus \bW_{\alpha,ar_t},\cW_{\alpha,t})
\setminus F_t( \bW_{\alpha,a r_t}, \cW_{\alpha,t} ) ] 
\nonumber
	\\ \leq \Pr[N_t \geq 1] \leq \E[N_t],
	~~~~~~
	\label{1127a}
\eea
	where we set
$$
N_t := \sum_{\bx,\bfy,\bz \in \cW_{\alpha,t}}^{\neq} h_{r_t}(\bx,\bfy,\bz)
 {\bf 1} 
\{ \bq_{r_t}(\bx,\bfy,\bz) \in \bW_{\alpha,ar_t} \}
	{\bf 1} \{ \cW_{\alpha,t} (B(\bq_{r_t}(\bx,\bfy,\bz),r_t) ) < k \},
$$ 
and $\sum^{\neq}$  means we are summing over triples of distinct points in
	$\cW_{\alpha,t}$.
By the Mecke formula, $\E[N_t]$ is bounded by a constant times
\bean
 t^3 (tr_t^3)^{k-1} \int_{\bW} d\bx
 \int_{\bW} d\bfy
 \int_{\bW} d\bz h_{r_t}(\bx,\bfy,\bz)
 {\bf 1}\{  \bq_{r_t}(\bx,\bfy,\bz) \in \bW_{\alpha,ar_t}  \}
\\
\times \exp(- t f_0 |B(\bq_{r_t}(\bx,\bfy,\bz),r_t) \cap \bW_{\alpha,ar_t}|). 
\eean
Hence by (\ref{wdglb}), 
there is a constant $c$ such that $\E[N_t]$
 is bounded by
\bea
c t^3 (tr_t^3)^{k-1} \int_{\bW_\alpha} d\bx
 \int_{\bW_\alpha} d\bfy
 \int_{\bW_\alpha} d\bz h_{r_t}(\bx,\bfy,\bz)
 {\bf 1}\{  \bq_{r_t}(\bx,\bfy,\bz) \in \bW_{\alpha,ar_t}  \}
\nonumber \\ 
\times
 \exp \left\{ - (2\alpha /3) f_0 t  r_t^3 - c' f_0 t r_t^2 
\pi_3( \bq_{r_t}(\bx,\bfy,\bz) ) \right. 
\nonumber \\
\left. - c' f_0 tr_t^2  [\pi_2( \bq_{r_t}(\bx,\bfy,\bz)) - 
\pi_3(\bq_{r_t}(\bx,\bfy,\bz))   \cot \alpha  ]
 \right\}.  
\label{0528c}
\eea

We claim that  (\ref{rt3}) implies
\bea
\exp(-(2 \alpha/3) f_0 t r_t^3 ) 
= O(t^{-1/3} (\log t)^{(1/3)-k}).
\label{0528b}
\eea 
Indeed, if $\alpha = \alpha_1 \leq \pi/2$,
then the expression on the left side of (\ref{0528b}) divided by the one
on the right, tends to a finite constant. Otherwise, this ratio tends to zero.

Now we write $c''$ for $c' f_0$. 
By (\ref{rt3}),  (\ref{0528b})
and the assumption that $I$ is a bounded interval,
we obtain that $\E[N_t]$
 is bounded by
a constant times
\bean
t^3 (\log t)^{k-1}  t^{-1/3}(\log t)^{(1/3)-k } 
\int_{\bW_\alpha} d\bz \int_{\bW_\alpha} d\bfy 
\int_0^{2 ar_t} du \int_0^{2 ar_t} dv h_{r_t}( (0,u,v),\bfy,\bz) 
\\ \times
 {\bf 1}\{  \bq_{r_t}((0,u,v),\bfy,\bz) \in \bW_\alpha  \}
\\
\times
 \exp \left\{  -c'' t r_t^2 
[v + (\pi_3( \bq_{r_t}((0,u,v) ,\bfy,\bz) ) -v) ]
\right\}
\\
\times \exp 
\left\{
- c'' t r_t^2  [\pi_2( \bq_{r_t}((0,u,v),\bfy,\bz)) 
- \pi_3(\bq_{r_t}((0,u,v),\bfy,\bz))    \cot \alpha  ]
 \right\} .  
\eean

Writing $\bu:= (0,u,v)$,
and changing variables to $\bfy'= r_t^{-1}(\bfy -\bu)$, 
and $\bz'= r_t^{-1}(\bz-\bu)$,  we obtain that 
the previous expression  is bounded by 
a constant times the quantity
\bean
b_t := t^{8/3} (\log t)^{-2/3 }    r_t^6
\int_0^{2 ar_t} dv
\int_{\R^3} d\bz' \int_{\R^3} d\bfy' 
\int_0^{2 ar_t} du 
f_t(u,v,\bfy',\bz') g_t(u,v,\bfy',\bz'),
\eean
where we set
\bean
f_t(u,v,\bfy',\bz'):= 
e^{- c''t r_t^2 v} \exp( - c'' t r_t^2  [
\pi_3( \bq_{r_t}(\bu, \bu + r_t \bfy', \bu + r_t \bz') ) -v ] )
\\
\times 
h_{r_t}(\bu, \bu + r_t \bfy', \bu + r_t \bz') ,
\eean
and
\bean
g_t(u,v,\bfy',\bz') := 
{\bf 1}   \{ 
\pi_2 (\bq_{r_t} (\bu, \bu + r_t \bfy', \bu + r_t \bz') )  
> \pi_3( \bq_{r_t} (\bu, \bu + r_t \bfy', \bu + r_t \bz') ) \cot \alpha 
\}
\\ \times 
\exp \{ - c'' t r_t^2 [ 
\pi_2 (\bq_{r_t} (\bu, \bu + r_t \bfy', \bu + r_t \bz')  )
- \pi_3( \bq_{r_t} (\bu, \bu + r_t \bfy', \bu + r_t \bz') ) \cot \alpha) ] \}.
\eean
Observe that
 $f_t(u,v,\bfy',\bz')$
 does not depend on $u$, so it is equal
to $f_t(0,v,\bfy',\bz')$.
Therefore we can take this factor
 outside the
  innermost integral. Also, setting $\bv:= (0,0,v)$,
we have
$$
\bq_{r_t}(\bu,\bu+ r_t \bfy',\bu+ r_t \bz') 
= \bq_{r_t}(\bv,\bv+ r_t \bfy',\bv+ r_t \bz') + (0,u,0), 
$$
Therefore, setting $\bw := \bw(t,v, \bfy',\bz') := \bq_{r_t}(\bv, \bv  + r_t \bfy',
\bv + r_t \bz') $,
we have
$$
g_t(u,v,\bfy',\bz') = 
{\bf 1}\{ \pi_2( \bw)  + u > \pi_3(\bw)  \cot \alpha_1 \}
\exp( - c'' t r_t^2  (
\pi_2( \bw)  + u - \pi_3(\bw)  \cot \alpha_1 ) ). 
$$ 
Defining 
$
u_0 := u_0 (t,v, \bfy',\bz') :=
\pi_3( \bw) 
\cot \alpha_1 -
\pi_2( \bw ) ,
$
 we obtain for the innermost integral that
\bean
\int_0^{2 a r_t}
g_t (u,v,\bfy',\bz') du = \int_{u_0}^{2 ar_t} \exp (- c'' t r_t^2
(u- u_0) ) du  =  O((tr_t^2)^{-1} ).
\eean
Also,
we have
\bean
e^{c'' t r_t^2 v} f_t(0,v, \bfy', \bz') =
\exp [- c'' t r_t^2 \pi_3(\bq_{r_t} (o,r_t \bfy',r_t \bz') ) ] 
 h_{r_t}(o, r_t \bfy', r_t \bz') 
\\
= \exp( - c'' t r_t^3 \pi_3(\bq_1(o,\bfy',\bz'))  ) 
h_1(o,\bfy',\bz')) 
\\
=: \tf_t(\bfy',\bz'). 
\eean
Combining all this, we obtain that $b_t$ is bounded by 
 a constant times
\bean
t^{8/3} (\log t)^{-2/3} r_t^6 (t r_t^2)^{-1}\int_0^{2 ar_t} 
e^{-c'' t r_t^2 v} dv
\int_{\R^3} \int_{\R^3} d\bz' d\bfy' \tf_t(\bfy',\bz') , 
\eean
and since the first integral is $O((tr_t^2)^{-1})$,
this is bounded by a constant times
\bean
(t r_t^3)^{2/3} (\log t)^{-2/3} \int_{\R^3}
 \int_{\R^3} \tf_t(\bfy',\bz') d\bfy' d\bz'.
\eean
By dominated convergence the last integral tends to zero,
while by (\ref{rt3}) the prefactor tends to a constant.
Thus
 $b_t \to 0$ as $t \to \infty$. 
 \textcolor{\blue}{
 Therefore also $\E[N_t] \to 0$, and
 combined with
	(\ref{1127a}) and (\ref{eqwedcov}), this yields
(\ref{0517b2a})}.
\qed
\end{proof}

The next lemma helps us to show that if a given edge of $A$ is 
covered, then the part of $\partial A$ near that edge is
also likely to be covered.

\begin{lemm}
\label{lemrind}
Suppose $\alpha_1 \leq \alpha \leq \pi/2$. Let $a >1$.
	\textcolor{\blue}{Let $\bF := \partial \bH := \R^2 \times \{0\}$.}
Then $\Pr[F_t(\bW_{\alpha,0}, \cW_{\alpha,t} ) \setminus 
F_t(\bF \cap \bW_{\alpha,a r_t},\cW_{\alpha,t}) ] 
\to 0$ as $t \to \infty$. 
\end{lemm}
\begin{proof}
Given $\bx,\bfy \in \R^3$, and $r>0$,
if $\#( \partial B(\bx,r) \cap \partial B(\bfy,r) \cap \bF)=2 $, 
 then denote the two points of
 $ \partial B(\bx,r) \cap \partial B(\bfy,r) \cap \bF $, 
taken in increasing lexicographic order, by  $\tp_{r}(\bx,\bfy)$   and
  $\tq_{r}(\bx,\bfy)$.  
Define the indicator functions 
\bean
\th_r^{(1)}(\bx,\bfy): = {\bf 1} \{ 
\pi_2(\bfy) \geq \pi_2(\bx) \}
{\bf 1} \{
\# ( \partial B(\bx,r) \cap \partial B(\bfy,r) \cap \bF ) =2 \}
\\
\times {\bf 1} 
\{ \pi_2(\tp_{r}(\bx,\bfy)) > \pi_2(\bx) \}
; \\
\th_r^{(2)}(\bx,\bfy): = {\bf 1} \{ 
\pi_2(\bfy) \geq \pi_2(\bx) \}
{\bf 1} \{
\# ( \partial B(\bx,r) \cap \partial B(\bfy,r) \cap \bF ) =2 \}
\\
\times {\bf 1} 
\{ \pi_2(\tq_{r}(\bx,\bfy)) > \pi_2(\bx) \}.
\eean

Let $E_t$ here denote the (exceptional) event that there exist
two distinct points $\bx,\bfy \in \cW_{\alpha,t}$, such that $\partial B(\bx,r_t)
\cap \partial B(\bfy,r_t) \cap \bW_{\alpha,0} \neq \emptyset$.
Then $\Pr[E_t]=0$.

Suppose $F_t(\bW_{\alpha,0}, \cW_{\alpha,t}) \setminus 
F_t(\bF \cap \bW_{\alpha,a r_t},\cW_{\alpha,t}) $  
 occurs, and $E_t$ does not.
 Let 
	$\bw$ be a location in 
	$\overline{V_t(\cW_{\alpha,t})} \cap \bF$ of minimal depth (i.e., minimal second coordinate).
Then $\bw$ must lie at the intersection of the boundaries of  balls of
radius $r_t$ centred on points $\bx,\bfy \in \cW_{\alpha,t}$,
 that is, 
at $\tp_{r_t}(\bx,\bfy)$ or at $\tq_{r_t}(\bx,\bfy)$. Moreover $\bw$
must be covered by at most $k-1$ other balls centred on points of
$\cW_{\alpha,t}$. Also $\pi_2(\bw) > 
\min(\pi_2(\bx),\pi_2(\bfy)) $; otherwise, for sufficiently small $\delta$ the 
location
	$\bw+ (0,-\delta,0)$ would be in $\overline{  V_t ( \cW_{\alpha,t}})
	\cap \bF$
 and have a smaller depth than $\bw$.
	Finally, to have $\bw \in \bW_{\alpha, ar_t}$ we need
	$\pi_2(\bw) \leq a r_t$ and hence $\min(\pi_2(\bx),\pi_2(\bfy)) \leq a$.
Therefore
\bea
\Pr[ F_t(\bW_{\alpha,0}, \cW_{\alpha,t}) \setminus 
F_t(\bF \cap \bW_{\alpha,a r_t},\cW_{\alpha,t})  ] \leq  
\Pr[N_t^{(1)} \geq 1 ] + 
\Pr[N_t^{(2)} \geq 1 ], 
\label{0524a}
\eea 
where, \textcolor{\blue}{with $\sum_{\bx,\bfy \in \cW_{\alpha,t}}^{\neq}$
	denoting summation over ordered pairs of distinct points
	of $\cW_{\alpha,t}$,} 
we set
$$
N_t^{(1)} := \sum_{\bx,\bfy \in \cW_{\alpha,t}}^{\neq}
\th_{r_t}^{(1)}(\bx,\bfy) {\bf 1}
 \{ \cW_{\alpha,t}(B(\tp_{r_t}(\bx,\bfy), r_t  ) )
 < k+2 \} {\bf 1}\{\pi_2(\bx) \leq ar_t\};
$$
$$
N_t^{(2)} := \sum_{\bx,\bfy \in \cW_{\alpha,t}}^{\neq}
\th_{r_t}^{(2)}(\bx,\bfy) {\bf 1} \{ \cW_{\alpha,t}(B(\tq_{r_t}(\bx,\bfy), r_t
 ) ) < k+2 \} {\bf 1}\{\pi_2(\bx) \leq ar_t\}.
$$
By the Mecke formula, and 
(\ref{wdglb}) from 
Lemma \ref{wdglblem},
with $c':= c'(\alpha,a+1)$ we have that
$\E[N_t^{(1)}]$
is bounded by a constant times
\bean
 t^2 (t r_t^3)^{k-1} \int_0^{a r_t}  du \int_0^{r_t} dv
 \int_{\R^3} d \bfy
\th^{(1)}_{r_t} ((0,u,v),\bfy)
\exp( - (2 \alpha/3) f_0t r_t^3 
\nonumber \\
- c' f_0 t r_t^2 \pi_{2} (\tp_{r_t}((0,u,v),\bfy) ) 
),
\eean
and there is a similar bound for $\E[N_t^{(2)}]$,
 involving $\th_{r_t}^{(2)}$ and $\tq_{r_t}(\bx,\bfy)$.

Consider  $\E[ N_t^{(1)}]$ (we can treat $\E[N_t^{(2)} ]$ similarly),
and write $c''$ for $c'f_0$. By (\ref{0528b}),
  $\E[N_t^{(1)}]$   is bounded by
a constant times
\bea
 t^2 (tr_t^3)^{k-1}  t^{-1/3}
(\log t)^{(1/3)-k}
\int_0^{a r_t}  du
\int_0^{r_t} dv
 \int_{\R^3} d\bfy
\th^{(1)}_{r_t} ((0,u,v),\bfy)
\nonumber \\
\times  
\exp( - c'' t r_t^2 \pi_{2} (\tp_{r_t}((0,u,v),\bfy) ) ).
\label{0523b}
\eea
Set $\bu := (0,u,v)$ and $\bv:= (0,0,v)$. 
We shall now make the changes of variable 
 $u'= r_t^{-1} u$, $v'= r_t^{-1} v$ and 
$\bfy' = r_t^{-1} (\bfy - \bu)$. Also write
$\bu':= (0,u',v') = r_t^{-1} \bu$ and $\bv' := (0,0,v')$.
Then 
\bean
\tp_{r_t}(\bu,\bu + r_t \bfy') 
= r_t \tp_1(\bu',\bu' + \bfy') = r_t \tp_1(\bv',\bv' + \bfy') + r_t(0,u',0).
\eean
Also
\bean
\th_{r_t}^{(1)} (\bu,\bu + r_t \bfy') = 
\th_{1}^{(1)} (\bu' ,\bu' + \bfy')  =
\th_{1}^{(1)} (\bv' ,\bv' + \bfy').  
\eean

By our changes of variable,
 the integral in (\ref{0523b})
comes to
\bean
r_t^5
\int_0^{a }  du'
\int_0^{1} dv'
 \int_{\R^3} d \bfy'
\th^{(1)}_{1} (\bv',\bv' +   \bfy')
\exp( - c'' t r_t^3 u'
 - c'' t r_t^3 \pi_{2} (\tp_{1}(\bv',\bv' + \bfy') ) ). 
\eean
Therefore since $\int_0^a e^{- c'' t r_t^3 u'} du' = O((tr_t^3)^{-1})$, 
\textcolor{\blue}{and $tr_t^3 = O(\log t)$ by (\ref{rt3a}),}
the expression in (\ref{0523b}) is bounded by a constant times
\bean
 t^{5/3}
 (\log t)^{-2/3} 
r_t^5 (t r_t^3)^{-1} 
 \int_0^1 dv'
\int_{\R^3}  d\bfy' \th^{(1)}_{1} (\bv',\bv' +\bfy') 
\exp(  -
 c'' tr_t^3   \pi_{2} (\tp_{1}(\bv',\bv' + \bfy') ) ) . 
\eean
For each $v' \in [0,1]$ the function
$\bfy' \mapsto \th_1^{(1)} (\bv',\bv' + \bfy')$
has bounded support and is zero whenever $\pi_2(\tp_1(\bv',\bv' + \bfy')) 
\leq 0$
(because $\pi_2(\bv')=0$). 
Therefore 
in the last displayed expression the integral tends to zero by dominated 
convergence, while the prefactor tends to a finite constant by (\ref{rt3a}). 

Thus  $\E[N_t^{(1)}] \to 0$, and one can show similarly that 
  $\E[N_t^{(2)}] \to 0$. Hence by Markov's inequality,
the expression  on the left hand side
of (\ref{0524a}) tends to zero. 
\qed
\end{proof}

\textcolor{\blue}{The next lemma enables us to
reduce the limiting behaviour of the
probability that a region near a given edge of
$A$ is covered, to that of the corresponding probability
for the edge itself.}

\begin{lemm}
	\label{corowedge}
	Let $a_0 >1$. If $\alpha \in [\alpha_1,\pi/2]$, then
	$\Pr[F_t(\bW_{\alpha,0},\cW_{\alpha,t})
	\setminus F_t(\bW_{\alpha, a_0 r_t},\cW_{\alpha,t})] \to 0$
	as $t \to \infty$.
	Moreover, if $\alpha \in (\pi/2,2\pi)$ then 
	$\Pr[F_t(\bW_{\alpha,a_0 r_t}, \cW_{\alpha,t})] \to 1$
	as $t \to \infty$.
\end{lemm}
\begin{proof}
	First suppose $\alpha \leq \pi/2$.
	It follows from Lemma \ref{lemrind} that 
	$\Pr[F_t(\bW_{\alpha,0},\cW_{\alpha,t})
	\setminus F_t( (\partial \bW_\alpha) \cap
	\bW_{\alpha, a_0 r_t},\cW_{\alpha,t})] \to 0$. Combined
	with Lemma \ref{lemwedge}, this shows that
	\linebreak
	$\Pr[F_t(\bW_{\alpha,0},\cW_{\alpha,t})
	\setminus F_t(\bW_{\alpha, a_0 r_t},\cW_{\alpha,t})] \to 0$

	Now suppose $\alpha > \pi/2$.
	\textcolor{\blue}{In this case
	some of the geometrical arguments
	used in proving Lemmas \ref{lemwedge} and \ref{lemrind} do not
	work. Instead we can use Lemma \ref{lem2meta},
	taking $\mu_t = f_0 | \cdot \cap \bW_{\alpha}|$, 
	taking $W_t $ of that result to be $ \bW_{\alpha, a_0 r_t}$,
	with $a = 2f_0 \alpha/3$ and $b=1$. By (\ref{rt3a}),
	we have $tr_t^d \sim c \log t$ with $c = 1/(f_0 \pi)$.
	Hence $a c  = 2 \alpha / (3 \pi) > 1/3 = b/d$,
	so application of Lemma \ref{lem2meta} yields the claim.} 
	\qed
\end{proof}

Using Lemma \ref{lemangle}, choose $K >4$ such that 
for any edge or face $\varphi$ of $A$, and any other face $\varphi'$
of $A$ with $\varphi \setminus \varphi' \neq \emptyset$,
and all $x \in \varphi^o$,
we have $\dist(x, \partial \varphi) \leq (K/3) \dist (x, \varphi')$.
 Denote the vertices of $ A$ by $q_1,\ldots,q_\nu$,
and the faces of $A$ by $H_1,\ldots,H_m$.
 Recall that we denote the edges of $A$ by $e_1,\ldots,e_\kappa$
 with the angle subtended by $A$ at $e_i$ denoted $\alpha_i$ for each $i$.
For $ 1 \leq i \leq \kappa$,
denote the length of 
 the 
 edge  $e_i$ by $|e_i|$.

  Define the `corner regions' 
\bea
Q_t := \cup_{i=1}^\nu B(q_i,K(K+7)r_t) \cap A; ~~~~~
Q_t^+ := \cup_{i=1}^\nu B(q_i,K (K+9)r_t) \cap A.
\label{eqcorner}
\eea
\begin{lemm} 
\label{lemcornpol3}
It is the case that
  $\Pr[F_t(Q_t^+,\Po_t) ] \to 1$ as $t \to \infty$.
\end{lemm}
\begin{proof}
	\textcolor{\blue}{
	As in the proof of Lemma \ref{lemQ},
		we can apply Lemma \ref{lem2meta}, taking $W_t= Q_t^+$,
	and $\mu_t= f_0 | \cdot \cap A|$, with $b=0$}.
\qed
\end{proof}
\textcolor{\blue}{We now determine the limiting
probability of coverage for any edge of $A$.}

\begin{lemm}
\label{lemportions}
Let $i \in \{1,\ldots,\kappa\}$. 
If  $\alpha_i= \alpha_1 \leq \pi/2$ then
 \bean
\lim_{t \to \infty}( \Pr[F_t(e_i\setminus Q_t,\Po_t) ] ) 
= \exp ( - (|e_i|/(k-1)!)  ( \alpha_1/32)^{1/3}3^{1-k}  e^{-  \beta/3} ),
\eean
while if $\alpha_i > \min(\alpha_1,\pi/2)$ then
 $\lim_{t \to \infty} (\Pr[F_t(e_i \setminus Q_t, \Po_t) ] ) = 1$.
\end{lemm}

\begin{proof}
The portions of balls
$B(x,r_t), x \in \Po_t$ that intersect the edge $e_i$
  form a spherical Poisson Boolean model
in 1 dimension with (in the notation of Lemma \ref{lemHall})
$$
\lambda = f_0 (\alpha_i/2) t  r_t^2; ~~~ \delta = r_t; ~~~ 
\alpha = \frac{\theta_3 \alpha_i/(2 \pi)}{(\alpha_i/(2\pi)) \pi} = 4/3.
$$

By   (\ref{rt3a}), as $t \to \infty$ we have
$
t f_0 r_t^2 = 
 (2 \alpha_1 \wedge \pi )^{-2/3} (t f_0)^{1/3} (\log t)^{2/3}(1+o(1)),
$ 
so that
\bean
\log \lambda = (1/3)
\log( \alpha_i^3/(8(2\alpha_1 \wedge \pi)^2) ) + (1/3) \log (t f_0) + (2/3) \lglg t + o(1),
\eean
and $\lglg \lambda = \lglg t - \log 3 + o(1)$.
 Therefore
\bean
\alpha \delta \lambda - \log \lambda 
- (k-1) \lglg \lambda =
(2/3) f_0 \alpha_i t r_t^3 
- (1/3) \log ( \alpha_i^3/(8(2 \alpha_1 \wedge \pi)^2)) 
\\
- (1/3) \log ( t f_0) - (2/3) \lglg t - (k-1) \lglg t + (k-1) \log 3 + o(1), 
\eean
so that by (\ref{rt3a}) again, 
\bean
\lim_{t \to \infty} ( \alpha \delta \lambda - \log \lambda - (k-1)
 \lglg \lambda ) 
~~~~~~~~~~~~~~~~~~~~~~~~~~~~~~~~~~~~~~~
\\
=
\begin{cases} 
\beta/3 - (1/3) \log ( \alpha _1/32) + (k-1) \log 3 & \mbox{ if }
\alpha_i = \alpha_1 \leq \pi/2
\\
+ \infty  & \mbox{ otherwise. }
\end{cases}
\eean
We then obtain the stated results by application of Lemma \ref{lemHall}.
\qed
\end{proof}

Recall that $H_1,\ldots,H_m$ are the faces of $\partial A$.
For $t >0$ we define the `edge regions'
\bea
W_t &:=  &\cup_{i=1}^\kappa (e_i \oplus B(\bfo,(K+1)r_t) )\cap A ;
~~~~~
W_t^- := \cup_{i=1}^\kappa (e_i \oplus B(\bfo,Kr_t) ) \cap A;
\nonumber
\\
W_t^+  & :=  &\cup_{i=1}^\kappa (e_i \oplus B(\bfo,(K+4)r_t) ) \cap A.
\label{eqedge}
\eea
The next lemma provides the limiting probability that
the `interior regions' of all of the faces of $A$ are covered.
\textcolor{\blue}{Recall that
$|\partial_1 A|$ denotes the total area of all faces of $A$.}

\begin{lemm}
\label{lemintface}
Define event
 $G_t:= F_t(\partial A \setminus W_t^+,\Po_t).$
	It is the case that
\bea
\lim_{t \to \infty} (
\Pr[G_t] ) = \begin{cases}
	\exp \left(- c_{3,k}    |\partial_1 A| 
	e^{- 2 \beta/3} \right)  & \mbox{\rm if } 
\alpha_1 > \pi/2,   ~
 \mbox{\rm or } \alpha_1 = \pi/2, k=1
\\
1 & \mbox{\rm otherwise. }
\end{cases}
\label{0412a}
\eea
\end{lemm}
\begin{proof}
We claim that
 the events $F_t(H_1 \setminus W_t^+ ,\Po_t),\ldots,F_t(H_m 
\setminus W_t^+, \Po_t)$ are
	mutually independent. Indeed, for $i,j \in \{1,\ldots,m\}$ with
	$i \neq j$, if $x \in H_i \setminus
	W_t^+$ then $\dist(x, \partial H_i) \geq Kr_t$ so by our choice of $K$, 
	$\dist(x, H_j) \geq 3 r_t$, so the $r_t$-neighbourhoods
	of $H_1 \setminus W_t^+, \ldots,$ $ H_m \setminus W_t^+$ are 
	disjoint,
	and the independence follows.
	Therefore
$$
\Pr[ G_t ] = \prod_{i=1}^m \Pr[ F_t(H_i \setminus W_t^+, \Po_t) ].
$$
By (\ref{0517c}) from
 Lemma \ref{lemhalf3}
and an obvious rotation, for $1 \leq i \leq m$ we have
\bean
\lim_{t \to \infty} (\Pr[ F_t(H_i \setminus W_t^+,\Po_t) ] ) 
= \begin{cases}
	\exp \left( - c_{3,k}  |H_i| e^{-2 \beta/3} 
  \right) & \mbox{ if } \alpha_1 > \pi/2, ~ 
\\ & \mbox{ or } 
 \alpha_1 = \pi/2, k=1
\\
1 & \mbox{ otherwise, }
\end{cases}
\eean
and the result follows. 
\qed
\end{proof}

Next we estimate the probability that there is an uncovered region 
  near to a  face of $A$ but
 not touching that face, and not close to any edge of $A$.

\begin{lemm}
\label{lemEt2}
Define event
	$E_t^{(2)} = F_t(\partial A \cup W^+_t ,\Po_t) \setminus
F_t(A \setminus A^{(3r_t)}, \Po_t)$. 
	Then $\Pr[E_t^{(2)}] \to 0$ as $t \to \infty$.
\end{lemm}
\begin{proof}
For $1 \leq i \leq m$, let $S_{t,i}$, respectively $S_{t,i}^+$,
 denote the slab of thickness
$3 r_t$ consisting of all locations in  $A$ lying at a perpendicular
distance at most $3 r_t$ from the reduced face $H_i \setminus W_t$,
respectively from  $H_i \setminus W_t^-$.

	We claim that $S_{t,i}^+ \setminus S_{t,i} 
	\subset W_t^+$, for each $i \in \{1,\ldots,m\}$.
	Indeed, given $w \in S_{t,i}^+ \setminus
	S_{t,i}$, let $z$ be the nearest point in $H_i$ to $w$.
	Then $\|w-z\| \leq 3r_t$. Also
	$z \notin H_i \setminus W_t$, so $z \in W_t$.
	Therefore for some $j \leq \kappa$
	we have $\dist(z,e_j) \leq (K+1) r_t$, so $\dist(w,e_j)
	\leq (K+4) r_t$, so $w \in W_t^+$, justifying the claim.

	Suppose $E_t^{(2)} $ occurs.
	Let   $x \in V_t (\Po_t) \cap A \setminus A^{(3 r_t)}$.
	Choose $i \in \{1,\ldots,m \}$ and $y \in H_i$ such
	that $\|x-y \| = \dist(x,H_i) \leq 3 r_t$.
	Then since $x \notin W_t^+$, for all $j$ we have
	$$
	\dist (y, \partial H_j) \geq \dist(x,\partial H_j) 
	- 3r_t > (K+1) r_t,
	$$
	so $y \in H_i \setminus W_t$ and $x \in S_{t,i}$.
	Thus $V_t (\Po_t) $
intersects 
	 the slab $S_{t,i}$.
However, it does not intersect the face $H_i$, and nor does
it intersect 
	the set $S_{t,i}^+ \setminus S_{t,i},$ 
 (since by the earlier claim this set is contained in $W_t^+$ 
which is fully covered).  Hence by the union bound
$$
	\Pr[ E_t^{(2)} ] \leq \sum_{i=1}^m \Pr[ F_t((H_i \setminus W_t^{-})
\cup (S_{t,i}^+ \setminus S_{t,i}), \Po_t )  
\setminus F_t( S_{t,i}, \Po_t ) ].
$$
By (\ref{0517b}) from Lemma \ref{lemhalf3}, along with
an obvious rotation, each term in the above sum tends to zero.
	This gives us the result. 
	\qed
\end{proof}

Next we estimate the probability that there is an uncovered
region within distance 
$Kr_t$ of a $1$-dimensional edge of $\partial A$
 but not including any of that edge itself.

\begin{lemm}
\label{lemflaps}
It is the case that $\Pr[F_t( \cup_{i=1}^\kappa e_i \cup Q_t^+,\Po_t) 
\setminus F_t (W_t^+, \Po_t) ] \to 0$ as $t \to \infty$.
\end{lemm}
\begin{proof}
Let $i \in \{1,\ldots, \kappa\}$.
For $t >0$, let 
	$W^*_{i,t}$
denote the set of locations in $A$ at perpendicular distance at most
	$(K+4)r_t$ from the reduced edge $e_i \setminus Q_t$. 
	We claim that
	\bea
	(e_i  \oplus B(o,(K+4)r_t)) \cap A \setminus Q_t^+ \subset 
	W_{i,t}^*.
	\label{0104a}
	\eea
	Indeed, suppose 
	$x \in (e_i  \oplus B(o,(K+4)r_t)) \cap A \setminus Q_t^+$.
	Let $y \in e_i$ with $\|x-y\|= d(x, e_i)$.
	Then $\|x-y\| \leq (K+4)r_t$, and since $x \notin Q_t^+$,
	for all $j \in \{1\ldots,\nu\}$
	\bean
	\dist(y,q_j) & \geq & \dist(x,q_j) - (K+4)r_t
	\\
	& > & [K(K+9) - (K+4)] r_t 
	\geq K(K+7) r_t.
	\eean
	Therefore $y \notin Q_t$,
	so $x \in W^*_{i,t}$, demonstrating
	the claim.
Hence
\bean
\Pr[F_t (e_i \cup Q_t^+,\Po_t) \setminus F_t( 
	(e_i \oplus B(o,(K+4)r_t) ) \cap A ,\Po_t) ] 
	\\
\leq \Pr[F_t(e_i,\Po_t) \setminus F_t(W^*_{i,t} \setminus Q_t^+ ,\Po_t)],
\eean
	and \textcolor{\blue}{we claim that this probability tends to zero.
	Indeed,  if 
	 $x \in W_{i,t}^*
	\setminus Q_t^+$, then taking $x' \in e_i$
	 with $\|x-x'\| = \dist(x,e_i)$,
	 we have $\dist(x',\partial e_i) \geq 
	 \dist(x,\partial e_i) - (K+4)r_t \geq K(K+7)r_t$,
	 and so by the choice of $K$, for any face 
	 $\varphi' $ of $A$, other than the two faces
	 meeting at $e_i$, we have
	 $\dist(x',\varphi') > 3(K+7)r_t$, and hence
	 $\dist(x,\varphi') \geq (2K+17) r_t$.
	Then the claim follows by Lemma \ref{corowedge} and an obvious rotation.}
 Since 
\bean
F_t( \cup_{i=1}^\kappa e_i \cup Q_t^+,\Po_t) 
\setminus F_t ( W_t^+, \Po_t) 
~~~~~~~~~~~~~~~~~
~~~~~~~~~~~~~~~~~
~~~~~~~~~~~~~~~~~
\\
\subset  \cup_{i=1}^{\kappa} 
[F_t (e_i \cup Q_t^+,\Po_t)  \setminus F_t( 
	(e_i \oplus B(o,(K+4)r_t)) \cap A ,\Po_t) ],
\eean
the result follows by the union bound.
	\qed
\end{proof}

\begin{proof}[Proof of Theorem \ref{thwkpol3}]
First we estimate the probability 
of event  $F_t(A^{(3r_t)},\Po_t)^c $, that there is an uncovered
 region in $A$ distant more than $3r_t$ from $\partial A$.
	\textcolor{\blue}{
	We apply  Lemma \ref{lem2meta},
	with $\mu_t = \mu$, $W_t = A \setminus A^{(3r_t)}$,
	 $b=d$,  $a= f_0 \theta_3$, and  
	$c= 1/(f_0 \min(\pi, 2 \alpha_1))$.
	Then $(b/d) - ac = 1 - (4 \pi /(3 \min(\pi,2 \alpha_1))) <0$,
	so by Lemma \ref{lem2meta}},
\bea
\lim_{t \to \infty} ( \Pr[F_t(A^{(3r_t)},\Po_t)] ) = 1.
\label{0519a}
\eea

	Let $\partial^*A := ([(\partial A) \setminus W^+_t]
	\cup \cup_{i=1}^\kappa e_i ) \setminus Q_t$ (note that the definition
	depends on $t$ but we omit this from the notation).
	This  is the union of reduced faces and reduced edges  of $\partial A$.

We  have the event inclusion 
$
 F_t(A,\Po_t) 
\subset
 F_t(\partial^* A,\Po_t)  
$,  and
by the union bound 
\bean
\Pr[ F_t(\partial^* A,\Po_t)  \setminus F_t(A,\Po_t) ]  
	\leq  \Pr[F_t(Q_t^+)^c] +
	\Pr[ F_t( \cup_{i=1}^\kappa e_i \cup Q_t^+,\Po_t)  \setminus
	F_t(W^+_t,\Po_t) ]
 \nonumber \\
+
 \Pr[ F_t(\partial A \cup W_t^+,\Po_t)  \setminus F_t(A \setminus A^{(3r_t)},
\Po_t) ]  
+  \Pr[ F_t(A^{(3r_t)},\Po_t)^c ],
\eean
and the four probabilities on the right tend to zero 
by Lemma \ref{lemcornpol3}, Lemma \ref{lemflaps},
Lemma \ref{lemEt2} and (\ref{0519a}) respectively.
Therefore
\bea
\lim_{t \to \infty} \Pr[ F_t(A,\Po_t) ] =
\lim_{t \to \infty} \Pr[F_t(\partial^* A, \Po_t)],
\label{limlim1}
\eea
provided the limit on the right exists.

Next we claim that the events
$F_t(e_i \setminus Q_t,\Po_t)$, $1 \leq i \leq \kappa$,
are mutually independent. 
	Indeed, for distinct $i,j \in \{1,\ldots,\kappa\}$,
	if $x \in e_i \setminus
	Q_t$ then 
	$\dist(x, \partial e_i) \geq K(K+7)r_t$ 
	by (\ref{eqcorner}),
	so
	by our choice of $K$, 
	$\dist(x, e_j) \geq 3 r_t$ for $t$ sufficiently large.
	Therefore  the $r_t$-neighbourhoods
	of $e_1 \setminus Q_t, \ldots, e_\kappa \setminus Q_t$ 
	are disjoint,
	and the independence follows.
Hence by Lemma
\ref{lemportions},
\bea
\lim_{t \to \infty}
 (\Pr [ F_t(\cup_{i=1}^\kappa e_i \setminus Q_t, \Po_t) ] ) 
\nonumber \\
= \begin{cases}
\exp \left( - \sum_{\{i: \alpha_i = \alpha_1\}} 
(|e_i|/(k-1)!) (\alpha_1/32)^{1/3} 3^{1-k} e^{- \beta/3} \right)
&~ {\rm if} ~ \alpha_1 \leq \pi/2  
\\
1 & ~ {\rm otherwise}. 
\end{cases}
~~~
\label{0525b}
\eea
Observe next that
	by the definition of $W_t^+$ at (\ref{eqedge}),
	the set $\partial A \setminus W_t^+$
	is at Euclidean distance at least $K r_t$ from all of the
	edges $e_i$. Therefore the events 
$F_t( \cup_{i=1}^\kappa e_i \setminus Q_t, \Po_t) $ and
$ F_t (\partial A \setminus W_t^+,\Po_t) $ are independent. Therefore  
using (\ref{limlim1})
we have
\bean
\lim_{t \to \infty} \Pr[ F_t(A,\Po_t) ] 
= \lim_{t \to \infty} \Pr 
[F_t( \cup_{i=1}^\kappa e_i \setminus Q_t, \Po_t) ]
\times
 \lim_{t \to \infty} \Pr 
[ F_t (\partial A \setminus W_t^+,\Po_t) ],
\eean
provided the two limits on the right exist. But we know from (\ref{0525b}) 
and Lemma \ref{lemintface} that these two limits do indeed exist,
and what their values are. Substituting these two
values, 
we obtain the
result stated for $R'_{t,k}$ in (\ref{0529a}).
Then we obtain the 
result stated for $R_{n,k}$ in (\ref{0529a}) by applying Lemma \ref{depolem}.
\qed
\end{proof}

\subsection{{\bf Proof of Theorem \ref{thsmoothgen}: first steps}}

\label{secfirststeps}

In this subsection we assume that $d \geq 2$ and  $A$ has $C^2$ boundary.
Let $\zeta \in \R $, and
assume   that $(r_t)_{t >0}$ satisfies  (\ref{rt3c}).
Let $k \in \N$.
Given any point process $\X$ in $\R^d$, and any $t >0$,
define the `vacant' region $V_t(\X)$ by (\ref{Vtdef}),
and given also $D \subset \R^d$, define $F_t(D,\X)$,
the event that
$D$ is  `fully covered' $k$ times, by
(\ref{Ftdef00}).

\textcolor{\blue}{
Given $x \in \partial A$ we can
express $\partial A$
locally in a neighbourhood of $x$,
after a rotation, as the graph of a $C^2$ function with
zero derivative at $x$.
As outlined  in earlier  in Section \ref{secstrategy},
we shall approximate to that function
by the graph of a piecewise affine function (in $d=2$, a piecewise
linear function).}

For each $x \in \partial A$, we can find an open neighbourhood
$\NN_x$ of $x$, a number $r(x) >0$ such that
$B(x, 3r(x)) \subset \NN_x$ and  a rotation $\rho_x$ about $x$
such that
$\rho_x(\partial A \cap \NN_x)$ is the graph of
a real-valued $C^2$ function $f$ defined on an open \textcolor{\blue}{ball} 
 $D \subset \R^{d-1}$, with
$\langle f'(u),e\rangle \leq 1/9$ for
all $u \in D$ and
all unit vectors $e$ in $\R^{d-1}$, where $\langle \cdot,\cdot \rangle$
denotes the Euclidean inner product 
in $\R^{d-1}$ and $f'(u):= (\partial_1 f(u),\partial_2f(u), \ldots, \partial_{d-1} f(u) )$ is the derivative of $f$ at $u$.
Moreover, by taking a smaller neighbourhood if necessary, 
we can also assume that there exists $\eps>0$ and
$a  \in \R$ such that 
$f(u) \in [a+ \eps,a+ 2 \eps]$ for all $u \in D$
and also  $\rho_x(A) \cap (D \times [a,a + 3 \eps])
= \{(u,z): u \in D, a \leq z \leq f(u)\}$.

By a compactness argument, we can and do take a finite collection
of points $x_1,\ldots, x_J \in \partial A$ such that
\bea
\partial A \subset  \cup_{j=1}^J B(x_j,r(x_j)). 
\label{bycompactness}
\eea
Then there are  constants $\eps_j >0$, and
rigid motions $\rho_j, 1 \leq j \leq J$, 
such that for each $j$ 
the set $\rho_j(\partial A \cap \NN_{x_j}) $ is
 the graph of a $C^2 $ function $f_j$ 
defined on a ball $I_j$ in $\R^{d-1}$, with
$\langle f'_j(u), e \rangle \leq 1/9$ for all $u \in I_j$ and
all unit vectors
$e \in \R^{d-1}$, and also with $\eps_j \leq f_j(u) \leq 2 \eps_j $
for all $u \in I_j$ and $\rho_j(A) \cap (I_j \times [0,3\eps_j]) =
\{(u,z):u \in I_j, 0 \leq z \leq f(u)\}$.

 Let $\Gamma \subset \partial A$ be
 a closed set such that $\Gamma \subset  B(x_j,r(x_j))$
 for some $j \in \{1,\ldots,J\}$, and such that
 \textcolor{\blue}{
 $\kappa(\partial \Gamma,r) = O(r^{2-d})$ as $r \downarrow 0$,
 where in this section we set
 $\partial \Gamma: = \Gamma \cap \overline{\partial A \setminus \Gamma} $,
 the boundary of $\Gamma$ relative to
 $\partial A$ (the $\kappa$ notation was given at (\ref{covnumdef})).} 
 To simplify notation we shall assume that 
 $\Gamma \subset B(x_1,r(x_1))$, and moreover that $\rho_1$ is the identity map. 
 Then $\Gamma= \{(u,f_1(u)): u \in U\}$ for some bounded set $U \subset \R^{d-1}$.
 Also,
 writing $\phi(\cdot)$ for $f_1(\cdot)$ from now on, we assume
 \bea
 \phi(U) \subset [\eps_1,2 \eps_1]
 \label{0901b}
 \eea
 and
 \bea
 A \cap (U \times [0,3 \eps_1]) = \{(u,z): u \in U, 0 \leq z \leq \phi(u) \}.
 \label{0901a}
 \eea

Note that for any $u,v \in U$, by the mean
value theorem we have for some $w \in [ u,v]$ that
\bea
|\phi(v) - \phi(u) | = | \langle v-u, \phi'(w)\rangle | \leq
(1/9) \|v-u\|.
\label{philip}
\eea

\textcolor{\blue}{
Choose (and keep fixed for the rest of this paper) constants $ \gamma_0, \gamma, \gamma'$
with
\bea
1/(2d) < \gamma_0 < \gamma < \gamma'  < 1/d.
\label{eqgamma}
\eea
The use of these will be for $t^{-\gamma}$,  $t^{-\gamma'}$ and $t^{-\gamma_0}$
to provide length scales that are 
different from each other and from that of $r_t$}.

When $d=2$, we approximate to $\Gamma$ by a polygonal line $\Gamma_t$ with
edge-lengths that are $\Theta(t^{-\gamma})$.
When $d=3$,
we approximate to $\Gamma$ by a polyhedral surface $\Gamma_t$ with all of its
vertices in $\partial A $,  and face diameters
 that are $\Theta(t^{-\gamma})$,  taking all the faces of
 $\Gamma_t$ to be triangles.  
 For general $d$, we wish to approximate to $\Gamma$ by a set $\Gamma_t$ given
 by a union of the $(d-1)$-faces in a certain simplicial complex of
 dimension $d-1$ embedded in $\R^d$.

 To do this,
 divide $\R^{d-1}$ into 
 cubes of dimension $d-1$ and side $t^{-\gamma}$, and divide each of these cubes
 into $(d-1)!$ simplices (we take these simplices to be closed).
 Let $U_t$ be the union  of all those
 simplices in the resulting tessellation of 
 $\R^{d-1}$ into simplices,  that are
 contained within $U$, and let $U_t^-$ be the
 union of those simplices in the tessellation which are contained
 within $U^{(3dt^{-\gamma})}$, where for $r>0$ we
 set $U^{(r)}$ 
 to be the set of $x \in U$ 
 at a Euclidean distance 
 more than $r$ from $\R^{d-1} \setminus U$.
 If $d=2$, the simplices
 are just intervals. 
 See Figure 3.

 \textcolor{\blue}{
 Let $\sigma^-(t)$, respectively
 $\sigma(t)$,} denote the number of simplices making up
 $U_t^-$, respectively $U_t$. 
 Choose $t_0 >0$ such that $\sigma^-(t) >0$ for all $t \geq t_0$.

\begin{figure}[!h]
\label{fig0}
\center
\includegraphics[width=8cm]{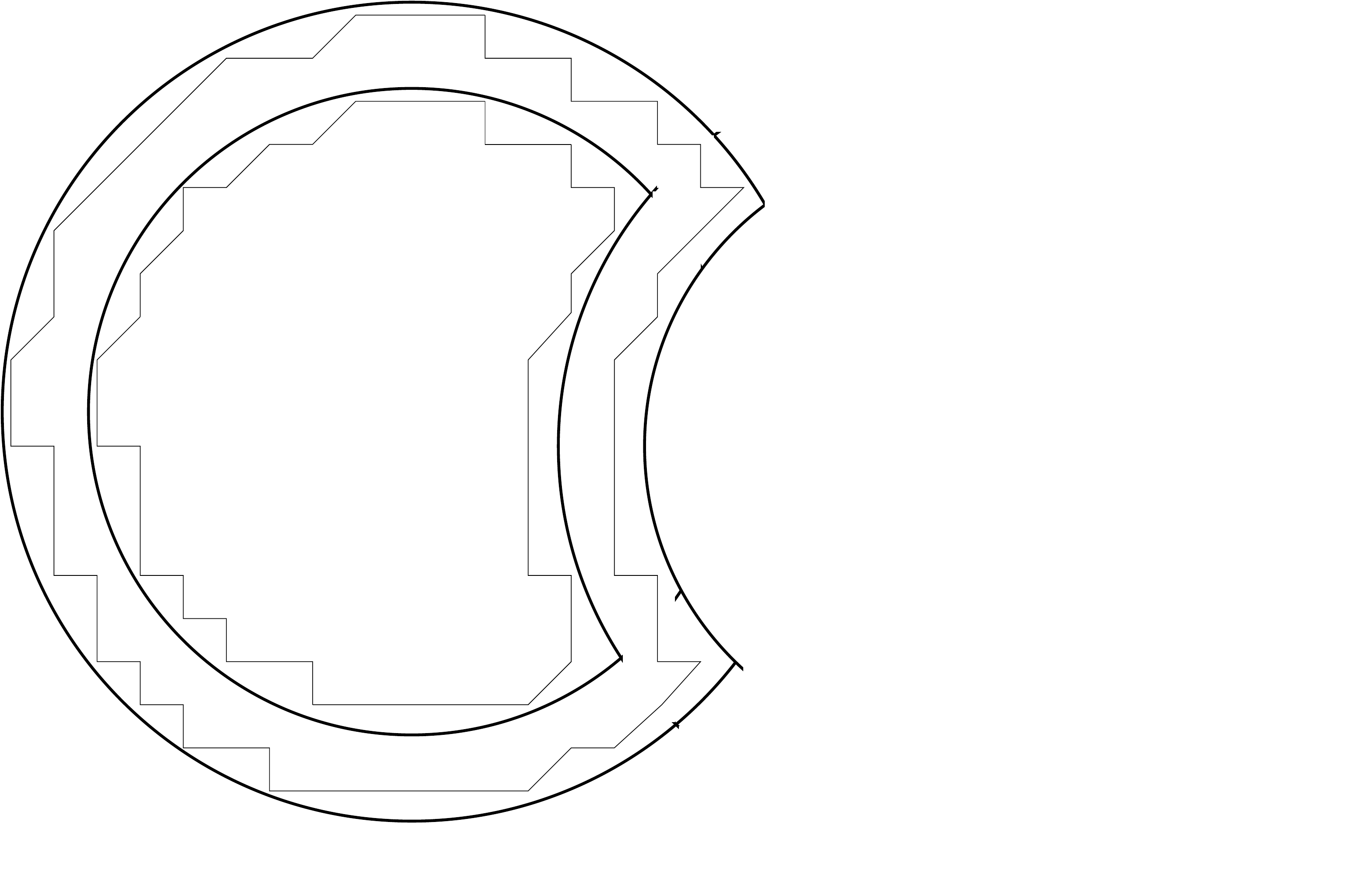}
\caption{Example in $d=3$.
	The outer crescent-shaped region is $U$,
	while the inner crescent is $U^{(3dt^{-\gamma})}$
	(the annular region is not to scale: its thickness should
	be 9 times the length of the shortest edges of the polygons).
	The outer polygon is $U_t$, while
	the inner polygon is $U_t^-$
	}
\end{figure}
 
  Let  $\psi_t: U_t \to \R$ be the function that is
  affine on each of the simplices
 making up $U_t$, and agrees with the function $\phi$ on each
 of the vertices of these simplices.
 Our approximating  surface (or polygonal line if $d=2$)
 will be defined by
  $\Gamma_t := \{(x, \psi_t(x)-K t^{-2 \gamma}): x \in U^-_t\}$, with 
  the constant $K$ given by the following lemma.
This lemma uses Taylor expansion to show that $\psi_t$
a good approximation to $\phi$. 

\begin{lemm}[\textcolor{\blue}{Polytopal approximation of
	$\partial A$}]
\label{lemtaylor3}
Set 
	$K :=
	\sup_{t \geq t_0, u \in U_t} ( t^{2 \gamma} 
	|\phi(u) - \psi_t(u) |)$.
	Then $K < \infty$. 
\end{lemm}
\begin{proof}
	Set $K_0 :=  d^3 \sup\{
	|\phi''_{\ell m}(v)| 
	, \ell, m \in \{1,\ldots, d-1\}, v \in U\}$, i.e. $d^3$ times
the supremum over all $v \in U$ of the max-norm of the Hessian of $\phi$ at $v$.
Then $K_0 < \infty$.

Given $t \geq t_0$,
	denote the simplices making up $U_t$ 
	by $T_{t,1},\ldots,T_{t,\sigma(t)}$.
	Let 
	\linebreak 
	$i \in \{1,\ldots,\sigma(t)\}$.
	Let  $u_0,u_1,\ldots, u_{d-1} \in \R^{d-1}$ be the vertices
	of $T_{t,i}$.

	Let $u \in T_{t,i}$. Then $u$ is a convex combination
	of $u_0,\ldots, u_{d-1}$;
	write $u= u_0 + \sum_{j=1}^{d-1} \alpha_j(u_j - u_0) $
	with $\alpha_j \geq 0$ for each $j$ and
	$\sum_{j=1}^{d-1} \alpha_j \leq 1$.
	Set $v_0 = u_0$, and for $1 \leq k \leq d-1$, 
	set $v_k :=  u_0 + \sum_{j=1}^k  \alpha_j(u_j-u_0)$.
	Then by the mean value theorem, for $1 \leq k \leq d-1$,
	 there exists $w_k \in T_{t,i}$
	 (in fact, $w_k \in [v_{k-1}, v_k]$), such
	that 
	\bean
	\phi(v_k) - \phi(v_{k-1}) = \langle v_k-v_{k-1} , \phi'(w_k) \rangle
	= \alpha_k \langle u_k-u_0 , \phi'(w_k) \rangle . 
	\eean
	Also, since $\psi_t$ is affine on $T_{t,i}$ and  agrees with
	$\phi$ on $u_0, u_1, \ldots, u_{d-1}$,
	there exists $\tiw_k \in T_{t,i}$ (in fact, $\tiw_k \in [u_0,u_k]$) 
	such that
	\bean
	\psi_t (v_k) - \psi_t (v_{k-1}) & = &
	 \alpha_k (\psi_t(u_k) - \psi_t(u_0) ) 
	 \\
	 & = & \alpha_k (\phi(u_k) - \phi(u_0) ) 
	= \alpha_k \langle u_k-u_0 , \phi'(\tiw_k) \rangle . 
	\eean
	Hence,
	\bean
	\phi(v_k) - \phi(v_{k-1}) - 
	(\psi_t(v_k) - \psi_t(v_{k-1}) ) 
	= \alpha_k \langle u_k-u_0, \phi'(w_k) - \phi'(\tiw_k) \rangle.
	\eean
	By the mean value theorem again, each
	component of $\phi'(w_k) -\phi'(\tiw_k)$ is bounded
	by $
	K_0 d^{-2}\|\tiw_k- w_k\|.
	$
	Therefore $\|\phi'(w_k) - \phi'(\tiw_k)\| \leq
	K_0  d^{-1} \|\tiw_k - w_k\|$.
	Since $\diam(T_{t,i}) = (d-1)^{1/2} t^{-\gamma}$,
	and since $u = v_{d-1}$ and $\phi(u_0) = \psi_t(u_0)$,
\bean
|\phi(u)- \psi_t(u)| =  \left| \sum_{k=1}^{d-1}
[\phi(v_k) - \phi(v_{k-1} ) - 
(\psi_t(v_k) - \psi_t(v_{k-1} ))] \right|
\\
\leq 
\sum_{k=1}^{d-1} \alpha_k \|u_k - u_0\| K_0 d^{-1} \|\tiw_k - w_k\|  
\\
\leq  \diam(T_{t,i})^2 K_0 d^{-1} 
\leq K_0  t^{-2 \gamma},
\eean
and therefore $K \leq K_0 < \infty$,
as required.
 \qed
 \end{proof}
 \textcolor{\blue}{We now subtract a constant from $\psi_t$ to obtain a
 piecewise affine function
 $\phi_t$ that approximates $\phi$ from below.}
 For $t \geq t_0$ and $u \in U_t$, define $\phi_t(u) := \psi_t(u) - 
 K t^{-2 \gamma}$, with $K$ given by  
 Lemma \ref{lemtaylor3}. Then
 for all $t \geq t_0, u \in U_t$ we have 
 $|\psi_t(u)-\phi(u)|\leq Kt^{-2\gamma}$ so that
 \bea 
	\phi_t(u) \leq  \phi(u) \leq \phi_t(u) + 2K t^{-2 \gamma}.
	\label{0126a}
\eea
 Define the set $\Gamma_t: =
 \{(u,\phi_t(u)): u \in U_t^-\}$.
 We refer to each $(d-1)$-dimensional face of $\Gamma_t$
 (given by the graph of $\phi_t$ restricted to one
 of the simplices in our triangulation of $\R^{d-1}$)
 as simply a {\em face} of $\Gamma_t$. 
 Denote these faces
 of $\Gamma_t$ by $H_{t,1}, \ldots, H_{t,\sigma^-(t)}$.
The number of faces, $\sigma^-(t)$,  
 is $\Theta(t^{(d-1)\gamma})$ as $t \to \infty$.
The perimeter (i.e., the $(d-2)$-dimensional Hausdorff measure of the boundary)
of each individual face is $ \Theta(t^{-(d-2)\gamma})$.

For $t \geq t_0$,
define subsets $A_t,A_t^-,
\tA_t,
A_t^{*},
A_t^{**}$
of
$ \R^d$
and Poisson processes $\Po'_t$ and $\tPo_t$ in $\R^d$  
by
\bea
A_t := \{(u,z): u \in U_t, 0
\leq z \leq \phi(u)\}, 
 ~~~
 \tA_t := \{(u,z): u \in U_t, 
 0 \leq z \leq \phi_t(u)\},
 \nonumber \\
A_t^- := \{(u,z): u \in U_t^-, 
0 \leq z \leq \phi(u)\}, 
~~~~~ ~~~~~ ~~~~~
 ~~~~~
 ~~~~~
 ~~~~~
 ~~~~~
 ~
 \nonumber \\
A_t^* := \{(u,z): u \in U_t^-, \phi_t(u) - (3/2) r_t \leq z \leq \phi(u)\},
 ~~~~~ 
 ~~~~~ 
 ~~~~~ 
 ~~
 ~~
 \nonumber \\
A_t^{**} := \{(u,z): u \in U_t^-, \phi_t(u) - (3/2) r_t \leq z \leq \phi_t(u)\},
 ~~~~~~~~ ~~~~~~~~
 ~
 \nonumber \\
 \Po'_t := \Po_t \cap A_t, ~~~ \tPo_t := \Po_t \cap \tA_t.
 ~~~~~~~~ ~~~~~~~~
 ~~~~~~~~ ~~~~~~~~ 
 ~~~~~~~ ~~~~~~~~ 
 \label{Pdtdef}
 \eea

Thus $A_t$ is a `thick slice' of $A$  near the boundary region $\Gamma$,
$\tA_t$ is an approximating region
having $\Gamma_t$ as  its upper boundary,
and $A_t^{*}$,
$A_t^{**}$
 are `thin slices' of $A$ also 
having $\Gamma$, respectively $\Gamma_t$, as upper boundary.
 By (\ref{0126a}), (\ref{0901b}) and  (\ref{0901a}),
$A_t^{**} \subset A_t^* \subset A_t^- \subset A_t \subset A$,
and $A_t^{**} \subset \tA_t \subset A_t$.
The rest of this
subsection, \textcolor{\blue}{and the next subsection,}
are devoted to proving the following
intermediate step towards a proof of Theorem \ref{thsmoothgen}.

\begin{prop}[Limiting coverage probability for approximating
	\textcolor{\blue}{polytopal} surface]
\label{lemsurf}
It is the case that
$
	\lim_{t \to \infty} \Pr[F_t(A_t^{**},\tPo_t) ] =
	\exp (-  c_{d,k} |\Gamma|  e^{-  \zeta} ), 
	$
	\textcolor{\blue}{where $|\Gamma|$ denotes the $(d-1)$-dimensional
	Hausdorff measure of $\Gamma$}.
\end{prop}

 The following corollary
 of Lemma \ref{lemtaylor3}
is a first step towards proving this.

	\begin{lemm}
		\label{corotaylor}
		(a) It is the case that
	$|A_t \setminus \tA_t| = O(t^{-2 \gamma})$ as  $t \to \infty$.

(b)
		Let $K$ be as given in Lemma \ref{lemtaylor3}. Then
	for all $t \geq t_0$ and 
	$x \in U_t^{(r_t)} \times \R$,
		$|B(x,r_t) \cap A_t \setminus \tA_t| \leq 2K \theta_{d-1} r^{d-1}_t t^{- 2 \gamma}$. 
\end{lemm}
\begin{proof}
	Since $|A_t \setminus \tA_t | = \int_{U_t} (\phi(u) - \phi_t(u)) du$, 
 where this is a $(d-1)$-dimensional Lebesgue integral,
	part (a)
	comes from 
	(\ref{0126a}).

For (b), let $x \in U_t^{(r_t)} \times \R$, and
let $u \in U_t^{(r_t)}$ be the projection of $x $ onto
the first $d-1$ coordinates. Then if $y \in B(x,r_t) \cap
A_t \setminus \tA_t$, we have $y = (v,s)$ with
$\|v-u\| \leq r_t$ and  $\phi_t(v) < s \leq \phi(v)$.
	Therefore using (\ref{0126a})  yields
$$
	|B(x,r_t) \cap A_t \setminus \tA_t| \leq \int_{B_{(d-1)}(u,r_t)}
	(\phi(v) - \phi_t(v)) dv \leq 2K \theta_{d-1} t^{-2\gamma} r_t^{d-1},
$$
where the integral is a $(d-1)$-dimensional Lebesgue integral.
 This gives part (b). 
	\qed
\end{proof}

\textcolor{\blue}{
The next lemma says that small balls centred in $A$, or very near to $A$,
have almost half of their volume in $A$.}

\begin{lemm}
\label{lemhs}
Let $\eps >0$.
	Then
	for all large enough $t$, all $x \in A_t^*$,
	and all $s \in [r_t/2,r_t]$,
	we have
$|B(x,s)  \cap  \tA_t|> (1- \eps) (\theta_d/2) s^d$.
\end{lemm}
\begin{proof}
	For all large enough $t$, all
	$x \in A_t^*$ and $s \in [r_t/2,r_t]$,
	we have $B(x,s )\cap A \subset A_t$,
	so $B(x,s) \cap A_t = B(x,s) \cap A$, 
	and hence by Lemma \ref{lemMyBk} and Lemma \ref{corotaylor}(b),
\bea
|B(x,s) \cap  \tA_t| 
	& = & |B(x,s) \cap A_t | - 
 |B(x,s ) \cap A_t \setminus \tA_t  | 
\nonumber 	\\
 & \geq & (1 - \eps/2) (\theta_d /2) s^d   
	- O( r_t^{d-1} t^{-2 \gamma}). 
	\nonumber 
\eea
	Since $t^{-2 \gamma} = o(r_t)$, this
	  gives us the result.
	  \qed
\end{proof}

Recall that $\bH$ and $\cU_t$ were defined just before Lemma \ref{lemhalf3a}.
\textcolor{\blue}{The next lemma provides a bound on the probability
that a region of diameter $O(r_t)$ within $A$ or $A_t^*$
is not fully covered. The bound
is relatively crude, but very useful for dealing 
with `exceptional' regions such as those near the boundaries
of faces in the polytopal approximation.}

\begin{lemm}
\label{lemcov3}
	Let $\delta \in (0,1)$, $K_1 >0$. Then as $t \to \infty$, 
\bea
	\sup_{z \in \R^d}
	\Pr[ F_t(B(z,K_1r_t) \cap A_t^* ,  \tPo_t)^c ]
	= O(t^{\delta - (d-1)/d}),
	\label{0816b}
	\\
\sup_{z \in \R^d} \Pr[ F_t(B(z,K_1r_t) \cap \bH, \cU_t )^c ]
	= O(t^{\delta - (d-1)/d}),
	\label{0816a}
\eea 
and
\bea
\sup_{z \in \R^d}
	\Pr[ F_t(B(z,K_1r_t) \cap A, \Po_t )^c ]
	= O(t^{\delta - (d-1)/d}).
	\label{0816c}
\eea
\end{lemm}
\begin{proof}
	\textcolor{\blue}{
	For (\ref{0816b}), it suffices to prove,  for any
	$(z_t)_{t >0}$ with $z_t \in \R^d $ for each $t $, that
	\bea
	\Pr[ F_t(B(z_t,K_1r_t) \cap A, \Po_t )^c ] = O(t^{\delta - (d-1)/d})
	~~~ {\rm as} ~ t \to \infty.
	\label{1128a}
	\eea
	To see this, we apply Lemma \ref{lem2meta}, 
	taking $\mu_t= f_0 | \cdot \cap 
	\tA_t|$, taking $W_t= B(z_t,K_1 r_t) \cap A_t^*$, with 
	$c= 2 (d-1)/(d \theta_d f_0)$ (using (\ref{rt3c})), $b=0$, and
	$a= (\theta_d f_0/2)(1-\delta)$
	(using Lemma \ref{lemhs}). Then
	$(b/d)- ac = - (1- \delta)(d-1)/d$. Application of Lemma \ref{lem2meta},
	taking $\eps = \delta/d$,
	shows that (\ref{1128a}) holds, and hence
	(\ref{0816b}).}

	The proofs of (\ref{0816a}) and (\ref{0816c})  are
	similar; for (\ref{0816c})
	we  have to use (\ref{lemMyBk}) directly, rather than using
	Lemma \ref{lemhs}. We omit the details. 
	\qed
\end{proof}

Let $\partial_{d-2} \Gamma_t  := \cup_{i=1}^{\sigma^-(t)} \partial_{d-2}
H_{t,i}$,
the union of all $(d-2)$-dimensional faces  in the boundaries 
of the faces making up
 $\Gamma_t$
\textcolor{\blue}{(the $H_{t,i}$ were defined just after (\ref{0126a})).
	Recall from (\ref{eqgamma}) that}
$\gamma' \in (\gamma,1/d)$.
For each $t >0$ 
define the sets $Q_t, Q^+_t \subset \R^d$  by
 \bea
Q_t  :=  ( \partial_{d-2} \Gamma_t \oplus B(o,7r_t)) 
;
~~~~~~~
Q^+_t  :=   
 (\partial_{d-2} \Gamma_t \oplus B(o,8dt^{- \gamma'}))
 \cap A_t^{*}.
 \label{Qplusdef}
\eea
\textcolor{\blue}{
Thus $Q_t^+$ is
a region near the corners of our polygon approximating
$\partial A$ (if $d=2$) or near the boundaries of the
faces of our polytopal surface approximating
$\partial A$ (if $d  \geq 3$). In the next lemma 
we show that $Q_t^+$ is  fully covered with high probability.} 
\begin{lemm}
\label{lemcorn3}
It is the case that
	$\Pr[F_t( Q^+_t, 
	 \tPo_t)] \to 1$ 
	as $t \to \infty$.
\end{lemm}
\begin{proof}
	Let $\eps \in (0,  \gamma'- \gamma)$. 
	For each face $H_{t,i}$ of $\Gamma_t$,  $1 \leq i \leq \sigma^-(t)$,
	we claim that we can
	take $x_{i,1},\ldots, x_{i,k_{t,i}} \in  H_{t,i}$
	 with $\max_{1 \leq i \leq \sigma^-(t)} k_{t,i}
	= O(t^{-\gamma'} t^{-(d-2)\gamma} /r_t^{d-1} )$,
	such that
	\bea
	( (\partial_{d-2} H_{t,i})
	\oplus B(o, 9dt^{-\gamma'}) )
	\cap H_{t,i}
	\subset \cup_{j=1}^{k_{t,i}} 
	B(x_{i,j},r_t).
	\label{0206a}
	\eea
	Indeed, we can cover $\partial_{d-2} H_{t,i}$
	by $O(t^{(d-2)(\gamma'-\gamma)})$ balls of radius
	$t^{-\gamma'}$, denoted $B_{i,\ell}^{(0)}$ say.
	\textcolor{\blue}{ 
	Replace each  ball $B_{i,\ell}^{(0)}$
	with a 
	ball $B'_{i,\ell}$ with the same centre as
	$B_{i,\ell}^{(0)}$ and with 
	 radius  $ 10 d t^{-\gamma'}$}. Then cover $B'_{i,\ell} \cap
	 H_{t,i}$
	 by $O((t^{-\gamma'} /r_t)^{d-1})$ balls of radius $r_t$.
	 Every point in the set on the left side of (\ref{0206a})
	 lies in one of these balls of radius $r_t$, and the claim
	 follows.

	 Given $x \in Q_t^+$, write $x=(u,h)$ with $u \in U_t^-$, $h \in \R$,
	 and set $y := (u, \phi_t(u))$. Then $\|y-x\| \leq 2 r_t$ and
	 there exists $i$ with $y \in H_{t,i}$. Take such an $i$.
	 Then $\dist(y, \partial_{d-2}H_{t,i})= \dist
	 (y, \partial_{d-2}\Gamma_t) \leq 9d t^{- \gamma'}$,
	 so $\|y- x_{i,j} \| \leq r_t$ for some  
	 $j \leq k_{t,i}$ by (\ref{0206a}). Therefore
	$Q_t^+ 
	\subset \cup_{i=1}^{\sigma^-(t)} \cup_{j=1}^{k_{t,i}}
	B(x_{i,j},3r_t)$, so by the union bound,
\bean
	\Pr [ F_t( Q^+_t ,  \tPo_t) ^c]
	\leq 
		\sum_{i=1}^{\sigma^-(t)} \sum_{j=1}^{k_{t,i}}
	\Pr[
	 F_t( B(x_{i,j}, 3 r_t) \cap A_t^*  ,
	  \tPo_t )^c]. 
	 \eean
	By Lemma \ref{lemcov3} 
	and the fact that $\sigma^-(t) = O(t^{(d-1) \gamma})$, 
	 the estimate in the last display  is 
	 $
	O(t^{\gamma - \gamma'} r_t^{1-d} 
	t^{\eps - (d-1)/d}) = O(t^{\eps +  \gamma -\gamma'}),
$
which tends to zero.
\qed
\end{proof}

\subsection{{\bf \textcolor{\blue}{
Induced coverage process and
	proof of Proposition \ref{lemsurf}}}}
\label{secinduced}

We shall conclude the proof of Proposition \ref{lemsurf} by means of
a device we refer to as the {\em induced coverage process.}
\textcolor{\blue}{This is obtained by taking the parts of $\tA_t$
near the flat parts of $\Gamma_t$, along with any  Poisson
points therein, and rearranging them into a flat
region of macroscopic size.}
The definition of this is somewhat simpler for $d=2$,
so 
for presentational purposes,
we shall  consider this case first. 

Suppose that $d=2$.  The closure of the set $\Gamma_t \setminus Q_t$
is a union of closed  line segments (`intervals') 
with total length $|\Gamma_t| - 14 r_t \sigma^-(t)$,
which tends to $| \Gamma|$ as $t \to \infty$
 because  $ \sigma^-(t) = O(t^{\gamma})$ and
$\gamma < 1/2$. Denote these intervals by $I_{t,1},\ldots,I_{t,\sigma^-(t)}$,
 taking $I_{t,i}$ to be a 
sub-interval of $H_{t,i}$ for $1 \leq i \leq \sigma^-(t)$.

For $1 \leq i \leq \sigma^-(t)$ define the closed
 rectangular strip (which we call a `block') 
$S_{t,i}$ (respectively  $S^+_{t,i}$)
of dimensions $|I_{t,i} | \times 2 r_t$ (resp. 
$|I_{t,i}| \times 4 r_t$)
 to consist  
of those locations inside $A_t$ lying within perpendicular distance at most
 $2r_t$ (resp., at most $4 r_t$) of $I_{t,i}$. That is,
\bean
S_{t,i} := I_{t,i} \oplus [o,2 r_t e_{t,i}],
~~~~
S^+_{t,i} := I_{t,i} \oplus [o,4 r_t e_{t,i}],
\eean
where $e_{t,i}$ is a unit vector perpendicular to $I_{t,i}$
pointing inwards into $\tA_t$ from $I_{t,i}$. 
Define the long interval $D_t := [0,  |\partial A_t| - 14 r_t  \sigma^-(t) ]$,
and the
long horizontal rectangular strips 
\bean
S_t:= D_t 
\times [0,2 r_t]; ~~~~~~
S^+_t:= D_t 
\times [0, 4r_t].
\eean
Denote the lower boundary of $S_t$ (that is, the set $D_t \times \{0\}$)
by $L_t$.

We shall verify in Lemma \ref{lemblox}
that $S_{t,1}^+,\ldots, S_{t,\sigma^-(t)}^+$ are disjoint.
Now choose   rigid motions $\rho_{t,i}$
of the plane, $1 \leq i \leq \sigma^-(t)$, such that
after applications of these rigid motions the blocks $S_{t,i}$ 
are lined up end to end to form the strip $S_t$, with the long edge $I_{t,i}$
of the block transported to part of the lower boundary $L_t$  of $S_t$.
In other words, choose the rigid motions so that the sets 
 $ \rho_{t,i}(S_{t,i})$, 
$ 1 \leq i \leq \sigma^-(t)$,
have pairwise disjoint interiors and their union is 
$S_t$, and also  $\rho_{t,i}(I_{t,i}) \subset L_t$ for 
$1 \leq i \leq \sigma^-(t)$ (see Figure 4).  This also implies that
$\cup_{i=1}^{\sigma^-(t)} \rho_{t,i}(S^+_{t,i}) = S_t^+$.

 \begin{figure*}[!h]
\label{fig1}
\center
\includegraphics[width=8cm]{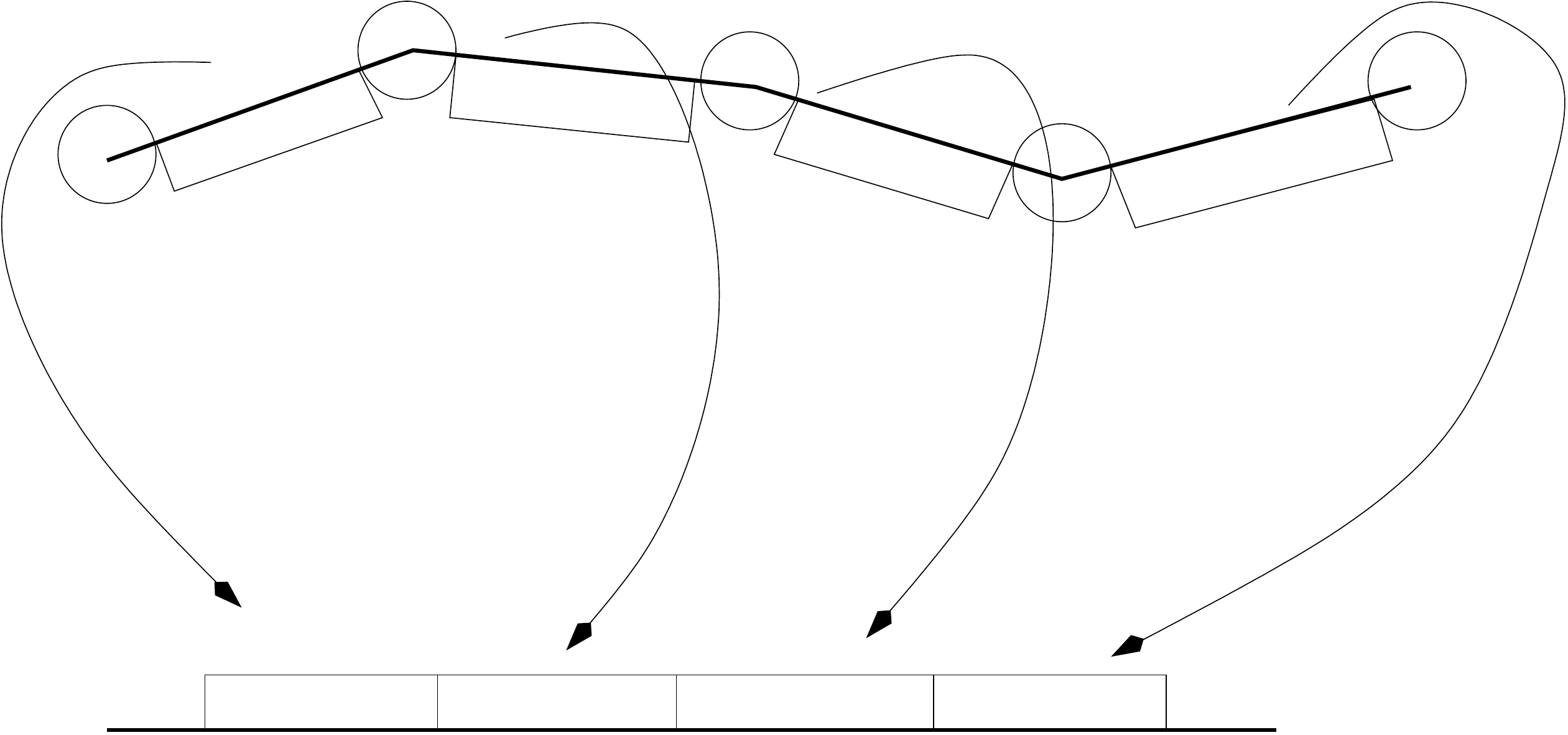}
\caption{The induced coverage process in $2$ dimensions.
	 The upper thick line is part of the set $\Gamma_t$,
	 while the lower thick line is part of the line $L_t$.
	 We show four of
	 the blocks $S_{t,i}$ next to the upper thick  line,
	 and four of the blocks
	 making up $S_t$ next to the lower thick line.
	 The arrows represent the rigid motions $\rho_{t,i}$.
	 The disks shown are part of $Q_t$.}
\end{figure*}

By the restriction, mapping and superposition theorems for Poisson
processes (see e.g. \cite{LP}),
the point process 
$\Po''_t:= \cup_{i=1}^{\sigma^-(t)} \rho_{t,i} (\Po'_t \cap S^+_{t,i}) $ is 
 a homogeneous Poisson process of intensity $t f_0$ in the
long strip $S_t^+$.

Now suppose   $d \geq 3$. To describe 
the induced coverage process in this case, we first
 define a `tartan' (plaid) region $T_t$ as follows.
Recall that  $\gamma' \in (\gamma,1/d)$.
Partition each face $H_{t,i}$, $1 \leq i \leq \sigma^-(t)$ into
a collection of
$(d-1)$-dimensional cubes of side $t^{- \gamma'}$ contained in
$H_{t,i}$, 
together with a boundary region contained within $\partial_{d-2} H_{t,i}
\oplus B(o,d t^{- \gamma'})$.
   Let $T_t$ be
the union (over all faces)  of the boundaries of the $(d-1)$-dimensional
cubes in this partition (see Figure 5).
Set
$T^+_t := [ T_t \oplus B(o,11r_t) ] \cap A_t^{**}$.

\begin{figure*}[!h]
\label{fig2}
\center
\includegraphics[width=8cm]{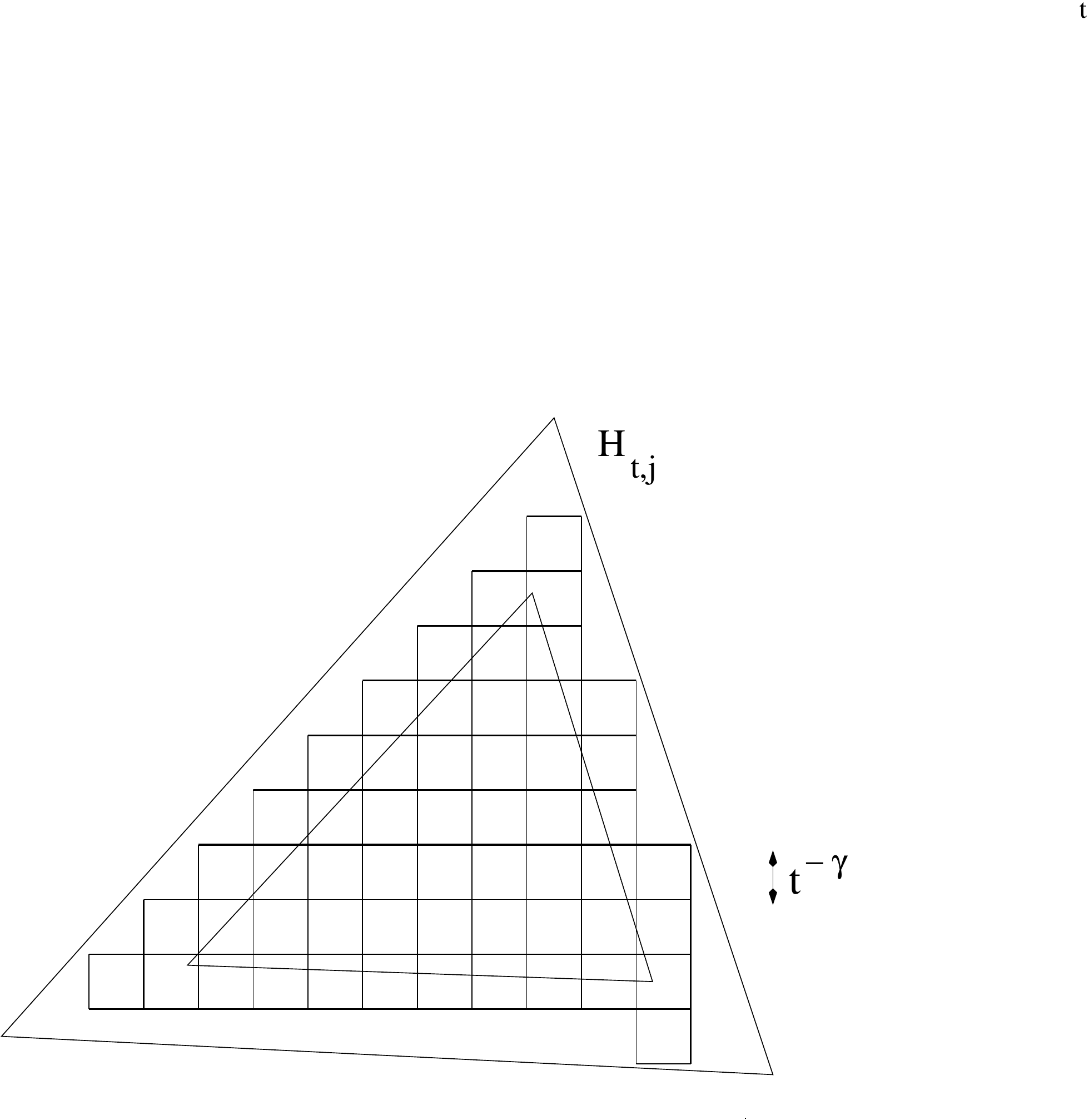}
\caption{Part of the `tartan' region $T_t$ when $d=3$.
	The outer triangle represents one face $H_{t,i}$,
	and the part of $T_t$ within $H_{t,i}$ is
	given by the union of the boundaries of the
	squares. The triangle has sides of length
	$\Theta(t^{-\gamma})$, while the squares
	have sides of length $t^{-\gamma'}$. The region
	between the two triangles is part of $Q_t^+$. It
	has thickness $8dt^{-\gamma'}$ (the constant $8d$ is
	not drawn to scale),
	and covers the whole boundary region
	not covered by the squares.}
\end{figure*}

  Enumerate the $(d-1)$-dimensional cubes in the above subdivision
  of the faces $H_{t,i}, 1 \leq i \leq \sigma^-(t)$, as
  $I^+_{t,1},\ldots, I^+_{t,\lambda(t)}$. For $1 \leq i \leq \lambda(t)$
  let $I_{t,i}:= I^+_{t,i} \setminus (T_t \oplus B(o, 7r_t))$, which is a 
  $(d-1)$-dimensional cube 
of side length $t^{-\gamma'} -14r_t$
  with the same centre and orientation as $I_{t,i}^+$.
 We claim that the total $(d-1)$-dimensional Lebesgue measure of these
 $(d-1)$-dimensional cubes
 satisfies
\bea
\lim_{t \to \infty}( | 
 \cup_{i=1}^{\lambda(t)} I_{t,i} | )
= 
|\Gamma|.
\label{0912a}
\eea
Indeed,  for $1 \leq i \leq \lambda(t)$ we have
$|I_{t,i}|/|I_{t,i}^+| = ((t^{-\gamma'} - 14r_t)/t^{-\gamma'})^{d-1}$,
which tends to one 
since $r_t = O( ( (\log t)/t )^{1/d })$ by (\ref{rt3c}), 
and  $\gamma' < 1/d$, so the proportionate amount removed
near the boundaries of the $(d-1)$-dimensional cubes $I^+_{t,i}$ to
give $I_{t,i}$ vanishes.
Also  the 
`boundary part' of a face
$H_{t,i}$ that is not contained in any of the 
of the $I^+_{t,j}$s has $(d-1)$-dimensional Lebesgue measure 
that is $O(t^{-(d-2)\gamma} t^{-\gamma'})$, so that
the total $(d-1)$-dimensional measure of 
the removed  regions near the boundaries of the faces is
$O(t^{(d-1)\gamma} \times t^{-(d-2)\gamma} t^{-\gamma'})=
O(t^{\gamma -\gamma'} )$, which tends to zero.
Thus the claim (\ref{0912a}) is justified.

For $1 \leq i \leq \lambda(t)$ define closed,
 rectilinear (but not aligned with the axes),
 $d$-dimensional cuboids  (which we call `blocks') 
$S_{t,i}$ (respectively  $S^+_{t,i}$),
as follows. 
Let $S_{t,i}$ (respectively $S^+_{t,i}$)
be the closure of the set of those locations $x \in \tA_t$ such
that $\dist(x,I_{t,i}) \leq  2 r_t $ (resp.,
$\dist(x,I_{t,i}) \leq  4 r_t $) and such that there exists $y \in I_{t,i}^o$
(the relative interior of $I_{t,i}$) 
satisfying $\|y-x\| = \dist(x, I_{t,i})$. For example, if $d=3$ then 
$S_{t,i}$ (resp. $S^+_{t,i}$) is a cuboid 
of dimensions $(t^{- \gamma'} - 14r_t)  \times (t^{-\gamma'} -14r_t) 
 \times 2 r_t$ 
 (resp.  $ (t^{- \gamma'} - 14r_t)  \times (t^{-\gamma'} -14r_t) \times 4 r_t
$)
 with $I_{t,i}$ as its base.
We shall verify in Lemma \ref{lemblox}
that $S_{t,1}^+,\ldots, S_{t,\sigma^-(t)}^+$ are disjoint.

Define a  region $D_t \subset \R^{d-1}$ that is approximately a rectilinear
hypercube with lower
left corner at the origin, and obtained as the union of $\lambda(t)$
disjoint $(d-1)$-dimensional 
cubes of  side $t^{-\gamma'} - 14 r_t$. We can and do arrange
that 
$D_t \subset [0,|\Gamma_t|^{1/(d-1)}+ t^{- \gamma}]^{d-1}$ for each $t$,
and $|D_t| \to |\Gamma|$ as $t \to \infty$.
Define the flat slabs
$$
S_t:= D_t  \times [0, 2 r_t]; ~~~~~~
S^+_t:= D_t  \times [0, 4r_t],
$$
and denote the lower boundary of $S_t$ (that is, the set
$D_t \times \{0\}$) by $L_t$.

Now choose   rigid motions $\rho_{t,i}$
of $\R^d$, $1 \leq i \leq \lambda(t)$, such that
under applications of these rigid motions the blocks $S_{t,i}$ 
are reassembled to form the slab $S_t$, with the square face
$I_{t,i}$
of the $i$-th block transported to part of the lower boundary $L_t$  of $S_t$.
In other words, choose the rigid motions so that the sets 
 $ \rho_{t,i}(S_{t,i})$, 
$ 1 \leq i \leq \lambda(t)$,
have pairwise disjoint interiors and their union is 
$S_t$, and also  $\rho_{t,i}(I_{t,i}) \subset L_t$ for 
$1 \leq i \leq \lambda(t)$.  This also implies that
$\cup_{i=1}^{\lambda(t)} \rho_{t,i}(S^+_{t,i}) = S_t^+$.

By the restriction, mapping and superposition theorems,
the point process $\Po''_t:= \cup_{i=1}^{\lambda(t)} \rho_{t,i} 
(\tPo_t \cap S^+_{t,i}) $ is 
 a homogeneous Poisson process of intensity $t$ in the flat slab
 $S_t^+$. 

 In both cases ($d=2$ or $d \geq 3$),
we extend $\Po''_t$ to a Poisson process $\cU'_t$
 in $ \bH := \R^{d-1} \times [0,\infty)$ as follows.
Let $\Po'''_t$ be a Poisson process of intensity $t f_0$ in
 $\bH \setminus S^+_t$,
independent of $\tPo_t$, and  set 
\bea
\cU'_t:= \Po''_t \cup \Po'''_t.
\label{Uprdef}
\eea
Then
$\cU'_t$ is a homogeneous Poisson process of intensity $tf_0$ in 
$\bH$.
 We call the collection of balls of radius $r_t$
centred on the points of this point process the {\em induced coverage process}.

 \textcolor{\blue}{The next lemma says that in $d \geq 3$,
 the `tartan' region
 $T_t^+$ is covered with high probability. It is needed because locations
 in $T_t^+$ lie near the boundary of blocks $S_{t,i}$, so that
 coverage of these locations by $\tPo_t $ does not necessarily
 correspond to coverage of their images in the induced coverage process.}

\begin{lemm}
	\label{lemtartan}
 Suppose $d \geq 3$. Then
	$ \lim_{t \to \infty} \Pr[ F_t(T^+_t, \tPo_t) ] =1.$  
\end{lemm}
\begin{proof}
	The total number of the $(d-1)$-dimensional cubes
	 (in the whole of $\Gamma_t$) in the partition described above
	is $O(t^{(d-1) \gamma'})$, and for each of these $(d-1)$-dimensional
	cubes the number of balls of
radius $r_t$ required
	to cover the boundary of the cube
	is $O((t^{-\gamma'} r_t^{-1})^{d-2})$.
	Thus we can take $x_{t,1}, \ldots, x_{t,k_t} \in \R^d$, with
	$k_t = O(t^{\gamma'} r_t^{2-d})$,  such that
	 $T_t \subset \cup_{i=1}^{k_t} B(x_{t,i},r_t)$.

	Then $T_t^+ \subset \cup_{i=1}^{k_t}
	B(x_{t,i},12r_t) \cap A_t^*$.  
	Let $\eps \in (0,(1/d)- \gamma')$.
	By Lemma \ref{lemcov3},
\bean
\Pr [ F_t(T^+_t,  \tPo_t) ^c] \leq
	\sum_{i=1}^{k_t} \Pr[F_t(B(x_{t,i},12r_t) \cap A_t^*, 
	 \tPo_t )^c]
	\\
	=
	O(t^{ \gamma'}  r_t^{2-d}  
	t^{\eps - (d-1)/d}) = O(t^{\eps + \gamma' -1/d}),
\eean
which tends to zero.
\qed
\end{proof}

So as to be able to treat the cases $d=2$ and $d \geq 3$ together,
for $d=2$
we define $\lambda(t):= \sigma^-(t)$ and $T_t^+ := T_t := \emptyset$
 (these were previously defined only for $d=3$).
 We verify next that the blocks $S_{t,i}^+, 1 \leq i \leq \lambda(t)$,
 are  pairwise disjoint. \textcolor{\blue}{This is needed to ensure
 that the Poisson processes $\tPo_t \cap S_{t,i}^+, 1 \leq i \leq
 \lambda(t)$, are mutually independent.}

 \begin{lemm}
	 \label{lemblox}
	 Suppose $d \geq 2$,   $t \geq t_0$ and 
	 $i,j \in \{1,2,\ldots,\lambda(t)\}$ with $i <j$. Then
	 $S_i^+ \cap S_j^+ = \emptyset$.
 \end{lemm}
 \begin{proof}
	 Suppose 
	 $S_i^+ \cap S_j^+ \neq \emptyset$; we shall obtain a contradiction.
	 Let $x \in S_i^+ \cap S_j^+ $.
	 Let $y$ be the closest point in $I_{t,i}$ to $x$, and 
	  $y'$ the closest point in $I_{t,j}$ to $x$.
	  Choose $k,\ell$ such that $I_{t,i} \subset H_{t,k}$
	  and $I_{t,j} \subset H_{t,\ell}$. Then $k\neq \ell$
	  since  if $d=2$ we have $k=i$ and $\ell =j$, while
	  if $d \geq 3$, if $k = \ell$ then clearly we would
	  have $S_i^+ \cap S_j^+ = \emptyset$.

	  Let $J_{t,k}$ be the projection of $H_{t,k}$ onto
	  $\R^{d-1}$, and write $y= (u,\phi_t(u))$ with
	  $u \in J_{t,k}$.
	  Let $v \in \partial J_{t,k}$. 
	  By (\ref{philip}) and (\ref{0126a}), 
	  \bean
	  |\phi_t(v)-  \phi_t(u) | \leq |\phi(v) -\phi(u)|
	  + 4 K t^{-2 \gamma} \leq
	  (1/9)\|v-u\| + 4 K t^{-2 \gamma}.
	  \eean
	  Since $y \in I_{t,i}$ we have $\|y - (v,\phi_t(v))\| 
	  \geq \dist(y,\partial H_{t,k}) \geq 7 r_t$, so that 
	  \bean
	  7r_t 
	  \leq \|u-v\| + |\phi_t(u) -
	  \phi_t(v)| \leq (10/9)\|u-v\| + 4Kt^{-2 \gamma},
	  \eean
	  and hence $\|u-v\| \geq (63/10) r_t -4Kt^{-2 \gamma}$,
	  and hence provided $t$ is large enough,
	  $
	  \dist(u, \partial J_{t,k} ) \geq 6 r_t.
	  $
	  Similarly, writing $y' := (u',\phi_t(u'))$ we have
	  $\dist(u', \partial J_{t,\ell}) \geq 6r_t$. Therefore
	  $$
	  \| y- y' \| \geq \|u- u'\| \geq 12 r_t.
	  $$
	  However, also $\|y-y'\| \leq \|y-x\| + \|y'-x\|  \leq 8 r_t$,
	  and we have our contradiction.
	  \qed
 \end{proof}

Denote the union of the boundaries (relative to $\R^{d-1} \times \{0\}$)
of the lower faces of
the blocks making up the strip/slab $S_t$, by $\tQ_t^0$,
and the $(9r_t)$-neighbourhood of this region by $\tQ_t$
\textcolor{\blue}{
(the $C$ can be viewed as standing for `corner region', at least
when $d=2$),}
i.e.
\bea
\tQ_t^0 
:= \cup_{i=1}^{\lambda(t)} \rho_{t,i} (\partial I_{t,i}),
~~~~~~~
\tQ_t := ( \tQ_t^0 \oplus B(o,9r_t) ) \cap \bH.
\label{tQdef}
\eea
Here $\partial I_{t,i}$ denotes the relative boundary of $I_{t,i}$.

 \textcolor{\blue}{The next lemma says that the corner region $\tQ_t$
 is covered with high probability. It is needed because locations
 in $Q_t^+$ lie near the boundaries of the blocks assembled to make
 the induced coverage process, so that coverage of these locations
 in the induced coverage process
 does not necessarily
 correspond to coverage of their pre-images in the original
 coverage process.}

\begin{lemm}
\label{lemtQ3}
Suppose $d \geq 2$. Then
	$\lim_{t \to \infty} 
	 \Pr[ F_t(\tQ_t ,\cU'_t) ] =1$.
\end{lemm}
\begin{proof}
	Set $\tg := \gamma$ if $d=2$, or if $d \geq 3$ set 
	$\tg := \gamma' $.
	Let $\eps \in (0, (1/d)- \tg)$.
	The number of $(d-1)$-dimensional cubes
	$\lambda(t)$ making up $L_t$ 
	is $O(t^{(d-1) \tg})$, and for each of these
	$(d-1)$ dimensional cubes,
	the number of balls of
radius $r_t$ required
	to cover the boundary is $O((t^{-\tg} r_t^{-1})^{d-2})$. Hence
	we can take points $x_{t,1}, \ldots, x_{t, m_t} \in L_t$,
	with $m_t = O(t^{\tg} r_t^{2-d})$,
	such that 
	$\tQ_t^0
	\subset \cup_{i=1}^{m_t} B(x_{t,i},r_t)$.
	Then $\tQ_t \subset \cup_{i=1}^{m_t} B(x_{t,i},10r_t)$.
	Hence by (\ref{0816a}) from
	Lemma \ref{lemcov3}, we
	obtain the estimate
$$
	\Pr [ F_t(\tQ_t \cap \bH, \cU'_t) ^c] =
	O(t^{\tg}  r_t^{2-d}  
	t^{\eps - (d-1)/d}) 
	= O(t^{\eps + \tg -1/d}),
$$
which tends to zero.
\qed
\end{proof}

\begin{lemm}[Limiting coverage probabilities for the induced coverage process]
	\label{leminduced}
	Suppose $d \geq 2$. Then
\bea
\lim_{t \to \infty} \Pr[F_t(S_t, \cU'_t)] 
= 
\lim_{t \to \infty} \Pr[F_t(L_t, \cU'_t)] 
	 = \exp(- c_{d,k} |\Gamma| e^{- \zeta} ) .
	 \label{0124c}
\eea
\end{lemm}

\begin{proof}
	The second  equality of (\ref{0124c}) is easily obtained using
	 (\ref{0517c2}) from Lemma \ref{lemhalf3a}.

	 \textcolor{\blue}{Recall that $L_t = D_t \times \{0\}$.}
	 Also $\partial L_t \subset \tQ_t^0$, so that
	  $(\partial D_t \oplus B_{(d-1)}(o,r_t)) \times [0,2 r_t]
	  \subset \tQ_t$,
	 and therefore  by (\ref{0517b2}) from Lemma \ref{lemhalf3a}, 
	 \bean
	 \Pr[( F_t(L_t, \cU'_t) \setminus F_t(S_t, \cU'_t) )
	 \cap F_t(\tQ_t, \cU'_t) ] \to 0.
	 \eean
	 Therefore using also
Lemma \ref{lemtQ3} shows that $\Pr[ F_t(L_t, \cU'_t) \setminus
	F_t(S_t, \cU'_t) ]
	\to 0$, and this gives us the rest of (\ref{0124c}).
	\qed
\end{proof}

\textcolor{\blue}{
	We are now ready to complete the proof of Proposition \ref{lemsurf}.}

 \begin{proof}[Proof of Proposition \ref{lemsurf}]
	 We shall approximate event $F_t (A^{**}_t  , \tPo_t)$ by 
	 events $ F_t(L_t, \cU'_t)$ and $F_t(S_t,\cU'_t)$,  
	 and apply Lemma \ref{leminduced}.

	Suppose 
	 $F_t(Q^+_t \cup T_t^+, \tPo_t) 
	 \setminus
	 F_t(A_t^{**}, \tPo_t) $
	occurs, and choose  
	 $x \in V_t(\tPo_t)  \cap A_t^{**} \setminus  (Q_t^+ \cup T_t^+)$. 
	Let $y \in \Gamma_t $ with $\|y-x\| = \dist (x,\Gamma_t)$.
	Then $\|y-x\| \leq 2r_t$,
	and since
	  $ x \notin Q_t^+ $, \textcolor{\blue}{by (\ref{Qplusdef})} 
	we have $\dist(x, \partial_{d-2} \Gamma_t) \geq 8d t^{-\gamma'}$,
	and hence 
	$\dist(y,\partial_{d-2}\Gamma_t) \geq 8dt^{-\gamma'} - 2 r_t
	\geq 7d t^{-\gamma'}$, 
	 provided $t$ is large enough.
	Therefore $y$ lies in the interior of the face $H_{t,i}$ for some
	$i$ and $x-y$ is perpendicular to $H_{t,i}$ (if $y \neq x$).
	 Also \textcolor{\blue}{(if $d \geq 3$),}
	 since
	 $x \notin T_t^+$,
	$\dist(x,T_t) \geq 11 r_t$, so
	$\dist(y,T_t) \geq 9 r_t$.
	Therefore $y \in I_{t,j}$ for some $j$, and
$x$ lies
	in the block $S_{t,j}$,  and moreover
	$\dist(x, (\partial S^+_{t,j}) \setminus I_{t,j} ) \geq 2r_t$.
	Hence $B(\rho_{t,j}(x),r_t) \cap \bH  \subset \rho_{t,j}(S^+_{t,j})$,
	 and hence \textcolor{\blue}{by (\ref{Uprdef})},
	$
	\cU'_t(B(\rho_{t,j}(x),r_t)) = 
	\tPo_t(B(x,r_t))< k,
	$
	 so event $F_t(S_t, \cU'_t) $ does not occur. Hence
$$
F_t(S_t, \cU'_t) \setminus F_t(A_t^{**}, \tPo_t) \subset
	 F_t(Q^+_t   \cup T^+_t, \tPo_t)^c,
$$
	 so  by Lemmas 
	 \ref{lemcorn3}  and  \ref{lemtartan},  
	 $\Pr[ F_t(S_t,\cU'_t) \setminus 
 F_t (A^{**}_t  , \tPo_t)  ] \to 0$, and hence using
 (\ref{0124c}) we have
	 \bea
	 \liminf_{t \to \infty} \Pr[F_t(A_t^{**},\tPo_t)] \geq 
	 \exp(-c_{d,k} |\Gamma|e^{-\zeta}).
 \label{0124d}
 \eea

	 Suppose $F_t (A^{**}_t , \tPo_t)  \setminus F_t(L_t,\cU'_t)$ occurs,
	 and
	 choose $ y \in L_t \cap V_t(\cU'_t)$.
	 Take $i \in \{1,\ldots,\lambda(t)\}$
	 such that $y \in \rho_{t,i}(I_{t,i})$.
	 Then 
	 $ \dist ( y, \rho_{t,i} (\partial I_{t,i})) \leq r_t $,
	 since otherwise $\rho_{t,i}^{-1}(y)$ would be
	 a location in $A^{**}_t \cap V_t( \tPo_t)$.
	 Thus $y \in \tQ_t$  \textcolor{\blue}{by (\ref{tQdef})}, and
	 therefore using Lemma \ref{lemtQ3} yields that
$$
\Pr [ F_t(A^{**}_t, \tPo_t) \setminus F_t(L_t,\cU'_t) ] \leq \Pr [ 
F_t(\tQ_t ,\cU'_t)^c ] \to 0 .
$$ 
Combining this with (\ref{0124c}) and (\ref{0124d})
completes the proof.
\qed
\end{proof}

\subsection{{\bf Proof of Theorem \ref{thsmoothgen}: conclusion}}

\label{seclast}

\textcolor{\blue}{Lemma \ref{lemsurf} gives
the limiting probability of coverage of a polytopal 
approximation to a region near part of $\partial A$.}
The next two lemmas show that
$\Pr[F_t(A^{**}_t, \tPo_t)]$ approximates
$\Pr[F_t(A^*_t,\Po'_t)]$ (recall the definitions at (\ref{Pdtdef})).
\textcolor{\blue}{From this we can deduce that
we get the same limiting probability 
 even after dispensing with the polytopal approximation.}

\begin{lemm}
\label{lemdiff4}
	Let $E^{(1)}_t:= F_t( A_t^{**}  , \tPo_t) \setminus F_t(A^*_t,\Po'_t)$. 
	Then
	$\Pr[E^{(1)}_t] \to 0$
	as $ t \to \infty$.
\end{lemm}
\begin{proof}
Let $\eps \in (0, 2\gamma - 1/d)$.
	Suppose $E^{(1)}_t \cap F_t(Q_t^+ , \Po'_t)$ occurs.
	Then
	since $\tPo_t \subset \Po'_t$, 
	$V_t(\Po'_t)$ intersects
	with $A^*_t \setminus A^{**}_t$, 
	and therefore by
	(\ref{0126a}), $V_t(\Po'_t)$ includes locations
distant at most  $2K t^{-2 \gamma}$ from $\Gamma_t$.
	Also $\Gamma_t \cap V_t (\Po'_t) = \emptyset$,
	since $\Gamma_t \subset A^{**}_t$.
 
	Pick a location $x \in \overline{V_t(\Po'_t) \cap A_t^*} $
	of minimal distance from
	$\Gamma_t$.  
	Then $x \notin Q_t^+$, so
	the nearest point in $\Gamma_t$ to $x$ lies in the
	interior of $H_{t,i}$ for some $i$.
	We claim that $x$ lies at the intersection of the boundaries of
$d$ balls of radius $r_t$ centred on points of $\Po'_t$;
	\textcolor{\blue}{
		this is proved similarly to the similar claim concerning
	$\bw$ in the proof of Lemma \ref{lemhalf3a}}.
	Moreover, $x$  is
	covered by \textcolor{\blue}{
		at most   $k-1 $ of the other balls centred in $\Po'_t$
	(in fact, by exactly $k-1$ such balls, but we do
	not need this below).} 
	Also $x$ does not lie in the interior of $A_t^{**}$.

	Thus if $E^{(1)}_t \cap F_t(Q_t^+ , \Po'_t)$ occurs, there
 must exist $d$ points $x_1,x_2,\ldots, x_d$
	of $\Po'_t $ such that $\cap_{i=1}^d
	\partial B(x_i,r_t)$ 
includes a point in $A_t^*$ but  outside the interior of 
	$A_t^{**}$, within distance $2Kt^{-2 \gamma}$ of 
$\Gamma_t$
 and further than $r_t$ from all but
	at most 
	$k-1$ other points of $\Po'_t$.
Hence by the Mecke formula,  integrating first over
 the positions of $x_2-x_1, x_3-x_1, \ldots,x_d - x_1$ 
and then  over the location
	of $x_1$, and using Lemma \ref{corotaylor}(a),
	Lemma \ref{lemhs} and (\ref{rt3d}),  we obtain for
	suitable constants $c, c'$ that
\bean
\Pr[E^{(1)}_t \cap F_t(Q_t^+ , \Po'_t)] 
	& 	\leq & c t^d r_t^{d(d-1)}
	t^{-2 \gamma}
	(t r_t^d)^{k-1} \exp(- (1-\eps) ( \theta_d/2 ) f_0 t r_t^d)
\\
	& \leq & c' (tr_t^d)^{k-1} t^{d +
	\eps -  2 \gamma - (d-1)/d } r_t^{d(d-1)} 
	\\
	& = & O( (tr_t^d)^{d+ k-2} t^{(1/d)- 2 \gamma + \eps}),
\eean
which tends to zero by (\ref{rt3c}). Also
$\Pr[ F_t(Q_t^+ , \Po'_t)]  \to 1$ by Lemma \ref{lemcorn3}, so 
$\Pr[E^{(1)}_t] \to 0$.  
\end{proof}

\begin{lemm}
\label{lemdiff3}
Let 
	$G_t :=  F_t(A_t^*, \Po'_t)  \setminus  F_t(A_t^{**}, \tPo_t)$.
	Then $\lim_{t \to \infty} \Pr[ G_t]=0$.
\end{lemm}
\begin{proof}
	If event $G_t$ occurs, then
	since $A^{**}_t \subset A_t^*$, 
	 there exists $y \in A^{**}_t \cap V_t(\tPo_t)$
	such that
$y $ is covered by at least one of the balls of radius $r_t$
 centred on $\Po'_t \setminus
\tPo_t$.
	Hence there exists $x \in \Po'_t \setminus
	\tPo_t$ with $B(x,r_t) \cap V_t(\tPo_t) \cap A^{**}_t \neq 
\emptyset$.
	Therefore  
\bea
G_t \subset  F_t( \cup_{x \in \Po'_t \setminus \tPo_t}
	B(x,r_t) \cap A_t^{**},   \tPo_t)^c .
	\label{0210a}
\eea

Let $\eps \in (0,2 \gamma - 1/d)$.
Let
	$ \cQ_t: = \Po'_t \setminus \tPo_t$. Then
	$\tPo_t $ and $\cQ_t$ are independent homogeneous Poisson processes
of intensity $t f_0$ in $\tA_t$, $A_t \setminus \tA_t$
respectively.
 By Lemma \ref{lemcov3} 
and the union bound,
there is a constant $c$ such that for
 any $m \in \N$ and any set of $m$ points $x_1,\ldots,x_m$ in $\R^d$, we have
\bean
\Pr \left[ F_t( \cup_{i=1}^m
	B({x_i},r_t) \cap A_t^{**} ,\tPo_t )^c  
\right]  
	\leq c m  t^{\eps - (d-1)/d}.
\eean
Let $N_t := \cQ_t(\R^d)$.
By Lemma \ref{corotaylor}(a), $\E[ 
N_t] = O(t^{1- 2 \gamma})$, so that
by conditioning on $\cQ_t$ we have for some constant $c'$ that
\bean
\Pr[ F_t( \cup_{x \in \cQ_t }
	B(x,r_t) \cap  A_t^{**} , \tPo_t )^c] 
	\leq c t^{\eps - (d-1)/d} \E[N_t]
\leq c' t^{(1/d)- 2 \gamma + \eps },
\eean
which tends to zero by the choice of $\eps$.
	Hence by (\ref{0210a}), $\Pr[G_t ] \to 0$.
	\qed
\end{proof}

 \textcolor{\blue}{To complete the proof of Theorem \ref{thsmoothgen},
 we shall break $\partial A$ into finitely many pieces, with each piece
 contained in a single chart. We would like to write the
 probability that all of $\partial A$ is covered as
 the product of probabilities for each piece, but to
 achieve the independence needed for this, we need to remove a
 region near the boundary of each piece. By separate
  estimates we can show the removed regions are covered with
  high probability, and this is the content of the next lemma.}

Recall from (\ref{eqgamma}) that  $\gamma_0 \in (1/(2d),\gamma)$.
 Define the sets
 $
 \DG_t := \partial \Gamma \oplus B(o, t^{-\gamma_0})
 $
 and
 $
 \DG_t^{+} := \partial \Gamma \oplus B(o, 2 t^{-\gamma_0}).
 $

\begin{lemm}
	\label{dBlem}
	It is the case that 
	$\lim_{t \to \infty}
	F_t( \DG_t^{+} \cap A, \Po_ t) =1$. 
\end{lemm}
\begin{proof}
	Let $\eps \in (0,2 \gamma_0 - 1/d)$.
	Since we assume \textcolor{\blue}{ $\kappa(\partial \Gamma,r)= O(r^{2-d})$
	as $r \downarrow 0$,}
	for each $t $
	we can take $x_{t,1},\ldots,x_{t,k(t)} \in \R^d$ with
	 $\partial \Gamma \subset \cup_{i=1}^{k(t)}
	 B(x_{t,i},t^{-\gamma_0})$, and
	with $k(t) = O(t^{(d-2)\gamma_0})$.
	Then $\DG_t^{+} \subset \cup_{i=1}^{k(t)} B(x_{t,i},3t^{-\gamma_0})$.
	For each $i \in \{1,\ldots, k(t)\}$, we can cover
	 the ball $B(x_{t,i},3 t^{-\gamma_0})$ 
	with $O((t^{-\gamma_0}/r_t)^d)$ smaller balls of radius $r_t$.
	Then we end up with  balls of
	radius $r_t$, denoted $B_{t,1},\ldots,B_{t,m(t)}$ say,
	such that $\DG_t^{+} \subset \cup_{i=1}^{m(t)} B_{t,i}$ and
	 $m(t) = O(r_t^{-d} t^{-2 \gamma_0})$.
	 By (\ref{0816c}) from  Lemma
	 \ref{lemcov3}, and the union bound, 
	 \bean
	 \Pr[ \cup_{i=1}^{m(t)} ( F_t(B_{t,i} \cap A,\Po_t )^c)] 
	 = O( r_t^{-d} t^{- 2 \gamma_0} t^{\eps - (d-1)/d})
	 = O(  t^{ (1/d) + \eps - 2 \gamma_0} ),
	 \eean
	 which tends to zero.
	 \qed
\end{proof}

Given $t >0$, define the sets
$
\Gamma^{(t^{-\gamma_0})}:= \Gamma \setminus \DG_t
$ and
$$
\Gamma^{(t^{-\gamma_0})}_{r_t} := (\Gamma^{(t^{-\gamma_0})} \oplus B(o,  r_t))
\cap A;
~~~~~
\Gamma_{r_t} := (\Gamma \oplus B(o,r_t)) \cap A,
$$
and define the event $F_t^\Gamma:= F_t(\Gamma^{(t^{- \gamma_0})}_{r_t}  , \Po_t )$.

Note that the definition of $F_t^\Gamma$ does not depend on the 
choice of chart. \textcolor{\blue}{This is needed for the last stage
of the proof of Theorem \ref{thsmoothgen}.
Lemma \ref{lemcap} below shows that
$\Pr[F_t^\Gamma]$ is well approximated by $\Pr[F_t(A_t^*, \Po'_t)]$ and
we have already determined the limiting behaviour of the latter.
We prepare for the proof of Lemma \ref{lemcap} with two
geometrical lemmas.}

\begin{lemm}
	\label{lemcapA}
	For all large enough $t$,
	it is the case that $\Gamma_{r_t}^{(t^{-\gamma_0})}  \subset 
	A_t^*$.
\end{lemm}
\begin{proof}
	Let $x \in \Gamma^{(t^{-\gamma_0})}_{r_t} $, and
	take $y \in \Gamma^{(t^{-\gamma_0})}$ with
	$\|x-y\| \leq r_t$.  
	Writing $y = (u,\phi(u))$ with
	$u \in U_t$,  we claim that
	$\dist(u, \partial U) \geq (1/2)t^{-\gamma_0}$. Indeed, if 
 we had	$\dist(u, \partial U) < (1/2)t^{-\gamma_0}$, then we could
	take $w \in \partial U$ with
 	$\|u -w \|  < (1/2)t^{-\gamma_0}$.
	Then $(w,\phi(w)) \in \partial \Gamma$ and 
	by (\ref{philip}), $|\phi(w) - \phi(u)| \leq (1/4) t^{-\gamma_0}$, so 
	$$
	\|(u,\phi(u)) - (w,\phi(w)) \| \leq \|u-w\| + |\phi(u)-\phi(w)|
	\leq (3/4)t^{-\gamma_0},
	$$
	contradicting the assumption that 
	 $y \in \Gamma^{(t^{-\gamma_0})}$, so the claim is justified.

	Writing $x = (v,s)$ with $v \in \R^{d-1}$,
	and $s \in \R$, we have $\|v-u\| \leq \|x-y\| \leq  r_t$, so 
	$\dist(v, \partial U) \geq (1/2)t^{-\gamma_0} - r_t$, and hence
	$v \in U_t^-$, provided $t$ is big enough
	($U_t^-$ was defined shortly after (\ref{eqgamma}).)   Also 
	$|\phi(v) - \phi(u)| \leq  r_t/4$ by (\ref{philip}),
	so  $|\phi_t(v) - \phi(u)| \leq r_t/2$, provided $t$ is
	big enough, by (\ref{0126a}).
	Also
	$|s - \phi(u)| \leq \|x-y\| \leq  r_t$, so 
	$
	|s - \phi_t(v)| \leq (3/2) r_t.
	$
	Therefore $x \in A^*_t$ by (\ref{Pdtdef}). 
	\qed
\end{proof}
\begin{lemm}
	\label{lemBA}
	For all large enough $t$, we have
	(a) $[A_t^* \oplus B(o,4 r_t)] \cap A \subset A_t$, 
	and (b) 
	 $[ A_t^* \oplus B(o,4r_t)]  
	  \cap \partial A \subset  \Gamma$,
	 and (c)
	 $[ \Gamma_{r_t}^{(t^{-\gamma_0})} \oplus B(o,4r_t)]  
	  \cap \partial A \subset  \Gamma$.
\end{lemm}
\begin{proof}
	Let
	 $x \in A_t^*$.
	 Write  $x = (u,z)$ with $u \in U_t^-$ and $\phi_t(u) - 3r_t/2 \leq z
	 \leq \phi(u)$.
	 
	 Let $y \in B(x,4 r_t) \cap A$,
	 and write $y = (v,s)$ with $v \in \R^{d-1}$
	 and $s \in \R$. Then $\|v-u\| \leq 4r_t$ so provided $t$ is
	 big enough, $v \in U_t$. Also $|s-z| \leq 4 r_t$, and
	 $|\phi(v) - \phi(u)| \leq r_t$
	 by (\ref{philip}), so 
	 $$
	 |s - \phi(v)| \leq |s - z| + |z - \phi(u)| + |\phi(u)- \phi(v)|
	 \leq 4r_t + 2r_t + r_t,
	 $$
	 and since $y \in A$, by (\ref{0901b}) and (\ref{0901a}) 
	 we must have $ 0 \leq s \leq \phi(v) $, provided $t$ is big enough.
	 Therefore $y = (v,s) \in A_t$, which gives us (a).

	 If also $y \in \partial A$, then $\phi(v)=s$, so 
	 $y \in \Gamma$. Hence 
	 we have part (b). 
	 Then by Lemma \ref{lemcapA} we also have part (c).
	 \qed
\end{proof}
\begin{lemm}
	\label{lemcap}
	It is the case that 
	$\Pr[F_t^\Gamma \triangle F_t(A_t^*, \Po'_t) ] \to 0$
	as $t \to \infty$.
\end{lemm}
\begin{proof}
	Since
	$\Gamma_{r_t}^{(t^{- \gamma_0})}  \subset A_t^*$
	by Lemma \ref{lemcapA},
	and moreover $\Po'_t \subset \Po_t$, it  follows that
	$F_t(A_t^*,\Po'_t) \subset
	F_t(\Gamma^{(t^{-\gamma_0})}_{r_t} , \Po_t) = F_t^\Gamma$. 
	Therefore it suffices to prove that
	\bea
	\Pr[F_t(\Gamma^{(t^{-\gamma_0})}_{r_t} , \Po_t)
	\setminus
	F_t(A_t^*,\Po'_t) ] \to 0.  
	\label{0830a}
	\eea

Let $\eps >0$. Suppose event
	$F_t(\Gamma^{(t^{-\gamma_0})}_{r_t} , \Po_t) \cap 
	F_t( \DG_t^{+} \cap A, \Po_t) \setminus F_t(A_t^*,\Po'_t)$
	occurs. Choose $x \in A_t^* \cap V_t(\Po'_t)$.
	Then 
	by Lemma \ref{lemBA}(a), $B(x,r_t) \cap A \subset A_t $.
	Hence $\Po_t \cap B(x,r_t) \subset \Po'_t$, and therefore
	$x \in V_t(\Po_t)$. 
	
	Since we are assuming $
	F_t(\Gamma^{(t^{-\gamma_0})}_{r_t} , \Po_t) $ occurs,
	we therefore have $\dist(x, \Gamma^{(t^{-\gamma_0})} ) > r_t $. 
	Since we also assume 
	$F_t(\DG_t^{+} \cap A,\Po_t)$, we also have $\dist(x, \partial \Gamma) \geq 
	2 t^{-\gamma_0}$ and therefore 
	$\dist (x,\DG_t) =
	\dist (x, (\partial \Gamma) \oplus B(o,t^{-\gamma_0})) \geq t^{-\gamma_0}
	> r_t$.
	Hence 
	$$
	\textcolor{\blue}{
	\dist (x,\Gamma) \geq \min(\dist(x, \Gamma^{(t^{-\gamma_0})}),
	\dist (x, (\partial \Gamma) \oplus B(o,t^{-\gamma_0})) 
	> r_t.}
	$$
	Moreover, by Lemma \ref{lemBA}(b), $\dist (x, (\partial A) \setminus \Gamma)
	\geq 4 r_t$. Thus $\dist(x, \partial A) > r_t$. 
	Moreover, $\dist(x,\partial A) \leq \dist (x,\Gamma) \leq 2r_t $
	because $x \in A_t^*$,
	\textcolor{\blue}{and therefore $x \notin A^{[\eps]}$ 
	(provided $t$ is large enough)
	since $\overline{A^{[\eps]}}$ is compact and contained in $A^o$
	(the set $A^{[\eps]}$ was defined in Section \ref{secdefs}.)}
	 Therefore the event
	$F_t(A^{(r_t)} \setminus A^{[\eps]},\Po_t)^c$ occurs.
	Thus, for large enough $t$ we have the event inclusion
	\bea
	F_t(\Gamma^{(t^{-\gamma_0})}_{r_t} , \Po_t) \cap 
	F_t( \DG_t^{+} \cap A, \Po_t) \setminus F_t(A_t^*,\Po'_t)
	\subset
	F_t( A^{(r_t)} \setminus A^{[\eps]} ,\Po_t)^c .
	\label{0901d}
	\eea

By  (\ref{rt3c}), 
\bea
\lim_{t \to \infty} ( \theta_d t f_0 r_t^d - \log (t f_0) - (d+k-2)
\lglg t )
=  \begin{cases}
	2 \zeta & {\rm if} d=2,k=1
	\\
	+\infty  & {\rm otherwise.}
\end{cases}
\label{0901c}
\eea
Hence by Proposition \ref{Hallthm} 
\textcolor{\blue}{(taking $B= A \setminus A^{[\eps]}$,
and using (\ref{Fnequiv}))},
\bea
\lim_{t \to \infty} \Pr[
	F_t( A^{(r_t)} \setminus A^{[\eps]} ,\Po_t) ]
= \begin{cases} 
	\exp(- 
	|A \setminus A^{[\eps]} |e^{-2 \zeta}) 
	& {\rm if~} d=2, k=1
	\\
	1 & {\rm otherwise.}
\end{cases}
\label{1230a}
	\eea
	Therefore since $\eps$ can be arbitrarily small
	and $|A \setminus A^{[\eps]} | \to 0$
	as $\eps \downarrow 0$,
the event displayed on the left hand side of
 (\ref{0901d}) has probability
tending to zero. Then using
	Lemma \ref{dBlem}, 
	  we have   (\ref{0830a}), which
	completes the proof.
	\qed
\end{proof}

\begin{coro}
	\label{coroFB}
	It is the case that 
	$
	\lim_{t \to \infty} 	\Pr[F_t^\Gamma]
	=
	\exp (-  c_{d,k} |\Gamma| e^{-  \zeta} ). 
$
\end{coro}
\begin{proof}
	By Lemmas \ref{lemdiff4} and \ref{lemdiff3},
	$\Pr[ F_t(A_t^*, \Po'_t) \triangle F_t(A^{**}_t, \tPo_t) ] \to 0$.
	Then by Lemma \ref{lemcap}, 
	$\Pr[ F_t^\Gamma \triangle F_t(A^{**}_t, \tPo_t) ] \to 0$,
	and now the result follows by Proposition \ref{lemsurf}.
	\qed
\end{proof}

\begin{proof}[Proof of Theorem \ref{thsmoothgen}]
	Let $x_1,\ldots,x_J $ and
	\textcolor{\blue}{$r(x_1),\ldots,r(x_J)$}
	be as described at (\ref{bycompactness}).
	Set $\Gamma_1 := B(x_1,r(x_1) ) \cap \partial A$, and for
	$j =2,\ldots, J$, let 
	 $$ 
	 \textcolor{\blue}{ \Gamma_j := \overline{
		 B(x_j,r(x_j) ) \cap \partial A \setminus \cup_{i=1}^{j-1}
	 B(x_i,r(x_i))}. }
 $$

	Then  $\Gamma_1,\ldots,\Gamma_J$
	comprise a finite collection of closed
	sets in $\partial A$
	with disjoint interiors, each of which
	\textcolor{\blue}{satisfies $\kappa(\Gamma_i,r) =
	O(r^{2-d})$ as $r \downarrow 0$,}
	and is
	contained in a single chart $B(x_j,r(x_j))$, and
	with union $\partial A$.
	For $1 \leq i \leq J$,
	define $F_t^{\Gamma_i}$ similarly to $F_t^\Gamma$, that is,
	$F_t^{\Gamma_i} := F_t(\Gamma_{i,r_t}^{(t^{-\gamma_0})} , \Po_t)$ with
	$$
	\Gamma_{i,r_t}^{(t^{-\gamma_0})} := 
	\left( \left[ \Gamma_i \setminus ((\partial \Gamma_i) \oplus B(o,t^{-\gamma_0}))
	\right] \oplus B(o,  r_t) \right) \cap A,
	$$
	and $\partial \Gamma_i := \Gamma_i \cap \overline{\partial A \setminus \Gamma_i}$.
	First we claim that
	the following event inclusion holds:
	\bean
	\cap_{i=1}^J F_t^{\Gamma_i}
	\cap
	F_t(A^{(r_t)}, \Po_t) 
	\setminus
	F_t(A,\Po_t) 
	\subset
	\left(\cap_{i=1}^J F_t([(\partial \Gamma_i) \oplus B(o,2t^{-\gamma_0})]
	\cap A, \Po_t) \right)^c.
	\eean
	Indeed, suppose 
	$
	\cap_{i=1}^J F_t^{\Gamma_i} 
	\cap
	F_t(A^{(r_t)}, \Po_t)
	\setminus
	F_t(A,\Po_t)
	$ occurs, and choose 
	$x \in A \cap V_t(\Po_t)$.
	Then $\dist(x,\partial A) \leq  r_t$ since
	we assume $F_t(A^{(r_t)},\Po_t)$ occurs.
	Then for some $i \in \{1,\ldots,J\}$
	and some $y \in \Gamma_i$ we have $\|x-y \| \leq  r_t$.
	Since we assume $F_t^{\Gamma_i}$ occurs,
	we have $x \notin \Gamma_{i,r_t}^{(t^{-\gamma_0})}$, and hence  
	$\dist(y, \partial \Gamma_i) \leq t^{- \gamma_0}$,
	so 
	  $\dist(x, \partial \Gamma_i) < 2t^{- \gamma_0}$.
	  Therefore $F_t([(\partial \Gamma_i) \oplus B(o, 2t^{-\gamma_0} )]
	  \cap A, \Po_t)$ fails to occur, justifying the claim.

	By  the preceding claim and and the union bound,
\bea
	\Pr[F_t(A,\Po_t)
	] \leq \Pr[ \cap_{i=1}^J F_t^{\Gamma_i}
		\cap F_t(A^{(r_t)},\Po_t) ] 
\nonumber	\\
	\leq \Pr[F_t(A,\Po_t) ]  
	+ \sum_{i=1}^J
	\Pr[ F_t([(\partial \Gamma_i) \oplus B(o,2t^{-\gamma_0})] \cap A, \Po_t)^c].
	\label{0916a}
\eea
	By  Lemma \ref{dBlem},
	$\Pr[ F_t([(\partial \Gamma_i) \oplus B(o,2t^{-\gamma_0})] \cap A, \Po_t)]
	\to 1 $ for each $i$.
Therefore by (\ref{Rdashdef}) and (\ref{0916a}), 
\bea
\lim_{t \to \infty} \Pr[R'_{t,k} \leq r_t] =
\lim_{t \to \infty} \Pr[F_{t} (A,\Po_t) ] = 
	\lim_{t \to \infty} \Pr[
	\cap_{i=1}^J F_t^{\Gamma_i}
		\cap F_t(A^{(r_t)}, \Po_t)
	],
	\label{0901e}
\eea
provided the last limit exists.
By Corollary \ref{coroFB},
we have for each $i$ that
 \bea
\lim_{t \to \infty}
	(\Pr[ F^{\Gamma_i}_t]) 
= \exp(- c_{d,k}|\Gamma_i|  e^{-\zeta} ).
\label{0901f}
\eea
	Also, we claim that for large enough $t$ 
	the events $F_t^{\Gamma_1}$, \ldots, $F_t^{\Gamma_J}$ are
	mutually independent.
	Indeed, given distinct $i,j \in \{1,\ldots,J\}$,
	if $x \in \Gamma_{i,r_t}^{(t^{-\gamma_0})}$ and
	 $y \in \Gamma_{j,r_t}^{(t^{-\gamma_0})}$, then
	we can take
	$y' \in \Gamma_j \setminus ( \partial \Gamma_j  \oplus B(o,t^{-\gamma_0}))$
	with $\|y'-y \| \leq r_t$. 
	If $\| x -y \| \leq 2r_t$ then by the triangle
	inequality $\|x - y'\| \leq 3r_t$, but since $y' \notin \Gamma_i$,
	this would contradict Lemma \ref{lemBA}(c). Therefore
	$\| x - y\| > 2r_t$, and hence the $r_t$-neighbourhoods
	of $
	 \Gamma_{i,r_t}^{(t^{-\gamma_0})}$ and of
	 $ \Gamma_{j,r_t}^{(t^{-\gamma_0})}$ are disjoint. This gives
	 us the independence claimed.

Now observe that $F_t(A^{(r_t)},\Po_t) \subset F_t(A^{(4r_t)},\Po_t)$,
and we claim that 
\bea
\Pr[ F_t(A^{(4r_t)},\Po_t) \setminus F_t(A^{(r_t)},\Po_t)] \to 0 
~~~
{\rm as} ~
t \to \infty.
\label{0920a}
\eea
Indeed, given $\eps >0$, for large $t $ 
the probability on the left side of (\ref{0920a})
is bounded by
$\Pr[ F_t(A^{(r_t)}\setminus A^{[\eps]},\Po_t)^c]  $, and by
(\ref{1230a}) the latter
probability tends to a limit which can be made arbitrarily
small by the choice of $\eps$.
Hence by Proposition \ref{Hallthm} 
(using (\ref{Fnequiv})) and (\ref{0901c}),
\bea
\lim_{t \to \infty} \Pr[ F_t(A^{(4r_t)},\Po_t) ]
=
\lim_{t \to \infty} \Pr[ F_t(A^{(r_t)},\Po_t)] 
\nonumber \\
=  \begin{cases}
	\exp( -  |A| e^{-2 \zeta} ) & {\rm if} ~ d=2,k=1
	\\
	1 & {\rm otherwise}.
\end{cases}
\label{0920c}
	\eea
Moreover, by (\ref{0901e}) and (\ref{0920a}),
\bea
\lim_{t \to \infty} \Pr[R'_{t,k} \leq r_t] =
	\lim_{t \to \infty} \Pr[
	\cap_{i=1}^J F_t^{\Gamma_i}
		\cap F_t(A^{(4r_t)}, \Po_t)
	],
	\label{0920b}
\eea
provided the last limit exists.
However, the events in the right hand side of (\ref{0920b})
are mutually independent, so using (\ref{0901f}), 
(\ref{0920c}) and (\ref{rt3c}), we obtain
	the  second equality of (\ref{eqmain3}).
We then obtain the rest of (\ref{eqmain3})  
using Lemma \ref{depolem}.
\qed
\end{proof}

{\bf Acknowledgements.}
I thank Alastair King and Andrew du Plessis
for some helpful conversations during the 
preparation of this paper. I thank  an anonymous referee,
	and also Xiaochuan Yang, for reading through earlier 
	versions and making some
helpful observations.

Data sharing not applicable to this article as no datasets were generated or analysed during the current study.

%
%



\end{document}